\documentclass[a4paper, 11pt, reqno]{amsart}

\usepackage{microtype}
\usepackage{fullpage}

\usepackage{amssymb}
\usepackage{mathtools}
\usepackage{enumitem}

\newcommand{\BB}{\mathbb{B}}
\newcommand{\CC}{\mathbb{C}}
\newcommand{\dd}{\,\mathrm{d}}
\newcommand{\dom}{\mathfrak{D}}
\newcommand{\ee}{\mathrm{e}}
\newcommand{\ii}{\mathrm{i}}
\newcommand{\Laplace}{{\operatorname{\Delta}}}
\newcommand{\NN}{\mathbb{N}}
\newcommand{\pd}{\partial}
\newcommand{\ran}{\operatorname{ran}}
\renewcommand{\restriction}{\mathord{\upharpoonright}}
\newcommand{\RR}{\mathbb{R}}
\renewcommand{\Re}{\operatorname{Re}}

\theoremstyle{definition}
\newtheorem{definition}{Definition}[section]

\theoremstyle{remark}
\newtheorem{remark}{Remark}[section]

\theoremstyle{plain}
\newtheorem{theorem}{Theorem}[section]

\theoremstyle{plain}
\newtheorem{proposition}{Proposition}[section]

\theoremstyle{plain}
\newtheorem{lemma}{Lemma}[section]

\theoremstyle{remark}
\newtheorem*{acknowledgment}{Acknowledgment}

\numberwithin{equation}{section}

\usepackage{hyperref}
    \hypersetup{
        bookmarksdepth = 3,
               linktoc = page,
            colorlinks = true,
             linkcolor = blue,
             citecolor = green,
    }

\usepackage{cleveref}
    \crefname{equation}{eq.}{eqs.}
    \Crefname{equation}{Eq.}{Eqs.}

\title{Stable blowup for focusing semilinear wave equations in all dimensions}

\author{Matthias Ostermann}
\address{University of Vienna, Faculty of Mathematics, Oskar-Morgenstern-Platz 1, 1090 Vienna, Austria}
\email{matthias.ostermann@univie.ac.at}
\thanks{This work has been supported by the Vienna School of Mathematics (VSM)}

\begin{document}
\begin{abstract}
We consider the wave equation with focusing power nonlinearity. The associated ODE in time gives rise to a self-similar solution known as the ODE blowup. We prove the nonlinear asymptotic stability of this blowup mechanism outside of radial symmetry in all space dimensions and for all superlinear powers. This result covers for the first time the whole energy-supercritical range without symmetry restrictions.
\end{abstract}
\maketitle
\tableofcontents
\section{Introduction}
This paper is concerned with the Cauchy problem for the focusing semilinear wave equation
\begin{equation}
\label{NLWIntro}
( - \pd_{t}^{2} + \Laplace_{x} ) \psi(t,x)  + \psi(t,x)|\psi(t,x)|^{p-1} = 0
\end{equation}
for $\psi:\RR^{1,d}\rightarrow\RR$, where $d\in\NN$ and $p\in\RR_{>1}$. This equation is invariant under the scaling transformation
\begin{equation}
\label{NLWScaling}
\psi \mapsto \psi^{\lambda}
\qquad\text{with}\qquad
\psi^{\lambda}(t,x) = \lambda^{-s_{p}} \psi\big(\tfrac{t}{\lambda},\tfrac{x}{\lambda}\big) \,, \quad
s_{p} = \frac{2}{p-1} \,,
\end{equation}
for any $\lambda>0$. The scaling \eqref{NLWScaling} determines the critical regularity $s_{\mathrm{c}} \coloneqq \frac{d}{2} - s_{p}$ for the scaling invariant homogeneous Sobolev space $\dot{H}^{s_{\mathrm{c}}}(\RR^{d}) \times \dot{H}^{s_{\mathrm{c}}-1}(\RR^{d})$. Moreover, \Cref{NLWIntro} also admits a conserved but indefinite energy
\begin{equation*}
E[\psi](t) =
\frac{1}{2}\big\| \big( \psi(t,\,.\,),\pd_{t}\psi(t,\,.\,) \big) \big\|_{\dot{H}^{1}(\RR^{d})\times L^{2}(\RR^{d})}^{2} - \frac{1}{p+1} \| \psi(t,\,.\,) \|_{L^{p+1}(\RR^{d})}^{p+1} \,,
\end{equation*}
which obeys the scaling law
\begin{equation*}
E[\psi^{\lambda}](t) = \lambda^{2\big(\frac{d}{2}-s_{p}-1\big)} E[\psi]\big(\tfrac{t}{\lambda}\big) \,.
\end{equation*}
According to the sign of the exponent $\frac{d}{2}-s_{p}-1$, \Cref{NLWIntro} is called
\begin{align*}
&\bullet\text{energy-subcritical} && \text{if } d=1,2 \text{ and } p>1 \quad\text{or}\quad d\geq 3 \text{ and } 1<p<1+\frac{4}{d-2}, & \\
&\bullet\text{energy-critical} && \text{if } d\geq 3 \text{ and } p=1+\frac{4}{d-2}, & \\
&\bullet\text{energy-supercritical} && \text{if } d\geq 3 \text{ and } p>1+\frac{4}{d-2}. &
\end{align*}
Classification into criticality classes provides a general and useful theme for capturing the exceedingly rich dynamical phenomena of scaling invariant nonlinear equations. For the basic well-posedness theory for semilinear wave equations we refer to \cite{MR1335386}, \cite{MR2233925}. In this context, it is noteworthy that the focusing character of the nonlinearity in \Cref{NLWIntro} may cause solutions to lose their regularity in finite time. This is demonstrated in the early work \cite{MR344697} by H.~Levine, where it is shown that initial data with negative energy lead to blowup in finite time. An important role in the description of singularity formation is played by solutions that are invariant under the scaling \eqref{NLWScaling}. Such solutions are called \emph{self-similar} and together with time translation symmetry it follows that they are of the form
\begin{equation*}
\psi(t,x) = (T-t)^{-s_{p}} \varphi\left( \tfrac{x}{T-t} \right)
\end{equation*}
for $\varphi:\RR^{d}\rightarrow\RR$ and a free constant $T>0$. Existence of a countable family of smooth radially symmetric self-similar solutions to \Cref{NLWIntro} has been proved by P.~Bizo\'{n} et al. \cite{MR2578477}, \cite{MR2351023} in three dimensions for odd powers. For greater than three space dimensions, this has been further investigated numerically in \cite{MR2804023}. A corresponding existence result of countably many self-similar blowup solutions in $d\geq 3$ and for powers $1+\frac{4}{d-2} < p \leq 1 + \frac{4}{d-3}$ has been showed by W.~Dai and T.~Duyckaerts \cite{MR4350584}. The ground state of these families exists throughout all space dimensions and for all $p>1$ and is known as the \emph{ODE blowup}
\begin{equation}
\label{ODEBlowup}
\psi_{T}^{*}(t,x) = c_{p}(T-t)^{-s_{p}} \,, \qquad c_{p} = (s_{p}(s_{p}+1))^{\frac{1}{p-1}} \,.
\end{equation}
This is an explicit example of a solution that evolves from smooth data and loses its regularity in finite time. By means of finite speed of propagation, one can even prepare compactly supported smooth initial data that develop exactly the same singularity as \eqref{ODEBlowup} in a past light cone. This raises the question of how universal the occurrence of such drastic breakdown of solutions is. In fact, numerical studies \cite{MR2097671} suggest that generic blowup in the large data evolution of \Cref{NLWIntro} is described by the ODE blowup profile. Towards demonstrating the universality of this blowup mechanism, the subject of this paper is to establish a robust and systematic stability theory for it.
\par\medskip
In one space dimension, F.~Merle and H.~Zaag \cite{MR2362418} proved indeed that ODE blowup describes the universal profile for any solution that blows up. In higher space dimensions and for powers $1 < p \leq 1+\frac{4}{d-1}$, their works \cite{MR2004432}, \cite{MR2115461} show that the rate of any blowup solution in the energy space is the same as the one exhibited by \eqref{ODEBlowup}. Later, they also gave a stability result for the ODE blowup family in higher space dimensions \cite{MR3302641}, \cite{MR3413856}. However, the underlying theory is based on the existence of a Lyapunov functional which restricts their results to subconformal powers $1<p<1+\frac{4}{d-1}$. A different approach was put forth by R.~Donninger \cite{MR2753616}. Together with B.~Sch\"{o}rkhuber, they studied in \cite{MR2909934} the stability of the blowup solution \eqref{ODEBlowup} for the radial three-dimensional energy-subcritical semilinear wave equation. Subsequently, they treated in \cite{MR3152726} the nonlinear stability problem in three dimensions for powers $p>3$ and were even able to remove the symmetry assumptions in \cite{MR3537340}. This programme could be implemented in \cite{MR3926085} also for the cubic wave equation outside of radial symmetry for $d\in\{5,7,9,11,13\}$. Within radial symmetry, stability of the ODE blowup has been established in \cite{MR3742520} for odd space dimensions and superconformal powers $p>1+\frac{4}{d-1}$. However, the underlying techniques in this work rely heavily on the radial structure. The first result for ODE blowup stability for the energy-critical radial wave equation at optimal regularity is due to R.~Donninger \cite{MR3662440} in three space dimensions. This has been achieved with Strichartz estimates in similarity coordinates for wave equations with self-similar potentials. There are generalizations to the critical radial wave equation in dimensions $d\in\{3,4,5,6\}$ in the works \cite{MR4101905} and \cite{MR4586813}. Lately, as an application of the functional setting in \cite{MR4713110}, E.~Csobo, I.~Glogi\'{c} and B.~Sch\"{o}rkhuber established non-radial stability of the ODE blowup for the quadratic wave equation in $d=7,9$.
\par\medskip
Still, this leaves open the stability of the ODE blowup, for instance, for almost the whole energy-supercritical range $p>1+\frac{4}{d-2}$ outside of radial symmetry. We will close the gaps in \Cref{MainTHM} below.
\subsection{Statement of the stability theorem}
By action of symmetries, one obtains from \eqref{ODEBlowup} a larger family of blowup solutions. Namely, the flow of \Cref{NLWIntro} is also preserved under the Poincar\'{e} symmetry of the underlying Minkowski spacetime $\RR^{1,d}$. This includes, in particular, time translations and Lorentz boosts $\Lambda(\beta)\in SO(1,d)$ whose components are given by
\begin{align*}
&&
\Lambda(\beta)^{0}{}_{0} &= \gamma(\beta) \,, &
\Lambda(\beta)^{0}{}_{i} &= -\gamma(\beta) \beta_{i} \,, && \\
&&
\Lambda(\beta)^{i}{}_{0} &= -\gamma(\beta) \beta^{i} \,, &
\Lambda(\beta)^{i}{}_{j} &= \delta^{i}_{j} + \frac{\gamma(\beta)^{2}}{1+\gamma(\beta)}\beta^{i}\beta_{j} \,, &&
\end{align*}
with Lorentz parameter $\beta\in\BB^{d}_{1}$ and Lorentz factor
\begin{equation*}
\gamma(\beta) = (1-\beta^{\top}\beta)^{-\frac{1}{2}} \,.
\end{equation*}
Precomposing the spatially homogeneous ODE blowup profile \eqref{ODEBlowup} with a Lorentz boost and a time translation, one obtains an explicit $(d+1)$-parameter family of blowup solutions to \Cref{NLWIntro} that is given by
\begin{equation}
\label{ODEBlowupFamily}
\psi_{\beta,T}^{*}(t,x) \coloneqq \psi_{0}^{*}\big(\Lambda(\beta)(t-T,x)\big) = c_{p}\gamma(\beta)^{-s_{p}}(T-t+\beta^{\top}x)^{-s_{p}} \,.
\end{equation}
Observe that this introduces solutions with spatial singularities. Nevertheless, by finite speed of propagation one can restrict the evolution to past light cones where these solutions remain smooth for small enough Lorentz parameters. This shows that a stability theory for the ODE blowup necessarily has to be formulated within light cones. With this, we come to the main result of this paper, where we show in \emph{any} space dimension that ODE blowup is nonlinearly stable under a large class of perturbations.
\begin{theorem}
\label{MainTHM}
Let $(d,p,k,R)\in\NN\times\RR_{>1}\times\NN\times\RR_{\geq 1}$ with $k>\frac{d}{2}$ and put
\begin{equation*}
\omega_{p} = \min\{1,s_{p}\}
\qquad\text{where}\qquad
s_{p} = \frac{2}{p-1} \,.
\end{equation*}
For all $ 0 < \varepsilon < \omega_{p}$ there are constants $0<\delta_{d,p,k,R,\varepsilon}<1$ and $C_{d,p,k,R,\varepsilon} > 1$ such that for all $0<\delta\leq\delta_{d,p,k,R,\varepsilon}$, $C\geq C_{d,p,k,R,\varepsilon}$ and all real-valued $(f,g)\in C^{\infty}(\RR^{d})\times C^{\infty}(\RR^{d})$ with
\begin{equation*}
\| (f,g) \|_{H^{k}(\RR^{d})\times H^{k-1}(\RR^{d})} \leq \frac{\delta}{C}
\end{equation*}
there exist parameters $\beta^{*}\in\overline{\BB^{d}_{\delta}}$ and $T^{*}\in \overline{\BB^{1}_{\delta}(1)}$ and a unique solution $\psi\in C^{\infty}\big(\Omega^{1,d}_{R}(T^{*})\big)$ in the extended past light cone
\begin{equation*}
\Omega^{1,d}_{R}(T^{*}) = \left\{ (t,x) \in \RR^{1,d} \mid 0\leq t < T^{*}, \, |x| \leq R(T^{*} - t) \right\}
\end{equation*}
to the Cauchy problem
\begin{equation*}
\renewcommand{\arraystretch}{1.2}
\left\{
\begin{array}{rcll}
(-\pd_{t}^{2}+\Laplace_{x})\psi(t,x) &=& \psi(t,x)|\psi(t,x)|^{p-1} \,, & \quad (t,x)\in \Omega^{1,d}_{R}(T^{*}) \,, \\
\psi(0,x) &=& \psi_{0,1}^{*}(0,x) + f(x) \,, & \quad x\in\RR^{d} \,, \\
(\pd_{0}\psi)(0,x) &=& (\pd_{0}\psi_{0,1}^{*})(0,x) + g(x) \,, & \quad x\in\RR^{d} \,,
\end{array}
\right.
\end{equation*}
such that the bounds
\begin{align*}
(T^{*}-t)^{-\frac{d}{2}+s_{p}+s}
\left\| \psi(t,\,.\,) - \psi_{\beta^{*},T^{*}}^{*}(t,\,.\,) \right\|_{\dot{H}^{s}(\BB^{d}_{R(T^{*}-t)})}
&\lesssim
(T^{*}-t)^{\omega_{p} - \varepsilon}
\intertext{for $s=0,1,\ldots,k$, and}
(T^{*}-t)^{-\frac{d}{2}+s_{p}+s}
\left\| \pd_{t}\psi(t,\,.\,) - \pd_{t}\psi_{\beta^{*},T^{*}}^{*}(t,\,.\,) \right\|_{\dot{H}^{s-1}(\BB^{d}_{R(T^{*}-t)})}
&\lesssim
(T^{*}-t)^{\omega_{p} - \varepsilon}
\intertext{for $s=1,\ldots,k$, hold for all $0\leq t < T^{*}$.}
\end{align*}
\end{theorem}
\begin{remark}
\label{remark1}
The solution converges to the ODE blowup in the following sense. Namely, if we consider $\beta^{*}\neq 0$ then
\begin{align*}
\left\| \psi_{\beta^{*},T^{*}}^{*}(t,\,.\,) \right\|_{\dot{H}^{s}(\BB^{d}_{R(T^{*}-t)})} &\simeq (T^{*} -t)^{\frac{d}{2}-s_{p}-s} \,, \\
\left\| \pd_{t}\psi_{\beta^{*},T^{*}}^{*}(t,\,.\,) \right\|_{\dot{H}^{s-1}(\BB^{d}_{R(T^{*}-t)})} &\simeq (T^{*} -t)^{\frac{d}{2}-s_{p}-s} \,,
\end{align*}
for all $0\leq t < T^{*}$, so the bounds in \Cref{MainTHM} imply
\begin{align*}
\frac{
\left\| \psi(t,\,.\,) - \psi_{\beta^{*},T^{*}}^{*}(t,\,.\,) \right\|_{\dot{H}^{s}(\BB^{d}_{R(T^{*}-t)})}
}{
\left\| \psi_{\beta^{*},T^{*}}^{*}(t,\,.\,) \right\|_{\dot{H}^{s}(\BB^{d}_{R(T^{*}-t)})}
} &\lesssim
(T^{*}-t)^{\omega_{p} - \varepsilon} \,, \\
\frac{
\left\| \pd_{t}\psi(t,\,.\,) - \pd_{t}\psi_{\beta^{*},T^{*}}^{*}(t,\,.\,) \right\|_{\dot{H}^{s-1}(\BB^{d}_{R(T^{*}-t)})}
}{
\left\| \pd_{t}\psi_{\beta^{*},T^{*}}^{*}(t,\,.\,) \right\|_{\dot{H}^{s-1}(\BB^{d}_{R(T^{*}-t)})}
} &\lesssim
(T^{*}-t)^{\omega_{p} - \varepsilon} \,,
\end{align*}
for all $0\leq t < T^{*}$. That is, convergence of $\psi(t,\,.\,)$ to $\psi_{\beta^{*},T^{*}}^{*}(t,\,.\,)$ as $t\nearrow T^{*}$ takes place along shrinking time slices $\{t\}\times\BB^{d}_{R(T^{*}-t)}$ and is quantified in Sobolev norms relative to the blowup behaviour of $\psi_{\beta^{*},T^{*}}^{*}(t,\,.\,)$.
\end{remark}
\begin{remark}
Even for non-radial perturbations far away from the center, the evolving solution remains stable in the sense of \Cref{remark1} within the extended past light cone $\Omega^{1,d}_{R}(T^{*})$. This improves the above mentioned stability results within light cones. Also, smoothness of the initial data persists for the corresponding solution so that the evolution can be interpreted classically.
\end{remark}
\begin{remark}
\Cref{MainTHM} is especially of interest in the much less explored energy-supercritical case, where it describes the evolution for large data near the ODE blowup without any symmetry restrictions. Since we do not distinguish for the stability problem any Sobolev space above critical regularity $s_{\mathrm{c}} = \frac{d}{2} - s_{p}$, we state and prove our stability result for \emph{any} integer regularities $k$ above $\frac{d}{2}$. In fact, if $p>5$ then $k>\frac{d}{2}$ is optimal among Sobolev spaces of integer order above critical regularity.
\end{remark}
\subsection{Overview of related research}
\label{SecRelRes}
A lot of progress has been achieved in the study of nonlinear wave equations, which is preceded by many impressive works and methods. Here, we only touch upon recent literature that is related to singularity formation in the focusing semilinear wave equation.
\par\medskip
For energy-critical focusing semilinear wave equations, a non-trivial smooth, radial, static solution is known explicitly, called the \emph{ground state}. In the work \cite{MR2461508}, C.~Kenig and F.~Merle showed that this solution is at the energy threshold between finite-time blowup and scattering, also see \cite{MR2470571} for the threshold case. This has led to an intensive study of \emph{type II} blowup, i.e., blowup solutions with bounded energy norm. J.~Krieger, W.~Schlag and D. Tataru \cite{MR2494455} gave the first construction of radial type II blowup in $d=3$ as a rescaled ground state plus a radiation term. Other solutions of this form were constructed by M.~Hillairet and P.~Rapha\"{e}l \cite{MR3006642} in $d=4$ and J.~Jendrej \cite{MR3579128} in $d=5$. We just remark that by now, many more constructions are available. Concerning the stability of type II blowup, the solutions constructed in \cite{MR2494455}, \cite{MR3205646} for the three dimensional quintic wave equation are stable along a co-dimension one manifold of radial initial data \cite{MR4194893}, \cite{MR3351183}. The global dynamics near the ground state in $d=3,5$ has been studied in \cite{MR3007693}, \cite{MR3086065}, \cite{MR3177940}, \cite{MR3302610} by J.~Krieger, K.~Nakanishi and W.~Schlag. A complete classification of radial type II blowup in $d=3$ has been established in a line of research by T.~Duyckaerts, C.~Kenig and F.~Merle leading to \cite{MR2966655}, \cite{MR3359522}, where they showed that any type II solution decouples asymptotically into a sum of travelling waves and a radiation term. In the non-radial setting for dimensions $d = 4,5$, weaker versions \cite{MR3678502}, \cite{MR3959860} are available. Such dynamical behaviour is posed under a \emph{soliton resolution conjecture}. Recently, several works have contributed to proofs of this conjecture for radial solutions, including the ones by T. Duyckaerts, C.~Kenig and F.~Merle together with C.~Collot, H.~Jia and Y.~Martel \cite{MR4163362}, \cite{MR4289254}, \cite{MR4567713}, \cite{MR4397184}, \cite{2022arXiv220101848C} in odd space dimensions and $d=4,6$, as well as by J.~Jendrej and A.~Lawrie \cite{2022arXiv220309614J} for all $d\geq 4$.
\par\medskip
Moreover, for the energy-supercritical semilinear wave equation, type II blowup is excluded for radial solutions by results of T.~Duyckaerts, C.~Kenig and F.~Merle \cite{MR3158532} for $d=3$ and by B. Dodson and A.~Lawrie \cite{MR3401013} for $d=5$. In particular, what T.~Duyckaerts and T.~Roy  \cite{MR3766119} showed is that radial solutions of \Cref{NLWIntro} in $d=3$ for $p>5$ either have unbounded critical norm or scatter. We also mention that among the class of self-similar solutions with finite energy, O.~Kavian and F.~Weissler \cite{MR1077471} showed that \Cref{NLWIntro} admits no non-trivial, real-valued, radially symmetric solution if $p\geq 1 + \frac{4}{d-2}$. In contrast to type II behaviour, solutions, whose critical norm becomes unbounded as the maximal time of existence is approached, are characterized as \emph{type I} blowup. We have already encountered the ODE blowup mechanism as such an example. C.~Collot \cite{MR3778126} used concentration of a soliton profile to describe another type of blowup mechanism for the radial energy-supercritical wave equation in dimensions $d\geq 11$ for odd powers above the Joseph-Lundgren exponent. Regarding the construction of solutions for the energy-supercritical wave equation, some results can be found in \cite{MR3736488}, \cite{MR3872322}, also see \cite{2022arXiv221113699K} for a countable family of finite co-dimensional stable self-similar blowup solutions. Recently in \cite{MR4292964} and \cite{MR4713110}, new examples of radial self-similar blowup solutions were discovered in closed form for the cubic and quadratic wave equation in all energy-supercritical space dimensions. For the cubic nonlinearity in $d=7$, I.~Glogi\'{c} and B.~Sch\"{o}rkhuber gave a proof of co-dimension one stability of this solution in a past light cone outside of radial symmetry. Together with E.~Csobo, they established the analogous stability result for the quadratic wave equation in $d=9$. The role of these explicit solutions as a threshold between ODE blowup and dispersion has been investigated numerically in \cite{MR4105355}. Furthermore, in the article \cite{MR4700297} on the radial quadratic wave equation in the lowest energy-supercritical space dimension $d=7$, the conditional blowup stability has been extended beyond the blowup time to a region approaching the Cauchy horizon of the singularity. This is based on a formulation in a novel coordinate system, see \cite{MR4338226} and \cite{MR4661000} for implementations in wave maps and Yang-Mills equations.
\subsection{Outline of the stability problem}
\label{SubSecOutline}
In order to study solutions of the nonlinear equation
\begin{equation*}
( - \pd_{t}^{2} + \Laplace_{x} ) \psi(t,x) + F(\psi(t,x)) = 0
\end{equation*}
that evolve from perturbations $(f,g)$ of the blowup $(\psi_{0,1}^{*}, \pd_{0} \psi_{0,1}^{*})$ at $t=0$, we introduce for parameters $T\in\RR_{>0}$ and $\beta\in\mathbb{B}^{d}_1$ a perturbation variable $u$ through a profile decomposition
\begin{equation*}
\psi = \psi_{\beta,T}^{*} + u \,.
\end{equation*}
This suggests to split the power nonlinearity $F(z) = z|z|^{p-1}$ according to
\begin{equation*}
F\big(\psi_{\beta,T}^{*}(t,x) + u(t,x)\big) = F\big(\psi_{\beta,T}^{*}(t,x)\big) + \mathcal{V}_{\beta,T}(t,x) u(t,x) + \mathcal{N}_{\beta,T}(u)(t,x)
\end{equation*}
into a linearized part with smooth potential
\begin{equation}
\label{PhysPot}
\mathcal{V}_{\beta,T}(t,x) \coloneqq F'\big(\psi_{\beta,T}^{*}(t,x)\big)
\end{equation}
and a nonlinear remainder given by
\begin{equation}
\label{PhysNonlin}
\mathcal{N}_{\beta,T}(u)(t,x) \coloneqq F\big(\psi_{\beta,T}^{*}(t,x)+u(t,x)\big) - F'\big(\psi_{\beta,T}^{*}(t,x)\big)u(t,x) - F\big(\psi_{\beta,T}^{*}(t,x)\big) \,.
\end{equation}
In terms of $u$, this leads to the Cauchy problem 
\begin{equation}
\label{CauchyProblem2}
\renewcommand{\arraystretch}{1.2}
\left\{
\begin{array}{rcl}
0 &=& \Big(-\pd_{t}^{2}+\Laplace_{x}+\mathcal{V}_{\beta,T}(t,x)\Big)u(t,x) + \mathcal{N}_{\beta,T}(u)(t,x) \,, \\
u(0,x) &=& f(x) + \psi_{0,1}^{*}(0,x) - \psi_{\beta,T}^{*}(0,x) \,, \\
(\pd_{0}u)(0,x) &=& g(x) + (\pd_{0}\psi_{0,1}^{*})(0,x) - (\pd_{0}\psi_{\beta,T}^{*})(0,x) \,.
\end{array}
\right.
\end{equation}
As noted above, due to the spatial singularities of $\psi_{\beta,T}^{*}$ and finite speed of propagation, this problem has to be reasonably posed within light cones. Employing classical similarity coordinates, defined via
\begin{equation}
\label{SimCo}
t = T-T\ee^{-\tau} \,, \quad x = T\ee^{-\tau}\xi \,, \qquad (\tau,\xi) \in [0,\infty)\times\BB^{d}_{R} \,,
\end{equation}
allows us to track the evolution of initial data actually in an extended past light cone
\begin{equation*}
\Omega^{1,d}_{R}(T) \coloneqq \left\{ (t,x) \in \RR^{1,d} \mid 0\leq t < T, \, |x| \leq R(T - t) \right\} \,, \qquad T>0 \,, \quad R\geq 1 \,.
\end{equation*}
The underlying analysis is split into three sections and implemented as follows.
\begin{itemize}[itemsep = 1em, topsep = 1em]
\item\textbf{\emph{The free part}} associated to problem \eqref{CauchyProblem2} is determined by the wave operator $-\pd_{t}^{2} + \Laplace_{x}$. We introduce with \Cref{TransitionMechanism} a convenient relation for the transition between the classical formulation of the wave operator and its rescaled incarnation within a first-order formalism in similarity coordinates \eqref{SimCo} as a densely defined linear operator $\mathbf{L}_{d,p,k,R}$ in the Sobolev space $H^{k}(\BB^{d}_{R})\times H^{k-1}(\BB^{d}_{R})$. With this operator, the aim is to generate a semigroup with good growth properties. We will employ the Lumer-Phillips theorem which requires two essential ingredients. First, we give a novel and systematic construction of dissipative inner products on the Sobolev spaces. This delicately exploits structural features of the wave equation outside of radial symmetry. Unlike previous constructions, we can encompass all space dimensions and all non-negative integer regularities. Secondly, we need a density result for the range of $\lambda\mathbf{I}-\mathbf{L}_{d,p,k,R}$. This can be reduced to the construction of a smooth approximate solution to a degenerate elliptic problem, which can be solved by means of a decomposition into spherical harmonics. This was first noted in \cite{MR3218816} in the context of a cubic wave equation in three space dimensions. We finish the section with \Cref{FreeSemigroupTHM}, where strongly continuous semigroups for the free wave flow in similarity coordinates are presented together with exponential growth bounds.\par
\item\emph{\textbf{The linearized part}} is composed of the wave operator with a smooth self-similar potential, $-\pd_{t}^{2}+\Laplace_{x}+\mathcal{V}_{\beta,T}$. This operator fits immediately into the functional analytic setting as a compact perturbation $\mathbf{L}_{\beta} = \overline{\mathbf{L}_{d,p,k,R}} + \mathbf{L}_{\beta}'$ which is the generator of the semigroup $\mathbf{S}_{\beta}(\tau)$ for the linearized wave flow in similarity coordinates. Since $d,p,k,R$ are fixed and do not vary, we omit them in our notation from here on. The Lorentz symmetry and time translation symmetry of \Cref{NLWIntro} induce the unstable eigenvalues $0$ and $1$ of $\mathbf{L}_{\beta}$. To deal with those instabilities in the wave evolution, we perform a spectral analysis for $\mathbf{L}_{\beta}$. In the spectral theorem \ref{SpecThmLbeta} it is proved that $\sigma(\mathbf{L}_{\beta}) \cap \mathbb{H}_{\omega_{0}} = \{0,1\}$ for appropriate right half-planes and the geometric eigenspaces and ranges of the associated spectral projections are computed explicitly. In case $\beta=0$, this follows from analysing solutions of hypergeometric differential equations. Then, we exploit compactness of the perturbation $\mathbf{L}_{\beta}'$ and Lipschitz continuous dependence with respect to the Lorentz parameter $\beta$ to infer the spectral theorem for small parameter values from abstract perturbation theory. Having this spectral information at hand, we conclude this section with \Cref{stableEvo} which gives a complete description of the linearized dynamics near the ODE blowup. Namely, we identify a family of finite-dimensional unstable subspaces on whose complement the semigroup $\mathbf{S}_{\beta}(\tau)$ is \emph{uniformly exponentially stable} with respect to the Lorentz parameter. The onset of the Lorentz parameter $\beta$ is a non-radial effect and the uniformity in the growth bounds with respect to it is crucial for establishing estimates in the ensuing nonlinear analysis. To accomplish this, we have provided in \Cref{AppendixA} a generalization of the Gearhart-Pr\"{u}ss-Greiner Theorem for semigroups that depend on additional parameters.\par
\item\emph{\textbf{The nonlinear part}} treats the full Cauchy problem \eqref{CauchyProblem2} in similarity coordinates. Using Duhamel's formula and the semigroups $\mathbf{S}_{\beta}(\tau)$, this problem turns into a fixed-point problem
\begin{equation*}
\mathbf{u}(\tau) = \mathbf{S}_{\beta}(\tau) \mathbf{U}_{\beta,T}(\mathbf{f}) + \int_{0}^{\tau} \mathbf{S}_{\beta}(\tau-\tau')\mathbf{N}_{\beta}(\mathbf{u}(\tau')) \dd \tau'\,, \qquad \tau \geq 0 \,,
\end{equation*}
with a corresponding initial data operator $\mathbf{U}_{\beta,T}(\mathbf{f})$ and a nonlinearity $\mathbf{N}_{\beta}$ which is locally Lipschitz continuous in $H^{k}(\BB^{d}_{R})\times H^{k-1}(\BB^{d}_{R})$. Now the task is to find parameters $\beta$ and $T$ so that this fixed-point equation has a global \emph{mild} solution $\mathbf{u}\in C\big( [0,\infty), H^{k}(\BB^{d}_{R})\times H^{k-1}(\BB^{d}_{R}) \big)$. In general, the linear instabilities obstruct the existence of such a global solution. Therefore, we stabilize the above equation by projecting it onto the stable subspace which produces in the Duhamel formula a natural correction term $\mathbf{C}_{\beta,T}\big( \mathbf{U}_{\beta,T}(\mathbf{f}),\mathbf{u}\big)$ that is subtracted from the initial data operator. Via this modification, the stabilized fixed-point problem becomes well-posed for all parameters $\beta$ and $T$ in some open neighbourhood of $0$ and $1$, respectively. To extract among those solutions $\mathbf{u}_{\beta,T}$ the one to the original fixed-point equation, the idea is to adjust the parameters $\beta$ and $T$ \emph{simultaneously} so that the correction term vanishes. This is equivalent to showing that the linear functional, which is obtained from dual pairing with the correction term, is the zero functional. Using Brouwer's fixed-point theorem, the existence of such $\beta^{*}$ and $T^{*}$ are verified in the proof of \Cref{StableNonlinFlowMild}. Finally, restriction properties of the semigroup and the structure of the nonlinearity enable us to upgrade the mild solution to a jointly smooth classical solution $\mathbf{u}\in C^{\infty}\big( \overline{(0,\infty)\times\BB^{d}_{R}} \big)^{2}$. By means of a transition relation to the classical formulation, this translates effortlessly into a proof of \Cref{MainTHM}.
\end{itemize}
Since we can deal with all linear instabilities in the nonlinear wave evolution at once, our adaptation of the Lyapunov-Perron method is strikingly simpler than previous approaches via modulation theory. Also, it generalizes easily to any finite number of linear instabilities. Lastly, we would like to point out that our functional setting for the free wave equation in similarity coordinates provides the analytical basis for extensions to other self-similar coordinate systems. This proposes new and exciting directions for the study of self-similar blowup in other nonlinear wave equations, e.g., the wave maps equation and Yang-Mills equations.
\subsection{Notation}
In this paper, Einstein summation convention is imposed.
\par\medskip
The set of natural numbers and real numbers is denoted by $\mathbb{N}$, $\mathbb{R}$, respectively. In $d$-dimensional Euclidean space $\RR^{d}$, the set $\mathbb{B}^{d}_{R}\subset\RR^{d}$ denotes an open ball of radius $R>0$ centred about $0\in\RR^{d}$. The respective closed ball is given by $\overline{\mathbb{B}^{d}_{R}}\subset\RR^{d}$. The $(d-1)$-dimensional sphere of radius $R>0$ is denoted by $\mathbb{S}^{d-1}_{R}$ with spherical measure $\mathrm{d}\sigma^{d-1}_{R}(\omega)$. The unit sphere is denoted by $\mathbb{S}^{d-1}$ with spherical measure $\mathrm{d}\sigma(\omega)$.\par
\medskip
In the complex plane $\mathbb{C}$, given $\omega\in\RR$ we reserve the notation $\mathbb{H}_{\omega} = \{ z\in \mathbb{C} \mid \Re(z) > \omega \}$ for a right half-plane and $\overline{\mathbb{H}_{\omega}} = \{ z\in \mathbb{C} \mid \Re(z) \geq \omega \}$ a closed right half-plane. For $a\in\CC$ we denote by $\mathbb{D}_{r}(a) = \{ z\in \CC \mid |a-z| < r \}$ the open disk of radius $r>0$ and similarly the closed disk $\overline{\mathbb{D}_{r}(a)} = \{ z\in \CC \mid |a-z| \leq r \}$.\par
\medskip
Let $A$ be a set and $a_{\alpha},b_{\alpha}\in \RR_{\geq 0}$ for $\alpha\in A$. We define that the relation $a_{\alpha} \lesssim b_{\alpha}$ holds if there exists a uniform constant $C>0$ such that the inequality $a_{\alpha} \leq C b_{\alpha}$ holds for all $\alpha\in A$. As usual, the relation $a_{\alpha} \gtrsim b_{\alpha}$ is defined as $a_{\alpha} \lesssim b_{\alpha}$ and $a_{\alpha}\simeq b_{\alpha}$ is defined as $a_{\alpha} \lesssim b_{\alpha}$ and $a_{\alpha} \gtrsim b_{\alpha}$, respectively.\par
\medskip
If $\Omega\subset\RR^{d}$ is open and $f:\Omega\rightarrow\RR$ is differentiable, we denote by $\pd_{i} f: \Omega\rightarrow\RR$ the partial derivative with respect to the $i$-th slot. For a multi-index $\alpha\in\NN_{0}^{d}$ we set $\pd^{\alpha} f = \pd_{1}^{\alpha_{1}}\ldots\pd_{d}^{\alpha_{d}} f$. The gradient of $f$ is given by $(\pd f) = (\pd_{1} f, \ldots, \pd_{d} f)$. For functions $f$ in one variable we write $f'$ for the derivative. For functions $u$ on Minkowski space $\RR^{1,d}$, we denote by $\pd_{0} u$ the partial derivative with respect to the time slot, i.e., $(\pd_{0} u)(t,x) = \pd_{t} u (t,x)$. For facts about spherical harmonics $Y^{\ell}_{m}$ on $\mathbb{S}^{d-1}$ we refer to \cite[Appendix A]{MR3537340} and the textbook \cite{MR2934227}. For hypergeometric functions ${}_{2}F_{1}$ we adopt the conventions from \cite[Chapter 15]{MR2723248}.\par
\medskip
For a domain $\Omega\subset\RR^{d}$ we introduce with $C^{\infty}(\overline{\Omega})$ the set of all smooth functions on $\Omega$ all whose derivatives are continuous up to the boundary of $\Omega$. If $\Omega$ is bounded, we define for $k\in\mathbb{N}\cup\{0\}$ the classical inhomogeneous Sobolev norm and homogeneous Sobolev seminorm
\begin{equation*}
\| f \|_{H^{k}(\Omega)} = \left( \sum_{0\leq|\alpha|\leq k} \| \pd^{\alpha} f \|_{L^{2}(\Omega)}^{2} \right)^{\frac{1}{2}}
\quad\text{and}\quad
\| f \|_{\dot{H}^{k}(\Omega)} = \left( \sum_{|\alpha|=k} \| \pd^{\alpha} f \|_{L^{2}(\Omega)}^{2} \right)^{\frac{1}{2}}
\end{equation*}
for $f\in C^{\infty}(\overline{\Omega})$, respectively. The Sobolev space $H^{k}(\Omega)$ is then defined as the completion of $C^{\infty}(\overline{\Omega})$ with respect to the inhomogeneous Sobolev norm.
\par
\medskip
We use boldface notation for tuples of functions, e.g.,
\begin{equation*}
\mathbf{f} \equiv (f_{1}, f_{2}) \equiv
\begin{bmatrix}
f_{1} \\
f_{2}
\end{bmatrix}
\qquad\text{or}\qquad
\mathbf{u}(t,\,.\,) \equiv \big( u_{1}(t,\,.\,),
u_{2}(t,\,.\,) \big) \equiv
\begin{bmatrix}
u_{1}(t,\,.\,) \\
u_{2}(t,\,.\,)
\end{bmatrix}
\,.
\end{equation*}
Linear operators that act on tuples of functions are also displayed in boldface notation. For instance, when $\mathfrak{H} = H_{1} \times H_{2}$ is a product of Hilbert spaces then $\mathbf{L}: \dom(\mathbf{L}) \subset \mathfrak{H} \rightarrow \mathfrak{H}$ denotes a linear operator where its domain $\dom(\mathbf{L}) \subset \mathfrak{H}$ is a linear subspace of $\mathfrak{H}$. The set $\mathfrak{L}(\mathfrak{H})$ consists of all bounded linear operators on $\mathfrak{H}$ equipped with the operator norm. The resolvent set and spectrum of a closed linear operator $\mathbf{L}$ are denoted by $\varrho(\mathbf{L})$ and $\sigma(\mathbf{L})$, respectively. The resolvent map is given by $\mathbf{R}_{\mathbf{L}}: \varrho(\mathbf{L}) \rightarrow \mathfrak{L}(\mathfrak{H})$, $z\mapsto (z\mathbf{I}-\mathbf{L})^{-1}$. For spectral theory of linear operators and underlying concepts we refer to \cite{MR1335452}, in particular Chapter 3 thereof for results about the separation of the isolated point spectrum. Theory for strongly continuous operator semigroups is treated in the textbook \cite{MR1721989}.
\section{The free wave evolution in similarity coordinates}
In the following, a functional analytic framework is implemented for the flow associated to the wave operator
\begin{equation*}
(\Box u)(t,x) \coloneqq (-\pd_t^{2} + \Laplace_{x}) u(t,x) \,, \quad u\in C^{\infty}\big([0,T)\times\RR^{d}\big) \,,
\end{equation*}
in classical similarity coordinates
\begin{equation*}
{\chi\mathstrut}_{T}: [0,\infty)\times\RR^{d}\rightarrow\RR^{1,d}\,,\quad(\tau,\xi)\mapsto(T-T\ee^{-\tau},T\ee^{-\tau}\xi) \,.
\end{equation*}
Note that the map ${\chi\mathstrut}_{T}$ defines a diffeomorphism onto its image $[0,T)\times\RR^{d}$. If $v\in C^{\infty}\big((0,\infty)\times\RR^{d}\big)$ is related to $u\in C^{\infty}\big((0,T)\times\RR^{d}\big)$ via
\begin{equation*}
v = u \circ{\chi\mathstrut}_{T} \,,
\qquad\text{then}\qquad
\Box u\circ{\chi\mathstrut}_{T} = \Box_{{\chi\mathstrut}_{T}} v \,,
\end{equation*}
where
\begin{equation}
\label{LaplaceBeltrami}
(\Box_{{\chi\mathstrut}_{T}}v)(\tau,\xi) \coloneqq
\Big( \frac{\ee^\tau}{T} \Big)^{2} \Big( - \pd_\tau^{2} -2\xi^{i}\pd_{\xi^{i}}\pd_\tau + (\delta^{ij}-\xi^{i}\xi^{j})\pd_{\xi^{i}}\pd_{\xi^{j}} - \pd_\tau - 2\xi^{i}\pd_{\xi^{i}} \Big) v(\tau,\xi) \,.
\end{equation}
This can be cast in a first-order formalism by defining rescaled evolution variables
\begin{equation}
\label{EvoVar}
u_{1}(\tau,\xi) = (T\ee^{-\tau})^{s_{p}} (u\circ{\chi\mathstrut}_{T})(\tau,\xi) \,, \qquad
u_{2}(\tau,\xi) = (T\ee^{-\tau})^{s_{p}+1} (\pd_{0}u\circ{\chi\mathstrut}_{T})(\tau,\xi) \,,
\end{equation}
where for $p>1$ the scaling
\begin{equation*}
s_{p} = \frac{2}{p-1}
\end{equation*}
is motivated by the aim to detect self-similar blowup solutions of a scaling invariant equation. Utilising the transformations for spacetime derivatives,
\begin{align*}
&&
(\pd_{0} u\circ{\chi\mathstrut}_{T})(\tau,\xi) &= \frac{\ee^\tau}{T} (\pd_{\tau} + \xi^{i}\pd_{\xi^{i}}) (u\circ{\chi\mathstrut}_{T})(\tau,\xi) \,, & \\
&&
(\pd_{i} u\circ{\chi\mathstrut}_{T})(\tau,\xi) &= \frac{\ee^\tau}{T}\pd_{\xi^{i}} (u\circ{\chi\mathstrut}_{T})(\tau,\xi) \,, & i = 1,\ldots,d,
\end{align*}
we obtain
\begin{equation*}
\pd_{\tau} u_{1}(\tau,\xi) = -s_{p} u_{1}(\tau,\xi) - \xi^{i}\pd_{\xi^{i}}u_{1}(\tau,\xi) + u_{2}(\tau,\xi)
\end{equation*}
and for the wave operator
\begin{align*}
(\Box u\circ{\chi\mathstrut}_{T})(\tau,\xi) = \Big( \frac{\ee^\tau}{T} \Big)^{s_{p}+2} \Big(&
- \pd_{\tau} u_{2}(\tau,\xi)
- \xi^{i} \pd_{\xi^{i}} u_{2}(\tau,\xi)
- (s_{p}+1) u_{2}(\tau,\xi) \\&
+ \Laplace_{\xi} u_{1}(\tau,\xi) \Big) \,.
\end{align*}
This leads to the following operation.
\begin{definition}
\label{WaveFlowOperation}
Let $(d,p,R)\in\NN\times\RR_{>1}\times\RR_{\geq 1}$. For $\mathbf{f}\in C^{\infty}(\overline{\BB^{d}_{R}})^{2}$ we define
\begin{align*}
\mathbf{L}_{d,p}\mathbf{f} \in C^{\infty}(\overline{\BB^{d}_{R}})^{2}
\quad\text{by}\quad
(\mathbf{L}_{d,p} \mathbf{f})(\xi) =
\begin{bmatrix}
- s_{p} f_{1}(\xi)
- \xi^{i} (\pd_{i} f_{1})(\xi)
+ f_{2}(\xi) \\
(\Laplace f_{1})(\xi) -
(s_{p} + 1) f_{2}(\xi) -
\xi^{i} (\pd_{i} f_{2})(\xi)
\end{bmatrix}
\,.
\end{align*}
\end{definition}
As a result, the relation
\begin{equation}
\label{TransitionMechanism}
\pd_\tau
\begin{bmatrix}
u_{1}(\tau,\,.\,) \\
u_{2}(\tau,\,.\,)
\end{bmatrix}
=
\mathbf{L}_{d,p}
\begin{bmatrix}
u_{1}(\tau,\,.\,) \\
u_{2}(\tau,\,.\,)
\end{bmatrix}
-
\begin{bmatrix}
0 \\
(T\ee^{-\tau})^{s_{p}+2} (\Box u \circ {\chi\mathstrut}_{T})(\tau,\,.\,)
\end{bmatrix}
\end{equation}
holds between the above defined evolution variables.
\subsection{Inner products and dissipative estimates}
We show that the operation in \Cref{WaveFlowOperation} can be realized as the generator of a strongly continuous semigroup for the free wave flow in classical similarity coordinates. This requires a dissipative estimate.
\subsubsection{Standard energy}
The dissipativity will be derived from standard energy estimates for the free wave equation. On bounded domains, however, the standard energy only provides a seminorm and will not give rise to a normed space, let alone a Banach space. This can be fixed by a suitably chosen boundary term.
\begin{definition}
\label{StandardEnergy}
Let $(d,R)\in\NN\times\RR_{>0}$. Fix $\varepsilon_{1}>0$. We define a sesquilinear form on $C^{\infty}(\overline{\BB^{d}_{R}})^{2}$ by
\begin{equation*}
\Big( \mathbf{f} \,\Big|\, \mathbf{g} \Big)_{\mathfrak{E}^{1}(\BB^{d}_{R})} =
\int_{\BB^{d}_{R}} \overline{\pd^{i} f_{1}}\pd_{i} g_{1} + \int_{\BB^{d}_{R}} \overline{f_{2}} g_{2} + 2 \varepsilon_{1} R^{-1} \int_{\mathbb{S}^{d-1}_{R}} \overline{f_{1}} g_{1}
\end{equation*}
and set
\begin{equation*}
\| \mathbf{f} \|_{{\mathfrak{E}^{1}(\BB^{d}_{R})}} \coloneqq \sqrt{\Big( \mathbf{f} \,\Big|\, \mathbf{f} \Big)_{\mathfrak{E}^{1}(\BB^{d}_{R})}} \,.
\end{equation*}
\end{definition}
To see that the boundary term ensures positive definiteness, we give an elementary proof of the following trace inequality.
\begin{lemma}
\label{TraceInequality}
Let $(d,R)\in\NN\times\RR_{>0}$. Then
\begin{equation*}
\| f \|_{H^{1}(\BB^{d}_{R})} \simeq \| \pd f \|_{L^{2}(\BB^{d}_{R})} + \| f \|_{L^{2}(\mathbb{S}^{d-1}_R)}
\end{equation*}
for all $f\in C^{\infty}(\overline{\BB^{d}_{R}})$.
\end{lemma}
\begin{proof}
\begin{enumerate}[wide,itemsep=1em,topsep=1em]
\item[``$\gtrsim$'':] The divergence theorem gives
\begin{align*}
\int_{\pd \BB^{d}_{R}} |f(\xi)|^{2} \dd\sigma^{d-1}_{R}(\xi) &= \frac{1}{R} \int_{\mathbb{B}^{d}_{R}} \pd_{\xi^{i}} \big( \xi^{i} |f(\xi)|^{2} \big) \dd\xi \\
&= \frac{d}{R} \int_{\mathbb{B}^{d}_{R}} |f(\xi)|^{2} \dd\xi + \frac{2}{R} \int_{\mathbb{B}^{d}_{R}} \Re\big( \overline{\xi^{i} \pd_{\xi^{i}} f(\xi)} f(\xi) \big) \dd\xi \,.
\end{align*}
Thus
\begin{equation*}
\int_{\pd \BB^{d}_{R}} |f(\xi)|^{2} \dd\sigma^{d-1}_{R}(\xi)  \leq \frac{d+1}{R} \int_{\mathbb{B}^{d}_{R}} |f(\xi)|^{2} \dd\xi + R \int_{\mathbb{B}^{d}_{R}} |(\pd f)(\xi)|^{2} \dd\xi \,.
\end{equation*}
\item[``$\lesssim$'':] On the other hand, rearranging the equality above yields for any $\varepsilon>0$
\begin{align*}
\int_{\mathbb{B}^{d}_{R}} |f(\xi)|^{2} \dd\xi &= \frac{R}{d} \int_{\pd \BB^{d}_{R}} |f(\xi)|^{2} \dd\sigma^{d-1}_{R}(\xi)  - \frac{2}{d} \int_{\mathbb{B}^{d}_{R}} \Re\big( \overline{\xi^{i} \pd_{\xi^{i}} f(\xi)} f(\xi) \big) \dd\xi \\&\leq
\frac{R}{d} \int_{\pd \BB^{d}_{R}} |f(\xi)|^{2} \dd\sigma^{d-1}_{R}(\xi) + \frac{2\varepsilon}{d} \int_{\mathbb{B}^{d}_{R}} |f(\xi)|^{2} \dd\xi \\&\indent+ \frac{R^{2}}{2\varepsilon d} \int_{\mathbb{B}^{d}_{R}} |(\pd f)(\xi)|^{2} \dd\xi \,.
\end{align*}
We can choose $\varepsilon=\frac{d}{4}>0$ to conclude
\begin{equation*}
\int_{\mathbb{B}^{d}_{R}} |f(\xi)|^{2} \dd\xi \leq \frac{2R}{d} \int_{\pd \BB^{d}_{R}} |f(\xi)|^{2} \dd\xi + \Big( \frac{2R}{d} \Big)^{2} \int_{\mathbb{B}^{d}_{R}} |(\pd f)(\xi)|^{2} \dd\xi \,.
\qedhere
\end{equation*}
\end{enumerate}
\end{proof}
It follows that the sesquilinear form in \Cref{StandardEnergy} is an inner product with induced norm equivalent to a Sobolev norm.
\begin{lemma}
\label{EnergyNorm}
Let $(d,R)\in\NN\times\RR_{>0}$. Consider the form $\| \,.\, \|_{\mathfrak{E}^{1}(\BB^{d}_{R})}$ in \Cref{StandardEnergy} for a fixed $\varepsilon_{1}>0$. Then
\begin{equation*}
\| \mathbf{f} \|_{\mathfrak{E}^{1}(\BB^{d}_{R})}
\simeq
\big\| (f_{1},f_{2}) \big\|_{H^{1}(\BB^{d}_{R})\times L^{2}(\BB^{d}_{R})}
\end{equation*}
for all $\mathbf{f} = (f_{1},f_{2}) \in C^{\infty}(\overline{\BB^{d}_{R}})^{2}$.
\end{lemma}
\begin{proof}
The trace inequality in \Cref{TraceInequality} yields
\begin{equation*}
\| \mathbf{f} \|_{\mathfrak{E}^{1}(\BB^{d}_{R})} \simeq \| \pd f_{1} \|_{L^{2}(\BB^{d}_{R})} + \| f_{2} \|_{L^{2}(\BB^{d}_{R})} + \| f_{1} \|_{L^{2}(\mathbb{S}^{d-1}_R)} \simeq \big\| (f_{1},f_{2}) \big\|_{H^{1}(\BB^{d}_{R})\times L^{2}(\BB^{d}_{R})}
\end{equation*}
for all $\mathbf{f} = (f_{1},f_{2}) \in C^{\infty}(\overline{\BB^{d}_{R}})^{2}$.
\end{proof}
The next crucial feature of the inner product is that the standard energy estimate for the free wave equation is incarnated as a dissipative estimate for the wave evolution operation.
\begin{lemma}
\label{EnergyDissipativity}
Let $(d,p,R)\in\NN\times\RR_{>1}\times\RR_{\geq 1}$. Fix $\varepsilon_{1}>0$ in \Cref{StandardEnergy}. We have
\begin{equation*}
\Re\Big( \mathbf{L}_{d,p} \mathbf{f} \,\Big|\, \mathbf{f}\Big)_{\mathfrak{E}^{1}(\BB^{d}_{R})} \leq
\Big( \frac{d}{2}-s_{p}-1\Big) \| \mathbf{f} \|_{\mathfrak{E}^{1}(\BB^{d}_{R})}^{2} +
\Big(-\frac{d}{2}+1+\varepsilon_{1}\Big) 2 \varepsilon_{1} R^{-1} \int_{\mathbb{S}^{d-1}_R} |f_{1}|^{2}
\end{equation*}
for all $\mathbf{f}\in C^{\infty}(\overline{\BB^{d}_{R}})^{2}$.
\end{lemma}
\begin{proof}
The integrands in the inner product are
\begin{align*}
\overline{(\pd_{i}[\mathbf{L}_{d,p}\mathbf{f}]_1)(\xi)} (\pd^{i} f_{1})(\xi) &=
-(s_{p}+1) |(\pd f_{1})(\xi)|^{2}
-\overline{\xi^{j} (\pd_{i} \pd_{j} f_{1})(\xi)} (\pd^{i} f_{1})(\xi) \\&\indent
- (\Laplace f_{1})(\xi) \overline{f_{2}(\xi)}
+\pd_{\xi^{i}} \Big( (\pd^{i} f_{1})(\xi) \overline{f_{2}(\xi)} \Big)
\end{align*}
and
\begin{equation*}
\overline{[\mathbf{L}_{d,p}\mathbf{f}]_2(\xi)} f_{2}(\xi) =
\overline{(\Laplace f_{1})(\xi)} f_{2}(\xi) - (s_{p}+1)|f_{2}(\xi)|^{2} - \overline{\xi^{j}(\pd_{j} f_{2})(\xi)} f_{2}(\xi) \,.
\end{equation*}
Using the fact
\begin{equation*}
\Re \Big( \overline{\xi^{j} (\pd_{j} \pd_{i} f_{1})(\xi)} (\pd^{i} f_{1})(\xi) \Big) = \frac{1}{2} \pd_{\xi^{j}} \Big( |(\pd f_{1})(\xi)|^{2} \xi^{j} \Big) - \frac{d}{2} |(\pd f_{1})(\xi)|^{2} \,,
\end{equation*}
the real part of the sum of the above integrands equates to
\begin{align*}
&\Re\Big(
\overline{(\pd_{i}[\mathbf{L}_{d,p}\mathbf{f}]_1)(\xi)} (\pd^{i} f_{1})(\xi) +
\overline{[\mathbf{L}_{d,p}\mathbf{f}]_2(\xi)}f_{2}(\xi)
\Big) \\&\indent=
\Big( \frac{d}{2} - s_{p} - 1 \Big) \Big( |(\pd f_{1})(\xi)|^{2} + |f_{2}(\xi)|^{2} \Big) + \pd_{\xi^{i}} \Re \Big( \overline{ f_{2}(\xi)} (\pd^{i} f_{1})(\xi) \Big) \\&\indent\indent-
\frac{1}{2} \pd_{\xi^{i}} \Big( |(\pd f_{1})(\xi)|^{2} \xi^{i} + |f_{2}(\xi)|^{2} \xi^{i}\Big) \,.
\end{align*}
The real part of the integrand in the boundary term is given by
\begin{equation*}
\Re\Big( \overline{[\mathbf{L}_{d,p}\mathbf{f}]_1(\xi)} f_{1}(\xi) \Big) =
- s_{p} |f_{1}(\xi)|^{2} + \Re\Big( \overline{f_{2}(\xi)} f_{1}(\xi) - \overline{\xi^{i}(\pd_{i} f_{1})(\xi)} f_{1}(\xi) \Big) \,.
\end{equation*}
Consequently,
\begin{align*}
&\Re\Big( \mathbf{L}_{d,p} \mathbf{f} \,\Big|\, \mathbf{f} \Big)_{\mathfrak{E}^{1}(\BB^{d}_{R})} =
\Big( \frac{d}{2} - s_{p} - 1 \Big) \int_{\BB^{d}_{R}} \Big( |(\pd f_{1})(\xi)|^{2} + |f_{2}(\xi)|^{2} \Big) \dd\xi
\\&\,+
\Big( -s_{p} + \varepsilon_{1} \Big) 2\varepsilon_{1} R^{-1} \int_{\mathbb{S}^{d-1}_R} |f_{1}(\xi)|^{2} \dd\sigma^{d-1}_{R}(\xi)
\\&\,+
R^{-1} \int_{\mathbb{S}^{d-1}_{R}} \Re \Big(
\overline{f_{2}(\xi)} \xi^{i} (\pd_{i} f_{1})(\xi) -
\overline{\xi^{i}  (\pd_{i} f_{1})(\xi)} 2\varepsilon_{1} f_{1}(\xi) +
\overline{f_{2}(\xi)} 2\varepsilon_{1} f_{1}(\xi)
\Big) \dd\sigma^{d-1}_{R}(\xi)
\\&\,-
R^{-1} \int_{\mathbb{S}^{d-1}_R} \frac{1}{2} \Big(
|\xi|^{2} |(\pd f_{1})(\xi)|^{2} +
|f_{2}(\xi)|^{2} +
4\varepsilon_{1}^{2} |f_{1}(\xi)|^{2} \Big) \dd\sigma^{d-1}_{R}(\xi)
\\&\,+
\frac{1-R^{2}}{2R} \int_{\mathbb{S}^{d-1}_R} |f_{2}(\xi)|^{2} \dd\sigma^{d-1}_{R}(\xi) \,.
\end{align*}
Together with the elementary inequality
\begin{align*}
&
\Re \Big(
\overline{f_{2}(\xi)} \xi^{i} (\pd_{i} f_{1})(\xi) -
\overline{\xi^{i}  (\pd_{i} f_{1})(\xi)} 2\varepsilon_{1} f_{1}(\xi) +
\overline{f_{2}(\xi)} 2\varepsilon_{1} f_{1}(\xi)
\Big)
\\&\indent\indent\leq
\frac{1}{2} \Big( |\xi|^{2} |(\pd f_{1})(\xi)|^{2} + |f_{2}(\xi)|^{2} + 4\varepsilon_{1}^{2} |f_{1}(\xi)|^{2} \Big) \,,
\end{align*}
and since $R\geq 1$, the estimate
\begin{align*}
&
\Re\Big( \mathbf{L}_{d,p} \mathbf{f} \,\Big|\, \mathbf{f} \Big)_{\mathfrak{E}^1(\BB^{d}_{R})} \leq
\Big( \frac{d}{2} - s_{p} - 1 \Big) \int_{\BB^{d}_{R}} \Big( |(\pd f_{1})(\xi)|^{2} + |f_{2}(\xi)|^{2} \Big) \dd\xi \\&\indent\indent+
\Big( - s_{p} + \varepsilon_{1} \Big) 2\varepsilon_{1} R^{-1} \int_{\mathbb{S}^{d-1}_R} |f_{1}(\xi)|^{2} \dd\sigma^{d-1}_{R}(\xi) \\&\indent=
\Big( \frac{d}{2} - s_{p} - 1 \Big) \| \mathbf{f} \|_{\mathfrak{E}^1(\BB^{d}_{R})}^{2} + \Big( - \frac{d}{2} + 1 + \varepsilon_{1} \Big) 2\varepsilon_{1} R^{-1} \int_{\mathbb{S}^{d-1}_R} |f_{1}(\xi)|^{2} \dd\sigma^{d-1}_{R}(\xi)
\end{align*}
follows.
\end{proof}
\subsubsection{Higher energies}
The standard energy can be upgraded to higher energies with the following differential operators.
\begin{definition}
\label{DmuOperations}
Let $(d,R)\in\NN\times\RR_{>0}$. For $\mathbf{f}\in C^{\infty}(\overline{\BB^{d}_{R}})^{2}$ and $\mu=0,1,\ldots,d$ we define
\begin{align*}
\mathbf{D}_{\mu}\mathbf{f} \in C^{\infty}(\overline{\BB^{d}_{R}})^{2}
\qquad\text{by}\qquad
\mathbf{D}_{\mu} \mathbf{f} =
\left[
\begin{array}{l}
\delta_{\mu}{}^{i}\pd_{i} f_{1} +
\delta_{\mu}{}^{0} f_{2} \\
\delta_{\mu}{}^{i} \pd_{i} f_{2} +
\delta_{\mu}{}^{0} \Laplace f_{1}
\end{array}
\right]
\,.
\end{align*}
\end{definition}
These operators satisfy an essential commutation relation.
\begin{lemma}
\label{LdpDmu}
Let $(d,p,R)\in\NN\times\RR_{>1}\times\RR_{\geq 1}$. We have
\begin{equation*}
\mathbf{D}_{\mu} \mathbf{L}_{d,p} \mathbf{f} = \mathbf{L}_{d,p} \mathbf{D}_{\mu} \mathbf{f} - \mathbf{D}_{\mu} \mathbf{f}
\end{equation*}
for all $\mathbf{f}\in C^{\infty}(\overline{\BB^{d}_{R}})^{2}$ and $\mu = 0,1,\ldots,d$.
\end{lemma}
\begin{proof}
We prove this with a direct computation, namely
\begin{align*}
(\mathbf{D}_{i} \mathbf{L}_{d,p} \mathbf{f})(\xi) &=
\begin{bmatrix}
\pd_{\xi^{i}} \big( - s_{p} f_{1}(\xi) - \xi^{j} \pd_{\xi^{j}} f_{1}(\xi) + f_{2}(\xi) \big) \\
\pd_{\xi^{i}} \big( \Laplace_{\xi} f_{1}(\xi) - ( s_{p} + 1 ) f_{2}(\xi) - \xi^{j} \pd_{\xi^{j}} f_{2}(\xi) \big)
\end{bmatrix}
\\&=
\begin{bmatrix}
- s_{p} \pd_{\xi^{i}} f_{1}(\xi) - \xi^{j} \pd_{\xi^{j}} \pd_{\xi^{i}} f_{1}(\xi) + \pd_{\xi^{i}} f_{2}(\xi) \\
\Laplace_{\xi} \pd_{\xi^{i}} f_{1}(\xi) - ( s_{p} + 1 ) \pd_{\xi^{i}} f_{2}(\xi) - \xi^{j} \pd_{\xi^{j}} \pd_{\xi^{i}} f_{2}(\xi)
\end{bmatrix}
-
\begin{bmatrix}
\pd_{\xi^{i}} f_{1}(\xi) \\
\pd_{\xi^{i}} f_{2}(\xi)
\end{bmatrix}
\\&=
(\mathbf{L}_{d,p} \mathbf{D}_{i} \mathbf{f})(\xi) - (\mathbf{D}_{i}\mathbf{f})(\xi)
\end{align*}
and
\begin{align*}
(\mathbf{D}_{0} \mathbf{L}_{d,p} \mathbf{f})(\xi) &=
\begin{bmatrix}
\Laplace_{\xi} f_{1}(\xi) - ( s_{p} + 1 ) f_{2}(\xi) - \xi^{j} \pd_{\xi^{j}} f_{2}(\xi) \\
\Delta_{\xi} \big( - s_{p} f_{1}(\xi) - \xi^{j} \pd_{\xi^{j}} f_{1}(\xi) + f_{2}(\xi) \big)
\end{bmatrix}
\\&=
\begin{bmatrix}
 - s_{p} f_{2}(\xi) - \xi^{j} \pd_{\xi^{j}} f_{2}(\xi) + \Laplace_{\xi} f_{1}(\xi) \\
\Delta_{\xi} f_{2}(\xi) - (s_{p} + 1) \Delta_{\xi} f_{1}(\xi) - \xi^{j} \pd_{\xi^{j}} \Delta_{\xi} f_{1}(\xi) \big)
\end{bmatrix}
-
\begin{bmatrix}
f_{2}(\xi) \\
\Delta_{\xi} f_{1}(\xi)
\end{bmatrix}
\\&=
(\mathbf{L}_{d,p} \mathbf{D}_{0} \mathbf{f})(\xi) - (\mathbf{D}_{0}\mathbf{f})(\xi) \,.
\qedhere
\end{align*}
\end{proof}
Higher inner products are composed recursively by inserting these differential operators in the inner product from \Cref{StandardEnergy} and augmenting the result once again with a boundary term.
\begin{definition}
\label{InnerProducts}
Let $(d,R)\in\NN\times\RR_{>0}$. Fix $\varepsilon_{j}>0$ for $1 \leq j < \frac{d}{2}+1$. Based on \Cref{StandardEnergy}, we define sesquilinear forms on $C^{\infty}(\overline{\BB^{d}_{R}})^{2}$ recursively by
\begin{equation*}
\Big( \mathbf{f} \,\Big|\, \mathbf{g} \Big)_{\mathfrak{E}^{k}(\BB^{d}_{R})} =
\renewcommand{\arraystretch}{1.5}
\left\{
\begin{array}{ll}
\displaystyle{
\sum_{\mu=0}^{d} \Big( \mathbf{D}_{\mu} \mathbf{f} \,\Big|\, \mathbf{D}_{\mu} \mathbf{g} \Big)_{\mathfrak{E}^{k-1}(\BB^{d}_{R})} +
\frac{2 \varepsilon_{k}}{R} \int_{\mathbb{S}^{d-1}_R} \overline{f_{1}} g_1
} & \text{if } 2\leq k < \frac{d}{2} + 1 \,,
\\ \displaystyle{
\sum_{\mu=0}^{d} \Big( \mathbf{D}_{\mu} \mathbf{f} \,\Big|\, \mathbf{D}_{\mu} \mathbf{g} \Big)_{\mathfrak{E}^{k-1}(\BB^{d}_{R})} +
\Big( \mathbf{f} \,\Big|\, \mathbf{g} \Big)_{\mathfrak{E}^{k-1}(\BB^{d}_{R})}
} & \text{if } \frac{d}{2} + 1 \leq k \,,
\end{array}
\right.
\end{equation*}
and set
\begin{equation*}
\| \mathbf{f} \|_{{\mathfrak{E}^{k}(\BB^{d}_{R})}} \coloneqq \sqrt{\Big( \mathbf{f} \,\Big|\, \mathbf{f} \Big)_{\mathfrak{E}^{k}(\BB^{d}_{R})}} \,.
\end{equation*}
\end{definition}
These sesquilinear forms are positive definite with induced norms equivalent to higher Sobolev norms.
\begin{lemma}
\label{HigherEnergyNorm}
Let $(d,k,R)\in\NN\times\NN\times\RR_{>0}$. Consider the form $\| \,.\, \|_{\mathfrak{E}^{k}(\BB^{d}_{R})}$ in \Cref{InnerProducts} for fixed $\varepsilon_{j}>0$, where $1 \leq j < \frac{d}{2} + 1$. Then, the equivalence
\begin{equation*}
\| \mathbf{f} \|_{\mathfrak{E}^{k}(\BB^{d}_{R})}
\simeq
\big\| (f_{1},f_{2}) \big\|_{H^{k}(\BB^{d}_{R})\times H^{k-1}(\BB^{d}_{R})}
\end{equation*}
holds for all $\mathbf{f} = (f_{1},f_{2}) \in C^{\infty}(\overline{\BB^{d}_{R}})^{2}$.
\end{lemma}
\begin{proof}
\begin{enumerate}[wide, itemsep=1em, topsep=1em]
\item[``$\lesssim$'':] The inequality is valid in the base case $k=1$ by \Cref{EnergyNorm}. To proceed, note that for each $k\in\NN$ and $\mu=0,1,\ldots,d$
\begin{align*}
\| \mathbf{D}_{\mu} \mathbf{f}\|_{H^{k}(\BB^{d}_{R}) \times H^{k-1}(\BB^{d}_{R})} &\lesssim
\| \pd f_{1} \|_{H^{k}(\BB^{d}_{R})} + \| f_{2} \|_{H^{k}(\BB^{d}_{R})} \\&\indent+
\| \Laplace f_{1} \|_{H^{k-1}(\BB^{d}_{R})} +
\| f_{2} \|_{H^{k-1}(\BB^{d}_{R})} +
\| \pd f_{2} \|_{H^{k-1}(\BB^{d}_{R})} \\&\lesssim
\big\| (f_{1},f_{2}) \big\|_{H^{k+1}(\BB^{d}_{R})\times H^{k}(\BB^{d}_{R})}
\end{align*}
for all $\mathbf{f} \in C^{\infty}(\overline{\BB^{d}_{R}})^{2}$. Now, assume that for an arbitrary but fixed $k\in\NN$ the bound
\begin{equation*}
\| \mathbf{f} \|_{\mathfrak{E}^{k}(\BB^{d}_{R})} \lesssim \big\| (f_{1},f_{2}) \big\|_{H^{k}(\BB^{d}_{R})\times H^{k-1}(\BB^{d}_{R})}
\end{equation*}
holds for all $\mathbf{f}\in C^{\infty}(\overline{\BB^{d}_{R}})^{2}$. If $1\leq k< \frac{d}{2}$ we conclude from this with the trace inequality from \Cref{TraceInequality}
\begin{align*}
\| \mathbf{f} \|_{\mathfrak{E}^{k+1}(\BB^{d}_{R})} &\simeq \sum_{\mu=0}^{d} \| \mathbf{D}_{\mu} \mathbf{f} \|_{\mathfrak{E}^{k}(\BB^{d}_{R})} + \| f_{1} \|_{L^{2}(\mathbb{S}^{d-1}_R)} \\&\lesssim
\sum_{\mu=0}^{d} \| \mathbf{D}_{\mu} \mathbf{f} \|_{H^{k}(\BB^{d}_{R})\times H^{k-1}(\BB^{d}_{R})} + \| f_{1} \|_{L^{2}(\mathbb{S}^{d-1}_R)} \\&\lesssim
\big\| (f_{1},f_{2}) \big\|_{H^{k+1}(\BB^{d}_{R})\times H^{k}(\BB^{d}_{R})}
\end{align*}
for all $\mathbf{f}\in C^{\infty}(\overline{\BB^{d}_{R}})^{2}$. In case $k \geq \frac{d}{2} + 1$ the induction step follows directly from the above estimate and the induction hypothesis.
\item[``$\gtrsim$'':] Again, the inequality is valid in the base case $k=1$ by \Cref{EnergyNorm}. By induction, assume that for an arbitrary but fixed $k\in\NN$ the bound
\begin{equation*}
\| \mathbf{f} \|_{\mathfrak{E}^{k}(\BB^{d}_{R})} \gtrsim \big\| (f_{1},f_{2}) \big\|_{H^{k}(\BB^{d}_{R})\times H^{k-1}(\BB^{d}_{R})}
\end{equation*}
holds for all $\mathbf{f}\in C^{\infty}(\overline{\BB^{d}_{R}})^{2}$. Then, if $1\leq k< \frac{d}{2}$, we get with $\pd_{i} \mathbf{f} = \mathbf{D}_{i} \mathbf{f}$ the estimate
\begin{align*}
\big\| (f_{1},f_{2}) & \big\|_{H^{k+1}(\BB^{d}_{R})\times H^{k}(\BB^{d}_{R})} \\&\simeq \sum_{i=1}^{d} \big\| (\pd_{i} f_{1}, \pd_{i} f_{2}) \big\|_{H^{k}(\BB^{d}_{R})\times H^{k-1}(\BB^{d}_{R})} + \| f_{1} \|_{L^{2}(\BB^{d}_{R})} + \| f_{2} \|_{L^{2}(\BB^{d}_{R})} \\&\simeq
\sum_{i=1}^{d} \big\| (\pd_{i} f_{1}, \pd_{i} f_{2}) \big\|_{H^{k}(\BB^{d}_{R})\times H^{k-1}(\BB^{d}_{R})} + \| f_{1} \|_{L^{2}(\mathbb{S}^{d-1}_R)} + \| [\mathbf{D}_{0} \mathbf{f}]_1 \|_{L^{2}(\BB^{d}_{R})} \\&\lesssim
\sum_{\mu=0}^{d} \| \mathbf{D}_{\mu} \mathbf{f} \|_{H^{k}(\BB^{d}_{R})\times H^{k-1}(\BB^{d}_{R})} + \| f_{1} \|_{L^{2}(\mathbb{S}^{d-1}_R)}  \\&\lesssim
\sum_{\mu=0}^{d} \| \mathbf{D}_{\mu} \mathbf{f} \|_{\mathfrak{E}^{k}(\BB^{d}_{R})} + \| f_{1} \|_{L^{2}(\mathbb{S}^{d-1}_R)} \\&\simeq
\| \mathbf{f} \|_{\mathfrak{E}^{k+1}(\BB^{d}_{R})}
\end{align*}
for all $\mathbf{f}\in C^{\infty}(\overline{\BB^{d}_{R}})^{2}$. In case $k \geq \frac{d}{2} + 1$ the induction step is again immediate and this proves the other direction.
\qedhere
\end{enumerate}
\end{proof}
The inner products are by design compatible with the dissipative properties of the wave evolution.
\begin{proposition}
\label{HigherEnergyDissipativity}
Let $(d,p,k,R)\in\NN\times\RR_{>1}\times\NN\times\RR_{\geq 1}$. Fix $0<\varepsilon_{1}<\frac{1}{2}$ in \Cref{StandardEnergy} and
\begin{equation*}
\varepsilon_{j} = \Big( \frac{\varepsilon_{1}}{16 R^{2}} \Big)^{j-1} \varepsilon_{1} \prod_{i=1}^{j-1} \Big( \frac{d}{2} - i - \varepsilon_{1} \Big) \quad \text{for } 2 \leq j < \frac{d}{2} + 1
\end{equation*}
in \Cref{InnerProducts}. We have
\begin{align*}
&
\Re\Big( \mathbf{L}_{d,p} \mathbf{f} \,\Big|\, \mathbf{f} \Big)_{\mathfrak{E}^{k}(\BB^{d}_{R})} \\&\hspace{1em}\leq
\renewcommand{\arraystretch}{1.5}
\left\{
\begin{array}{ll}
\displaystyle{
\Big( \frac{d}{2}-s_{p}-k\Big) \| \mathbf{f} \|_{\mathfrak{E}^{k}(\BB^{d}_{R})}^{2} +
\Big(-\frac{d}{2}+k+\varepsilon_{1}\Big) \frac{2 \varepsilon_{k}}{R} \int_{\mathbb{S}^{d-1}_R} |f_{1}|^{2}
} & \text{ if } k < \frac{d}{2} + 1 \,, \\
\displaystyle{
( -s_{p}+\varepsilon_{1}) \| \mathbf{f} \|_{\mathfrak{E}^{k}(\BB^{d}_{R})}^{2}
} & \text{ if } \frac{d}{2} + 1 \leq k \,,
\end{array}
\right.
\end{align*}
for all $\mathbf{f}\in C^{\infty}(\overline{\BB^{d}_{R}})^{2}$.
\end{proposition}
\begin{proof}
\Cref{LdpDmu} implies for each $k\in\NN$
\begin{equation*}
\sum_{\mu=0}^{d} \Re\Big( \mathbf{D}_{\mu} \mathbf{L}_{d,p} \mathbf{f} \,\Big|\, \mathbf{D}_{\mu} \mathbf{f} \Big)_{\mathfrak{E}^{k}(\BB^{d}_{R})} = \sum_{\mu=0}^{d} \Re\Big( \mathbf{L}_{d,p} \mathbf{D}_{\mu} \mathbf{f} \,\Big|\, \mathbf{D}_{\mu} \mathbf{f} \Big)_{\mathfrak{E}^{k}(\BB^{d}_{R})} - \sum_{\mu=0}^{d} \| \mathbf{D}_{\mu}\mathbf{f} \|_{\mathfrak{E}^{k}(\BB^{d}_{R})}^{2}
\end{equation*}
and Young's inequality yields for the boundary term
\begin{align*}
&\Re\int_{\mathbb{S}^{d-1}_R}\overline{[\mathbf{L}_{d,p}\mathbf{f}]_1(\xi)} f_{1}(\xi) \dd\sigma^{d-1}_{R}(\xi) \\&\indent=
-s_{p} \int_{\mathbb{S}^{d-1}_R} |f_{1}(\xi)|^{2} \dd\sigma^{d-1}_{R}(\xi) \\&\indent\indent+ \Re \int_{\mathbb{S}^{d-1}_R} \Big( \overline{[\mathbf{D}_0\mathbf{f}]_1} f_{1}(\xi) - \overline{\xi^{i} [\mathbf{D}_i\mathbf{f}]_1} f_{1}(\xi) \Big) \dd \sigma^{d-1}_{R}(\xi) \\&\indent\leq
(-s_{p}+\varepsilon_{1}) \int_{\mathbb{S}^{d-1}_R} |f_{1}(\xi)|^{2} \dd\sigma^{d-1}_{R}(\xi) +
\frac{8R^{2}}{\varepsilon_{1}} \int_{\mathbb{S}^{d-1}_R} \sum_{\mu=0}^{d} \Big| [\mathbf{D}_{\mu}\mathbf{f}]_1(\xi)\Big|^{2} \dd\sigma^{d-1}_{R}(\xi) \,.
\end{align*}
We have proved in \Cref{EnergyDissipativity} that the estimate holds in case $k=1$. Let us proceed by induction and assume that the dissipative estimate holds for an arbitrary but fixed $k\in\NN$.\par
\medskip If $1 < k < \frac{d}{2} + 1$, we infer from this and the induction hypothesis
\begin{align*}
&\Re\Big( \mathbf{L}_{d,p} \mathbf{f} \,\Big|\, \mathbf{f} \Big)_{\mathfrak{E}^{k}(\BB^{d}_{R})} \leq
\Big( \frac{d}{2}-s_{p}-k\Big) \sum_{\mu=0}^{d} \| \mathbf{D}_{\mu}\mathbf{f} \|_{\mathfrak{E}^{k-1}(\BB^{d}_{R})}^{2} \\&\indent\indent+
(-s_{p}+\varepsilon_{1})\frac{2 \varepsilon_{k}}{R} \int_{\mathbb{S}^{d-1}_R} |f_{1}(\xi)|^{2} \dd\sigma^{d-1}_{R}(\xi) \\&\indent\indent+
\Big( \varepsilon_{k} \frac{8R^{2}}{\varepsilon_{1}\varepsilon_{k-1}} - \frac{d}{2} + k-1 + \varepsilon_{1} \Big) \frac{2 \varepsilon_{k-1}}{R} \int_{\mathbb{S}^{d-1}_R} \sum_{\mu=0}^{d} \Big|[\mathbf{D}_{\mu}\mathbf{f}]_1(\xi)\Big|^{2} \dd\sigma^{d-1}_{R}(\xi) \\&\indent\leq
\Big( \frac{d}{2}-s_{p}-k\Big) \sum_{\mu=0}^{d} \| \mathbf{D}_{\mu}\mathbf{f} \|_{\mathfrak{E}^{k-1}(\BB^{d}_{R})}^{2} \\&\indent\indent+
(-s_{p}+\varepsilon_{1})\frac{2 \varepsilon_{k}}{R} \int_{\mathbb{S}^{d-1}_R} |f_{1}(\xi)|^{2} \dd\sigma^{d-1}_{R}(\xi) \\&\indent=
\Big( \frac{d}{2}-s_{p}-k\Big) \| \mathbf{f} \|_{\mathfrak{E}^{k}(\BB^{d}_{R})}^{2} + \Big(-\frac{d}{2} + k + \varepsilon_{1} \Big) \frac{2 \varepsilon_{k}}{R} \int_{\mathbb{S}^{d-1}_R} |f_{1}(\xi)|^{2} \dd\sigma^{d-1}_{R}(\xi) \,,
\end{align*}
since
\begin{equation*}
\varepsilon_{k} = \frac{\varepsilon_{1}}{16 R^{2}}\varepsilon_{k-1} \Big( \frac{d}{2}-(k-1)-\varepsilon_{1} \Big) > 0 \,,
\end{equation*}
and the bound is proved. Notice that if $k = \Big\lceil\frac{d}{2}\Big\rceil$ we have $k < \frac{d}{2} + 1$ but $k + 1 \geq \frac{d}{2}$ and the bound reads
\begin{align*}
\Re\Big( \mathbf{L}_{d,p} \mathbf{f} \,\Big|\, \mathbf{f} \Big)_{\mathfrak{E}^{k}(\BB^{d}_{R})} &\leq
\Big( \frac{d}{2}-s_{p}-\Big\lceil\frac{d}{2}\Big\rceil\Big) \| \mathbf{f} \|_{\mathfrak{E}^{k}(\BB^{d}_{R})}^{2} \\&\indent+
\Big(-\frac{d}{2}+\Big\lceil\frac{d}{2}\Big\rceil+\varepsilon_{1}\Big) \frac{2 \varepsilon_{k}}{R} \int_{\mathbb{S}^{d-1}_R} |f_{1}|^{2} \\&\leq
(-s_{p}+\varepsilon_{1}) \| \mathbf{f} \|_{\mathfrak{E}^{k}(\BB^{d}_{R})}^{2} \,.
\end{align*}
\par
\medskip
Finally, if $k\geq \frac{d}{2}+1$ we conclude from the induction hypothesis
\begin{align*}
&
\Re\Big( \mathbf{L}_{d,p} \mathbf{f} \,\Big|\, \mathbf{f} \Big)_{\mathfrak{E}^{k}(\BB^{d}_{R})} = \sum_{\mu=0}^{d} \Re\Big( \mathbf{D}_{\mu}\mathbf{L}_{d,p} \mathbf{f} \,\Big|\,\mathbf{D}_{\mu}\mathbf{f} \Big)_{\mathfrak{E}^{k-1}(\BB^{d}_{R})} + \Re\Big( \mathbf{L}_{d,p} \mathbf{f} \,\Big|\, \mathbf{f} \Big)_{\mathfrak{E}^{k-1}(\BB^{d}_{R})} \\&\indent\leq
(-s_{p}+\varepsilon_{1}) \Big( \sum_{\mu=0}^{d} \| \mathbf{D}_{\mu}\mathbf{f} \|_{\mathfrak{E}^{k-1}(\BB^{d}_{R})}^{2} + \| \mathbf{f} \|_{\mathfrak{E}^{k-1}(\BB^{d}_{R})}^{2} \Big) - \sum_{\mu=0}^{d} \| \mathbf{D}_{\mu}\mathbf{f} \|_{\mathfrak{E}^{k-1}(\BB^{d}_{R})}^{2} \\&\indent\leq
(-s_{p}+\varepsilon_{1}) \| \mathbf{f} \|_{\mathfrak{E}^{k}(\BB^{d}_{R})}^{2}
\end{align*}
and the estimates are proved.
\end{proof}
\subsection{Properties of the range}
Besides dissipative estimates, we need a density property of the range of the wave evolution operation. This will be established via the following technical lemma.
\begin{lemma}
\label{EllipticApproximation} 
Let $(d,k,R)\in\NN\times\NN\times\RR_{\geq 1}$. For all $\varepsilon>0$ and all $F\in C^{\infty}(\overline{\BB^{d}_{R}})$ there is an $f\in C^{\infty}(\overline{\BB^{d}_{R}})$ such that
\begin{equation*}
\big\| (\Box_{{\chi\mathstrut}_{T}} v)(\log T, \,.\,) - F \big\|_{H^{k-1}(\BB^{d}_{R})} < \varepsilon \,, \qquad v(\tau,\,.\,) \coloneqq \Big( \frac{\ee^{\tau}}{T} \Big)^{\frac{d}{2}} f \,.
\end{equation*}
\end{lemma}
\begin{proof}
Let $f,F\in C^{\infty}(\overline{\BB^{d}_{R}})$ and for $\lambda\in\CC$ define $v_{\lambda}(\tau,\,.\,) \coloneqq \Big( \frac{\ee^{\tau}}{T} \Big)^{\lambda} f$. Using \Cref{LaplaceBeltrami},
\begin{equation}
\label{EllipticProblem}
(\Box_{{\chi\mathstrut}_{T}} v_{\lambda})(\log T,\,.\,) - F =
\Big( (\delta^{ij}-\xi^{i}\xi^{j}) \pd_{i}\pd_{j} - 2 (\lambda+1) \xi^{i} \pd_{i} - \lambda(\lambda+1) \Big) f - F \,.
\end{equation}
So, the task in this lemma is the construction of an approximate solution to a degenerate elliptic equation.\par
\medskip
In the multidimensional case $d\geq 2$ we can employ a spherical harmonics decomposition. Indeed, for $n\in\NN$ and $\xi \in \overline{\BB^{d}_{R}}\setminus\{0\}$ consider
\begin{equation*}
F_{n}(\xi) = \sum_{\ell=0}^n\sum_{m\in\Omega_{d,\ell}} F_{\ell,m}(|\xi|) Y_{\ell,m}(\tfrac{\xi}{|\xi|}) \,, \qquad
F_{\ell,m}(\rho) \coloneqq \int_{\mathbb{S}^{d-1}} \overline{Y_{\ell,m}(\omega)} F(\rho\omega) \dd\sigma(\omega) \,,
\end{equation*}
which, by \cite[Lemma A.2]{MR3537340}, defines a function $F_{n}\in C^{\infty}(\overline{\BB^{d}_{R}})$ with
\begin{equation*}
\lim_{n\to\infty} \| F_{n} - F \|_{H^{k-1}(\BB^{d}_{R})} = 0 \,.
\end{equation*}
Hence, given $\varepsilon>0$, there is an $N\in\NN$ such that
\begin{equation}
\label{DensityEstimate}
\big\| (\Box_{{\chi\mathstrut}_{T}} v_{\lambda})(\log T, \,.\,) - F \big\|_{H^{k-1}(\BB^{d}_{R})} < \big\| (\Box_{{\chi\mathstrut}_{T}} v_{\lambda})(\log T, \,.\,) - F_{N} \big\|_{H^{k-1}(\BB^{d}_{R})} + \varepsilon \,.
\end{equation}
Moreover, consider smooth functions of the form
\begin{equation}
\label{DecoSpherHarm}
f(\xi) = \sum_{\ell=0}^N \sum_{m\in\Omega_{d,\ell}} f_{\ell,m}(|\xi|) Y_{\ell,m}(\tfrac{\xi}{|\xi|}) \,,
\end{equation}
notice for $(\rho,\omega)\in(0,\infty)\times\mathbb{S}^{d-1}$ the relations
\begin{align*}
\pd_{\rho} f(\rho\omega) &= \omega^i(\pd_{i} f)(\rho\omega) \,, \\
\pd_{\rho}^{2} f(\rho\omega) &= \omega^i\omega^j(\pd_{i}\pd_{j} f)(\rho\omega) \,, \\
(\Laplace f)(\rho\omega) &= \Big( \pd_{\rho}^{2} + \frac{d-1}{\rho} \pd_{\rho} + \frac{1}{\rho^{2}} \Laplace^{\mathbb{S}^{d-1}}_{\omega} \Big) f(\rho\omega) \,,
\end{align*}
and compute for the wave operator
\begin{align*}
&(\Box_{{\chi\mathstrut}_{T}} v_{\lambda})(\log T,\rho\omega) - F_{N}(\rho\omega) \\&\indent\indent=
\sum_{\ell=0}^N \sum_{m\in\Omega_{d,\ell}} \Big( a_{2}(\rho) f_{\ell,m}''(\rho) + a_{1}(\rho) f_{\ell,m}'(\rho) + a_{0}(\rho) f_{\ell,m}(\rho) - F_{\ell,m} \Big) Y_{\ell,m}(\omega)
\end{align*}
where the coefficients are given by
\begin{align*}
a_{2}(\rho) &= (1-\rho^{2}) \,, \\
a_{1}(\rho) &= 2 \Big( \frac{d-3}{2} - \lambda \Big) \rho + (d-1) \frac{1-\rho^{2}}{\rho} \,, \\
a_{0}(\rho) &= - \lambda (\lambda+1) - \frac{\ell(\ell+d-2)}{\rho^{2}} \,.
\end{align*}
Given the spherical harmonics approximation $F_{N}$ as above, it follows that there is an $f\in C^{\infty}(\overline{\BB^{d}_{R})}$ such that
\begin{equation}
\label{ModeEquation}
\renewcommand{\arraystretch}{1.2}
\left\{
\begin{array}{rcll}
(\Box_{{\chi\mathstrut}_{T}} v_{\lambda})(\log T, \,.\,) &=& F_{N} & \text{in } \BB^{d}_{R} \,, \\
v_{\lambda}(\tau,\,.\,) &=& \displaystyle{\Big( \frac{\ee^{\tau}}{T} \Big)^{\lambda} f} \,, &
\end{array}
\right.
\end{equation}
if and only if for $\ell=0,\ldots,N$ and $m\in\Omega_{d,\ell}$ there are $f_{\ell,m}\in C^{\infty}([0,R])$ such that
\begin{equation}
\label{DecoupledODE}
\begin{array}{rcll}
a_2 f_{\ell,m}'' + a_1 f_{\ell,m}' + a_0 f_{\ell,m} &=& F_{\ell,m} & \text{in } (0,R) \,.
\end{array}
\end{equation}
Each of the latter equations has finite regular singular points $\{-1,0,1\}$ with Frobenius indices
\begin{equation}
\{\ell,-(d+\ell-2)\} \text{ at } \rho = 0
\qquad\text{and}\qquad
\Big\{0,\frac{d}{2}-\lambda-\frac{1}{2} \Big\} \text{ at } \rho = 1 \,.
\end{equation}
By means of the transformation
\begin{equation*}
f_{\ell,m}(\rho) = \rho^{\ell} g_{\ell,m}(\rho^{2})
\end{equation*}
we obtain
\begin{align*}
&
a_{2}(\rho) f_{\ell,m}''(\rho) + a_{1}(\rho) f_{\ell,m}'(\rho) + a_{0}(\rho) f_{\ell,m}(\rho) \\&\indent=
4\rho^{\ell}\Big( \rho^{2}(1-\rho^{2}) g_{\ell,m}''(\rho^{2}) + (c-(a+b+1)\rho^{2}) g_{\ell,m}'(\rho^{2}) - ab g_{\ell,m}(\rho^{2}) \Big) \,,
\end{align*}
with parameters
\begin{equation*}
a = \frac{1}{2} \Big( \lambda+\ell \Big) \,,
\quad
b = \frac{1}{2} \Big( \lambda+\ell+1 \Big) \,,
\quad
c = \frac{d}{2} + \ell \,.
\end{equation*}
In order to construct a smooth solution to \Cref{DecoupledODE} we shall employ Duhamel's principle, for which we need a fundamental system of solutions to the hypergeometric differential equation
\begin{equation}
\label{HypergEq}
z(1-z) g''(z) + (c-(a+b+1)z) g'(z) - ab g(z) = 0 \,.
\end{equation}
From now on, let us fix
\begin{equation*}
\lambda = \frac{d}{2} \,,
\end{equation*}
implying
\begin{equation*}
a = \frac{1}{2} \Big( \frac{d}{2}+\ell \Big) \,,
\quad
b=a+\frac{1}{2} \,,
\quad
c = 2a \,.
\end{equation*}
Then the Frobenius indices of \Cref{HypergEq} are
\begin{equation*}
\Big\{0,1-2a\Big\} \text { at } z=0
\qquad\text{and}\qquad
\Big\{0, - \frac{1}{2} \Big\} \text { at } z=1
\end{equation*}
and a fundamental system to \Cref{HypergEq} in $\RR$ near $z=1$ is given in terms of hypergeometric functions
\begin{equation*}
\psi_{a,1}(z) = {}_{2}F{}_{1}(a,a+\tfrac{1}{2},\tfrac{3}{2};1-z) \,, \quad
\widetilde{\psi}_{a,1}(z) = |1-z|^{-\frac{1}{2}} {}_{2}F{}_{1}(a,a-\tfrac{1}{2},\tfrac{1}{2};1-z) \,,
\end{equation*}
with Wronskian
\begin{equation*}
W(\psi_{a,1},\widetilde{\psi}_{a,1})(z) = \frac{1}{2} z^{-2a} (1-z)^{-1} |1-z|^{-\frac{1}{2}} \,.
\end{equation*}
The explicit form of the Wronskian can be computed from Abel's identity. We also remark that both $\psi_{a,1}$, $\widetilde{\psi}_{a,1}$ extend analytically to solutions in $(0,\infty)$ away from $z=1$. The Frobenius indices at $z=0$ are possibly separated by a positive integer and the analytic solution at $z=0$ is given by
\begin{equation*}
\psi_{a,0}(z) = {}_{2}F{}_{1}(a,a+\tfrac{1}{2},2a;z)
\end{equation*}
and related to $\psi_{a,1},\widetilde{\psi}_{a,1}$ through the connection formula
\begin{align}
\begin{split}
\label{ConnectionFormula}
\psi_{a,0} &= \frac{\Gamma(2a)\Gamma(-\tfrac{1}{2})}{\Gamma(a)\Gamma(a-\tfrac{1}{2})} \psi_{a,1} + \frac{\Gamma(2a)\Gamma(\tfrac{1}{2})}{\Gamma(a)\Gamma(a+\tfrac{1}{2})} \widetilde{\psi}_{a,1} \\&=
-2^{2a-1}(2a-1) \psi_{a,1} + 2^{2a-1} \widetilde{\psi}_{a,1}
\end{split}
\end{align}
on $(0,1)$ and by extension also on $(1,\infty)$. This formula implies
\begin{equation*}
W(\psi_{a,0},\psi_{a,1})(z) = -2^{2a-1} W(\psi_{a,1},\widetilde{\psi}_{a,1})(z) = - 2^{2a-2} z^{-2a} (1-z)^{-1} |1-z|^{-\frac{1}{2}}
\end{equation*}
and, in particular, that $\psi_{a,0},\psi_{a,1}$ are linearly independent and therefore form a fundamental system of solutions in $(0,1)\cup(1,\infty)$. From this we obtain a fundamental system of solutions
\begin{equation}
\label{FunSys}
\phi_{d,\ell,0}(\rho) \coloneqq 2^{-\frac{1}{2}(\frac{d}{2}+\ell-1)} \rho^{\ell} \psi_{a,0}(\rho^{2}) \,, \qquad
\phi_{d,\ell,1}(\rho) \coloneqq 2^{-\frac{1}{2}(\frac{d}{2}+\ell-1)} \rho^{\ell} \psi_{a,1}(\rho^{2}) \,,
\end{equation}
to the homogeneous problem associated to \Cref{DecoupledODE} with Wronskian
\begin{equation*}
W_{d,\ell}(\rho) \coloneqq W(\phi_{d,\ell,0},\phi_{d,\ell,1})(\rho) = - \rho^{-(d-1)} (1-\rho^{2})^{-1} |1-\rho^{2}|^{-\frac{1}{2}} \,.
\end{equation*}
According to Duhamel's principle,
\begin{align}
\label{SolutionDuhamel}
f_{\ell,m}(\rho) &=
+ \phi_{d,\ell,0}(\rho) \int_{\rho}^{1} \frac{\phi_{d,\ell,1}(z)}{W_{d,\ell}(z)} \frac{F_{\ell,m}(z)}{1-z^{2}} \dd z +
\phi_{d,\ell,1}(\rho) \int_{0}^{\rho} \frac{\phi_{d,\ell,0}(z)}{W_{d,\ell}(z)} \frac{F_{\ell,m}(z)}{1-z^{2}} \dd z \\\nonumber&=-\phi_{d,\ell,0}(\rho) \int_{\rho}^{1} z^{d-1} |1-z^{2}|^{\frac{1}{2}} \phi_{d,\ell,1}(z) F_{\ell,m}(z) \dd z \\\nonumber&\indent-\phi_{d,\ell,1}(\rho) \int_{0}^{\rho} z^{d-1} |1-z^{2}|^{\frac{1}{2}} \phi_{d,\ell,0}(z) F_{\ell,m}(z) \dd z
\end{align}
is a smooth solution to the inhomogeneous problem \eqref{DecoupledODE} in $(0,1)\cup(1,R]$. Our aim is to prove that $f_{\ell,m}$ belongs to $C^{\infty}([0,R])$ which yields via \Cref{DecoSpherHarm} a function $f\in C^{\infty}(\overline{\BB^{d}_{R}})$. Note that the connection formula implies that $|1-z^{2}|^{\frac{1}{2}}\phi_{d,\ell,0}(z)$ is bounded on $(0,1)$ and thus
\begin{equation*}
c_{d,\ell,m} \coloneqq \int_{0}^{1} z^{d-1} (1-z^{2})^{\frac{1}{2}} \phi_{d,\ell,0}(z) F_{\ell,m}(z) \dd z
\end{equation*}
exists and is finite. Additionally to $\phi_{d,\ell,0}$, $\phi_{d,\ell,1}$, let
\begin{equation*}
\widetilde{\phi}_{d,\ell,1}(\rho) \coloneqq 2^{+\frac{1}{2} (\frac{d}{2}+\ell-1)} \rho^{\ell} \widetilde{\psi}_{d,\ell,1}(\rho^{2})
\end{equation*}
and note that by construction
\begin{equation*}
\widetilde{\varphi}_{d,\ell,1}(\rho) \coloneqq |1-\rho|^{\frac{1}{2}}\widetilde{\phi}_{d,\ell,1}(\rho)
\end{equation*}
defines an analytic function in a neighbourhood of $\rho = 1$. By employing the connection formula and factorizing a zero, we get
\begin{align*}
f_{\ell,m}(\rho) &= - c_{d,\ell,m} \phi_{d,\ell,1}(\rho) + \phi_{d,\ell,1}(\rho) \int_{\rho}^{1} z^{d-1} |1-z^{2}|^{\frac{1}{2}} \widetilde{\phi}_{d,\ell,1}(z) F_{\ell,m}(z) \dd z
\\&\indent -\widetilde{\phi}_{d,\ell,1}(\rho) \int_{\rho}^{1} z^{d-1} |1-z^{2}|^{\frac{1}{2}} \phi_{d,\ell,1}(z) F_{\ell,m}(z) \dd z \\&=
-c_{d,\ell,m} \phi_{d,\ell,1}(\rho) + \phi_{d,\ell,1}(\rho) \int_{\rho}^{1} z^{d-1} \sqrt{1+z} \widetilde{\varphi}_{d,\ell,1}(z) F_{\ell,m}(z) \dd z \\&\indent-
(1-\rho) \widetilde{\varphi}_{d,\ell,1}(\rho) \int_{0}^{1} \Big( w^{d-1} \sqrt{1+w} \phi_{d,\ell,1}(w) F_{\ell,m}(w) \Big)\Big|_{w=1+z(\rho-1)} \sqrt{z} \dd z \,,
\end{align*}
which shows that $f_{\ell,m}$ is smooth at $\rho = 1$ and thus $f_{\ell,m} \in C^{\infty}((0,R])$. The other regular singular point at $\rho=0$ stems from the change of variables in the spherical harmonics decomposition. Incidentally, the fundamental system \eqref{FunSys} is known in closed form for $\rho \in (0,1)$, see \cite[eqs. 15.4.9 and 15.4.18]{MR2723248}. For convenience, we make use of the explicit expression for $\phi_{d,\ell,1}$ to find near $\rho = 0$ the behaviour
\begin{align*}
\phi_{d,\ell,0}(\rho) &= \rho^{\ell} \varphi_{d,\ell,0}(\rho) \,, \\
\phi_{d,\ell,1}(\rho) &=
\renewcommand{\arraystretch}{1.2}
\left\{
\begin{array}{ll}
\rho^{-(d+\ell-2)} \varphi_{d,\ell,1}(\rho) & \text{if } (d,\ell) \neq (2,0) \,, \\
-\log(\rho) \phi_{2,0,0}(\rho) + \varphi_{2,0,1}(\rho) & \text{if } (d,\ell) = (2,0) \,,
\end{array}
\right.
\end{align*}
where $\varphi_{d,\ell,0}$ and $\varphi_{d,\ell,1}$ are analytic near $\rho = 0$ with $\varphi_{d,\ell,0}(0),\varphi_{d,\ell,1}(0)\neq 0$. Using this, we find the coarse bounds
\begin{align*}
\Big| \phi_{d,\ell,0}(\rho)\int_{\rho}^{1} z^{d-1} (1-z^{2})^{\frac{1}{2}} \phi_{d,\ell,1}(z) F_{\ell,m}(z) \dd z \Big| &\lesssim 
\renewcommand{\arraystretch}{1.2}
\left\{
\begin{array}{ll}
1 \,, & \ell = 0 \,, \\
\rho \,, & \ell = 1 \,, \\
-\rho^{2}\log(\rho) \,, & \ell = 2 \,, \\
\rho^{2} \,, & \ell \geq 3 \,,
\end{array}
\right.
\\
\Big| \phi_{d,\ell,1}(\rho)\int_{0}^{\rho} z^{d-1} (1-z^{2})^{\frac{1}{2}} \phi_{d,\ell,0}(z) F_{\ell,m}(z) \dd z \Big| &\lesssim
\renewcommand{\arraystretch}{1.2}
\left\{
\begin{array}{ll}
\rho^{2} \,, & (d,\ell) \neq (2,0) \,, \\
\rho \,, & (d,\ell) = (2,0) \,, \\
\end{array}
\right.
\\
\Big| \phi_{d,\ell,0}'(\rho)\int_{\rho}^{1} z^{d-1} (1-z^{2})^{\frac{1}{2}} \phi_{d,\ell,1}(z) F_{\ell,m}(z) \dd z \Big| &\lesssim
\renewcommand{\arraystretch}{1.2}
\left\{
\begin{array}{ll}
1 \,, & \ell = 0 \,, \\
1 \,, & \ell=1 \,, \\
-\rho\log(\rho) \,, & \ell = 2 \,, \\
\rho \,, & \ell \geq 3 \,,
\end{array}
\right.
\\
\Big| \phi_{d,\ell,1}'(\rho)\int_{0}^{\rho} z^{d-1} (1-z^{2})^{\frac{1}{2}} \phi_{d,\ell,0}(z) F_{\ell,m}(z) \dd z \Big| &\lesssim \rho \,,
\end{align*}
near $\rho=0$. This implies for the solution \eqref{DecoSpherHarm}
\begin{align*}
|f(\xi)| &\lesssim \sum_{\ell=0}^N \sum_{m\in\Omega_{d,\ell}} |f_{\ell,m}(|\xi|)| \lesssim 1 \,, \\
|(\pd f)(\xi)| &\lesssim \sum_{m\in\Omega_{d,0}} |f_{0,m}'(|\xi|)| + \sum_{\ell=1}^N \sum_{m\in\Omega_{d,\ell}} \Big( |\xi|^{-1} |f_{\ell,m}(|\xi|)| + |f_{\ell,m}'(|\xi|)| \Big) \lesssim 1 \,,
\end{align*}
near $\xi=0$. It follows that $f\in C^{\infty}(\overline{\BB^{d}_{R}}\setminus\{0\}) \cap H^1(\BB^{d}_{R})$. In particular, from \Cref{ModeEquation,EllipticProblem} we infer that $f$ is a weak solution to
\begin{equation*}
(\delta^{ij} - \xi^{i} \xi^{j}) (\pd_{i} \pd_{j} f)(\xi) - (d+2)\xi^{i} (\pd_{i} f)(\xi) - \tfrac{d(d+2)}{4} f(\xi) = F_{N}(\xi)
\end{equation*}
in $\BB^{d}_{R}$. Note that the left-hand side defines a uniformly elliptic operator on $\BB^{d}_{\frac{1}{2}}$ so that we conclude $f\in C^{\infty}(\overline{\BB^{d}_{R}})$ from elliptic regularity \cite[p. 334, Theorem 3]{MR2597943}. Thus, we have obtained a smooth solution to problem \eqref{ModeEquation} which implies with \eqref{DensityEstimate} the result.\par
\medskip
In dimension $d=1$, \Cref{EllipticProblem} reads
\begin{equation*}
(\Box_{{\chi\mathstrut}_{T}} v_{\lambda})(\log T,\xi) - F(\xi) =
(1-\xi^{2}) f''(\xi) - 2 (\lambda+1) \xi f'(\xi) - \lambda(\lambda+1) f(\xi) - F(\xi)
\end{equation*}
and it remains to construct a solution $f\in C^{\infty}([-R,R])$ to
\begin{equation}
\label{d=1ODE}
(1-\xi^{2}) f''(\xi) - 2 (\lambda+1) \xi f'(\xi) - \lambda(\lambda+1) f(\xi) = F(\xi)
\quad\text{in } (-R,R) \,.
\end{equation}
Note that this equation has Frobenius indices $\{0,-\lambda\}$ at $\xi=-1,1$. Via the transformation
\begin{equation*}
f(\xi) = g(z) \,, \qquad z = \frac{1-\xi}{2} \,,
\end{equation*}
we see that
\begin{align*}
(1-\xi^{2}) f''(\xi) &- 2 (\lambda+1) \xi f'(\xi) - \lambda(\lambda+1) f(\xi) \\&\indent=
z(1-z) g''(z) - (c - (a+b+1)z) g'(z) - ab g(z)
\end{align*}
with parameters
\begin{equation*}
a = \lambda \,, \qquad
b = \lambda + 1 \,, \qquad
c = \lambda+1 \,,
\end{equation*}
i.e., \Cref{d=1ODE} is of hypergeometric type. From this, we infer for the choice
\begin{equation*}
\lambda = \frac{1}{2}
\end{equation*}
a fundamental system to the homogeneous problem associated to \Cref{d=1ODE} given in closed form by
\begin{equation*}
\phi_{-1}(\xi) = \frac{1}{\sqrt{|1-\xi|}} \,, \qquad
\phi_{1}(\xi) = \frac{1}{\sqrt{|1+\xi|}} \,,
\end{equation*}
with Wronskian
\begin{equation*}
W(\xi) = W(\phi_{-1},\phi_{1})(\xi) = - \frac{1}{(1-\xi^{2}) \sqrt{|1-\xi^{2}|}} \,.
\end{equation*}
Consequently, a particular solution to \Cref{d=1ODE} is given by
\begin{align*}
f(\xi) &= \phi_{-1}(\xi) \int_{\xi}^{1} \frac{\phi_{1}(z)}{W(z)} \frac{F(z)}{1-z^{2}} \dd z +
\phi_{1}(\xi) \int_{-1}^{\xi} \frac{\phi_{-1}(z)}{W(z)} \frac{F(z)}{1-z^{2}} \dd z \\&=
-\frac{1}{\sqrt{|1-\xi|}} \int_{\xi}^{1} \sqrt{|1-z|} F(z) \dd z - \frac{1}{\sqrt{|1+\xi|}} \int_{-1}^{\xi} \sqrt{|1+z|} F(z) \dd z \\&=
-(1-\xi) \int_{0}^{1} F(1+z(\xi-1)) \sqrt{z} \dd z - (1+\xi) \int_{0}^{1} F(-1+z(\xi+1)) \sqrt{z} \dd z
\end{align*}
and clearly $f\in C^{\infty}([-R,R])$.
\end{proof}
This technical lemma furnishes the link to the following density property.
\begin{proposition}
\label{DenseRange}
Let $(d,p,k,R)\in\NN\times\RR_{>1}\times\NN\times\RR_{\geq 1}$. For all $\varepsilon>0$ and all $\mathbf{F}\in C^{\infty}(\overline{\BB^{d}_{R}})^{2}$ there is an $\mathbf{f}\in C^{\infty}(\overline{\BB^{d}_{R}})^{2}$ such that
\begin{equation*}
\big\| \big( (\tfrac{d}{2}-s_{p}) \mathbf{I} - \mathbf{L}_{d,p} \big)\mathbf{f} - \mathbf{F} \big\|_{H^{k}(\BB^{d}_{R})\times H^{k-1}(\BB^{d}_{R})} < \varepsilon \,.
\end{equation*}
\end{proposition}
\begin{proof}
Let $\mathbf{F}\in C^{\infty}(\overline{\BB^{d}_{R}})^{2}$. Let $\lambda\in\CC$, $f_{1}\in C^{\infty}(\overline{\BB^{d}_{R}})$ and define $\mathbf{f}_{\lambda,p}\in C^{\infty}(\overline{\BB^{d}_{R}})^{2}$ by
\begin{equation*}
\mathbf{f}_{\lambda,p}(\xi)=
\begin{bmatrix}
f_{1}(\xi)\\
(\lambda+s_{p})f_{1}(\xi) + \xi^{i} (\pd_{i} f_{1})(\xi) - F_1(\xi)
\end{bmatrix}
\,.
\end{equation*}
A computation in terms of the variable
\begin{align*}
\mathbf{v}_{\lambda+s_{p}}(\tau,\xi) &\coloneqq \Big(\frac{\ee^\tau}{T} \Big)^{\lambda+s_{p}} \mathbf{f}_{\lambda,p}(\xi) \\&=
\begin{bmatrix}
v_{\lambda+s_{p},1}(\tau,\xi) \\
\pd_{\tau} v_{\lambda+s_{p},1}(\tau,\xi) + \xi^{i} \pd_{\xi^{i}} v_{\lambda+s_{p},1}(\tau,\xi) - \Big(\frac{\ee^\tau}{T} \Big)^{\lambda+s_{p}} F_1(\xi)
\end{bmatrix}
\end{align*}
reveals
\begin{align}
\label{SpectralWave}
(\lambda\mathbf{I} - \mathbf{L}_{d,p})\mathbf{f}_{\lambda,p}-\mathbf{F} &=
\left.\Big( \pd_{\tau} - s_{p} \mathbf{I} - \mathbf{L}_{d,p} \Big) \mathbf{v}_{\lambda+s_{p}}(\tau,\,.\,)\right|_{\tau=\log T} - \mathbf{F} \\
\label{SpectralBox}
&=-
\begin{bmatrix}
0 \\
( \Box_{{\chi\mathstrut}_{T}} v_{\lambda+s_{p},1})(\log T,\,.\,) + F_{\lambda,p}
\end{bmatrix}
\,,
\end{align}
where $F_{\lambda,p}\in C^{\infty}(\overline{\BB^{d}_{R}})$ is given by
\begin{align*}
F_{\lambda,p}(\xi) &= (\lambda+s_{p}+1) F_{1}(\xi) + \xi^{i} (\pd_{i} F_{1})(\xi) + F_{2}(\xi) \,.
\end{align*}
Now by \Cref{EllipticApproximation} we can choose $f_{1}\in C^{\infty}(\overline{\BB^{d}_{R}})$ so that we get via \Cref{SpectralBox} the assertion.
\end{proof}
\subsection{Generation of semigroups}
\label{SecFreeSG}
The map from \Cref{WaveFlowOperation} is a densely defined operator in Sobolev spaces.
\begin{definition}
\label{WaveFlowOperator}
Let $(d,p,k,R)\in\NN\times\RR_{>1}\times\NN\times\RR_{\geq 1}$. The operator $\mathbf{L}_{d,p,k,R}: \dom(\mathbf{L}_{d,p,k}) \subset H^{k}(\BB^{d}_{R})\times H^{k-1}(\BB^{d}_{R}) \rightarrow H^{k}(\BB^{d}_{R})\times H^{k-1}(\BB^{d}_{R})$ is densely defined by
\begin{equation*}
\mathbf{L}_{d,p,k,R}\mathbf{f} = \mathbf{L}_{d,p}\mathbf{f} \,,
\qquad
\dom(\mathbf{L}_{d,p,k,R}) = C^{\infty}(\overline{\BB^{d}_{R}})^{2} \,.
\end{equation*}
\end{definition}
We obtain in terms of strongly continuous semigroups a precise notion for the free wave flow in extended past light cones in each space dimension and for all scalings.
\begin{theorem}
\label{FreeSemigroupTHM}
Let $(d,p,k,R)\in\NN\times\RR_{>1}\times\NN\times\RR_{\geq 1}$. The operator $\mathbf{L}_{d,p,k,R}$ is closable and its closure $\overline{\mathbf{L}_{d,p,k,R}}$ is the generator of a strongly continuous operator semigroup
\begin{equation*}
\mathbf{S}_{d,p,k,R}: \RR_{\geq 0} \rightarrow \mathfrak{L} \left( H^{k}(\BB^{d}_{R})\times H^{k-1}(\BB^{d}_{R}) \right)
\end{equation*}
which satisfies that for any $0<\varepsilon<\frac{1}{2}$ there is a constant $M_{d,p,k,R,\varepsilon}\geq 1$ such that
\begin{equation*}
\| \mathbf{S}_{d,p,k,R}(\tau) \mathbf{f} \|_{H^{k}(\BB^{d}_{R})\times H^{k-1}(\BB^{d}_{R})} \leq M_{d,p,k,R,\varepsilon} \ee^{\omega_{d,p,k,\varepsilon} \tau} \| \mathbf{f} \|_{H^{k}(\BB^{d}_{R})\times H^{k-1}(\BB^{d}_{R})}
\end{equation*}
for all $\mathbf{f}\in H^{k}(\BB^{d}_{R})\times H^{k-1}(\BB^{d}_{R})$ and all $\tau\geq 0$, where
\begin{equation*}
\omega_{d,p,k,\varepsilon} = -s_{p} + \max\Big\{ \frac{d}{2}-k,\varepsilon \Big\}
\qquad\text{and}\qquad
s_{p} = \frac{2}{p-1} \,.
\end{equation*}
\end{theorem}
\begin{proof}
The operator $\mathbf{L}_{d,p,k,R}$ is linear and densely defined in $H^{k}(\BB^{d}_{R})\times H^{k-1}(\BB^{d}_{R})$. We fix $\varepsilon_{1} = \varepsilon$ in \Cref{HigherEnergyDissipativity} and get that $\mathbf{L}_{d,p,k,R}-\omega_{d,p,k,\varepsilon}\mathbf{I}$ is dissipative with respect to the inner product given in \Cref{InnerProducts}. The range of $\lambda\mathbf{I} - ( \mathbf{L}_{d,p,k,R}-\omega_{d,p,k,\varepsilon}\mathbf{I} )$ is dense in $H^{k}(\BB^{d}_{R})\times H^{k-1}(\BB^{d}_{R})$ for $\lambda=\frac{d}{2}-\max\left\{\frac{d}{2}-k,\varepsilon\right\} > 0$ by \Cref{DenseRange}. Hence, we infer from the Lumer-Phillips Theorem \cite[p. 83, Theorem 3.15]{MR1721989} the semigroup. Since by \Cref{HigherEnergyNorm} the induced norm from the inner product is equivalent to the classical Sobolev norm on $H^{k}(\BB^{d}_{R})\times H^{k-1}(\BB^{d}_{R})$, the exponential growth bound follows.
\end{proof}
\begin{remark}
If $k\geq \frac{d}{2}$ is an integer, the growth estimates for the semigroup $\mathbf{S}_{d,p,k,R}$ from \Cref{FreeSemigroupTHM} yield the growth bound $\omega(\mathbf{S}_{d,p,k,R}) = -s_{p}$, see \cite[p. 40]{MR1721989} for a definition. Indeed, just note that $\lambda = -s_{p}$ is an eigenvalue of $\mathbf{L}_{d,p,k,R}$ to the constant function $\mathbf{f} = [1,0] \in\dom(\mathbf{L}_{d,p,k,R})$ for each $(d,p,k,R)\in\NN\times\RR_{>1}\times\NN\times\RR_{\geq 1}$.
\end{remark}
\section{The linearized wave evolution}
\label{IntroLinEVO}
In this section, we treat the linearized wave operator $-\pd_{t}^{2} + \Laplace_{x} + \mathcal{V}_{\beta,T}(t,x)$ with potential
\begin{equation*}
\mathcal{V}_{\beta,T}(t,x) = (s_{p}+1) (s_{p}+2) \left( 1-\beta^{\top}\beta \right) \left( T-t + \beta^{\top}x \right)^{-2} \,,
\end{equation*}
as it arises in problem \eqref{CauchyProblem2}. The transition from the classical formulation of this equation to our first order formalism in similarity coordinates ${\chi\mathstrut}_{T}:[0,\infty)\times\RR^{d}\rightarrow\RR^{1,d}$ is based on relation \eqref{TransitionMechanism} and given in terms of the variable
\begin{equation}
\label{CauchyVariable}
\mathbf{u}(\tau,\,.\,)=
\begin{bmatrix}
u_{1}(\tau,\,.\,) \\
u_{2}(\tau,\,.\,)
\end{bmatrix}
=
\begin{bmatrix}
\hfill(T\ee^{-\tau})^{s_{p}} (u\circ{\chi\mathstrut}_{T})(\tau,\,.\,) \\
(T\ee^{-\tau})^{s_{p}+1} (\pd_{0}u\circ{\chi\mathstrut}_{T})(\tau,\,.\,)
\end{bmatrix}
\,,
\quad
u\in C^{\infty}\big((0,T)\times\RR^{d}\big) \,,
\end{equation}
by the relation
\begin{equation}
\label{LinWaveRel}
\pd_\tau
\mathbf{u}(\tau,\,.\,)
=
\mathbf{L}_{d,p}
\mathbf{u}(\tau,\,.\,)
+\mathbf{L}_{\beta}'\mathbf{u}(\tau,\,.\,)
-
\begin{bmatrix}
0 \\
(T\ee^{-\tau})^{s_{p}+2}\Big(\big( \Box u + \mathcal{V}_{\beta,T}u \big) \circ {\chi\mathstrut}_{T} \Big)(\tau,\,.\,)
\end{bmatrix}
\,,
\end{equation}
where
\begin{equation}
\label{PotentialIntro}
\mathbf{L}_{\beta}'\mathbf{u}(\tau,\,.\,)(\xi) = \begin{bmatrix}
0 \\
(s_{p}+1) (s_{p}+2) \left( 1-\beta^{\top}\beta \right) \left( 1 + \beta^{\top}\xi \right)^{-2} u_{1}(\tau,\xi)
\end{bmatrix}
\,.
\end{equation}
\Cref{LinWaveRel} yields a description of the linearized wave flow near the ODE blowup. We are interested in the stability of the underlying evolution which is determined by spectral properties of the perturbed generator. In the following subsections, we specify this within the functional analytic setting outlined in \cref{SecFreeSG}.
\subsection{Functional setting}
\Cref{FreeSemigroupTHM} provides a natural choice for function spaces to frame the linear theory.
\begin{definition}
\label{SobolevSpaces}
Let $(d,k,R)\in\NN\times\NN\times\RR_{>0}$. We define a Hilbert space by
\begin{equation*}
\mathfrak{H}^{k}(\BB^{d}_{R}) \coloneqq H^{k}(\BB^{d}_{R}) \times H^{k-1}(\BB^{d}_{R}) \,, \qquad
\| \mathbf{f} \|_{\mathfrak{H}^{k}(\BB^{d}_{R})} \coloneqq \| (f_{1},f_{2}) \|_{H^{k}(\BB^{d}_{R}) \times H^{k-1}(\BB^{d}_{R})} \,.
\end{equation*}
We also define for $\delta>0$ a closed ball
\begin{equation*}
\mathfrak{H}^{k}_{\delta}(\BB^{d}_{R}) \coloneqq \{ \mathbf{f} \in \mathfrak{H}^{k}(\BB^{d}_{R}) \mid \| \mathbf{f} \|_{\mathfrak{H}^{k}(\BB^{d}_{R})} \leq \delta \} \,.
\end{equation*}
\end{definition}
In these spaces, we employ a unique bounded linear extension to introduce the potential \eqref{PotentialIntro} that arises from the linearization around the blowup.
\begin{definition}
\label{Potential}
Let $(d,p,k,R)\in \NN\times\RR_{>1}\times\NN\times\RR_{\geq 1}$. Let $R_{0} > R$ and $\beta\in\overline{\BB^{d}_{R_{0}^{-1}}}$. Let $V_{\beta}\in C^{\infty}(\overline{\BB^{d}_{R}})$ be given by
\begin{equation*}
V_{\beta}(\xi) = (s_{p}+1) (s_{p}+2) \left( 1-\beta^{\top}\beta \right) \left( 1 + \beta^{\top}\xi \right)^{-2} \,.
\end{equation*}
We define the bounded linear operator
\begin{equation*}
\mathbf{L}_{\beta}' \in \mathfrak{L} \left( \mathfrak{H}^{k}(\BB^{d}_{R}) \right)
\qquad\text{by}\qquad
\mathbf{L}_{\beta}'\mathbf{f} =
\begin{bmatrix}
0 \\
V_{\beta} f_{1}
\end{bmatrix}
\,.
\end{equation*}
\end{definition}
Now, the linearized equation is composed of a perturbation of the free equation.
\begin{definition}
\label{LinearWaveEvolutionOperator}
Let $(d,p,k,R)\in \NN\times\RR_{>1}\times\NN\times\RR_{\geq 1}$. Let $R_{0} > R$ and $\beta\in\overline{\BB^{d}_{R_{0}^{-1}}}$. The closed linear operator $\mathbf{L}_{\beta}: \dom(\mathbf{L}_{\beta})\subset\mathfrak{H}^{k}(\BB^{d}_{R})\rightarrow\mathfrak{H}^{k}(\BB^{d}_{R})$ is densely defined by
\begin{equation*}
\mathbf{L}_{\beta} = \overline{\mathbf{L}_{d,p,k,R}} + \mathbf{L}_{\beta}' \,, \qquad \dom(\mathbf{L}_{\beta}) = \dom(\overline{\mathbf{L}_{d,p,k,R}}) \,.
\end{equation*}
\end{definition}
Perturbation theory for semigroups yields that this operator is indeed the generator of the linearized wave flow.
\begin{proposition}
\label{SGbeta}
Let $(d,p,k,R)\in \NN\times\RR_{>1}\times\NN\times\RR_{\geq 1}$. Let $R_{0} > R$ and $\beta\in\overline{\BB^{d}_{R_{0}^{-1}}}$. The closed linear operator $\mathbf{L}_{\beta} :\dom(\mathbf{L}_{\beta}) \subset \mathfrak{H}^{k}(\BB^{d}_{R}) \rightarrow \mathfrak{H}^{k}(\BB^{d}_{R})$ is the generator of a strongly continuous operator semigroup
\begin{equation*}
\mathbf{S}_{\beta}: \RR_{\geq 0} \rightarrow \mathfrak{L}\left(\mathfrak{H}^{k}(\BB^{d}_{R})\right) \,.
\end{equation*}
Moreover, let $0<\varepsilon<\frac{1}{2}$ and $M_{d,p,k,R,\varepsilon}\geq 1$, $\omega_{d,p,k,\varepsilon}\in\RR$ be the constants from \Cref{FreeSemigroupTHM}. Then the bound
\begin{equation*}
\| \mathbf{S}_{\beta}(\tau) \|_{\mathfrak{L}\left(\mathfrak{H}^{k}(\BB^{d}_{R})\right)} \leq M_{d,p,k,R,\varepsilon} \ee^{ \Big( \omega_{d,p,k,\varepsilon} + M_{d,p,k,R,\varepsilon} \| \mathbf{L}_{\beta}' \|_{\mathfrak{L}\left(\mathfrak{H}^{k}(\BB^{d}_{R})\right)} \Big) \tau}
\end{equation*}
holds for all $\tau\geq 0$ and all $\beta\in\overline{\BB^{d}_{R_{0}^{-1}}}$.
\end{proposition}
\begin{proof}
The existence of a semigroup with the asserted generator and growth estimate is a direct consequence of the bounded perturbation theorem \cite[p. 158]{MR1721989} applied to the semigroup from \Cref{FreeSemigroupTHM}.
\end{proof}
However, this growth estimate does not account for a desired stable evolution. It is therefore indispensable to understand the stable and unstable components of the linearized wave flow. This will be achieved by characterizing the spectrum of the generator. We approach this perturbatively and, in a first step, investigate the nature of the perturbation.
\subsection{Properties of the potential}
Let us record the main features of the potential, i.e., its compactness and Lipschitz continuous dependence on the Lorentz parameter.
\begin{lemma}
\label{PotentialProperties}
Let $(d,p,k,R)\in \NN\times\RR_{>1}\times\NN\times\RR_{\geq 1}$. Let $R_{0} > R$ and $\beta\in\overline{\BB^{d}_{R_{0}^{-1}}}$. The operator $\mathbf{L}_{\beta}'\in\mathfrak{L}\left(\mathfrak{H}^{k}(\BB^{d}_{R})\right)$ is compact. Furthermore, there is a constant $L>0$ such that the bound
\begin{equation*}
\| \mathbf{L}_{\beta_{1}}' - \mathbf{L}_{\beta_{2}}' \|_{\mathfrak{L}\left(\mathfrak{H}^{k}(\BB^{d}_{R})\right)} \leq L |\beta_{1}-\beta_{2}|
\end{equation*}
holds for all $\beta_{1},\beta_{2}\in\overline{\BB^{d}_{R_{0}^{-1}}}$.
\end{lemma}
\begin{proof}
The operator $\mathbf{L}_{\beta}'$ is given by the composition of bounded linear operators
\begin{equation*}
\renewcommand{\arraystretch}{1.2}
\begin{array}{rcccccl}
\mathfrak{H}^{k}(\BB^{d}_{R}) & \rightarrow & H^{k}(\BB^{d}_{R}) & \hookrightarrow & H^{k-1}(\BB^{d}_{R}) & \rightarrow & \mathfrak{H}^{k}(\BB^{d}_{R}) \\
\begin{bmatrix}
f_{1} \\
f_{2}
\end{bmatrix}
&\mapsto& f_{1}  &\mapsto& f_{1}  &\mapsto&
\begin{bmatrix}
0 \\
V_{\beta} f_{1}
\end{bmatrix}
\,.
\end{array}
\end{equation*}
Since the embedding $H^{k}(\BB^{d}_{R})  \hookrightarrow H^{k-1}(\BB^{d}_{R})$ is compact for each $d,k\in\NN$, the operator $\mathbf{L}_{\beta}'$ is also compact.\par
\medskip
The fundamental theorem of calculus yields
\begin{equation*}
V_{\beta_{1}}(\xi) - V_{\beta_{2}}(\xi) = ( \beta_{1}^{i} - \beta_{2}^{i} ) \int_{0}^{1} \left.\pd_{{\beta'}^{i}} V_{\beta'}(\xi,\beta) \right|_{\beta'= \beta_{2} + z(\beta_{1} - \beta_{2})} \dd z
\end{equation*}
and since the map
\begin{equation*}
\BB^{d}_{R}\times\BB^{d}_{R_{0}^{-1}} \rightarrow \RR \,, \qquad (\xi,\beta) \mapsto V_{\beta}(\xi) \,,
\end{equation*}
belongs to $C^{\infty}(\overline{\BB^{d}_{R}\times\BB^{d}_{R_{0}^{-1}}})$, the bound follows.
\end{proof}
\subsection{Spectral analysis of the generator}
In fact, the evolution generated in \Cref{SGbeta} cannot be stable in the whole Hilbert space due to two unstable eigenvalues of the generator which are induced by the time translation symmetry and Lorentz symmetry of \Cref{NLWIntro}.
\begin{lemma}
\label{EV}
Let $(d,p,k,R)\in \NN\times\RR_{>1}\times\NN\times\RR_{\geq 1}$. Let $R_{0} > R$ and $\beta\in\overline{\BB^{d}_{R_{0}^{-1}}}$. Let $\mathbf{f}_{0,\beta,i},\mathbf{f}_{1,\beta}\in C^{\infty}(\overline{\BB^{d}_{R}})^{2}$ be given by
\begin{align}
\label{EV0beta}
\mathbf{f}_{0,\beta,i}(\xi) &=
\begin{bmatrix}
\hfill s_{p} c_{p} \gamma(\beta)^{-s_{p}} (1+\beta^{\top}\xi)^{-s_{p}-1} \xi_{i} \\
s_{p} (s_{p}+1) c_{p} \gamma(\beta)^{-s_{p}} (1+\beta^{\top}\xi)^{-s_{p}-2} \xi_{i}
\end{bmatrix}
\\&\nonumber\,+
\begin{bmatrix}
\hfill s_{p} c_{p} \gamma(\beta)^{-s_{p}+2} (1+\beta^{\top}\xi)^{-s_{p}} \beta_{i} \\
s_{p}^{2} c_{p} \gamma(\beta)^{-s_{p}} (1+\beta^{\top}\xi)^{-s_{p}-1} \beta_{i}
\end{bmatrix}
\,, \\
\label{EV1beta}
\mathbf{f}_{1,\beta}(\xi) &=
\begin{bmatrix}
\hfill s_{p} c_{p} \gamma(\beta)^{-s_{p}} (1+\beta^{\top}\xi)^{-s_{p}-1} \\
s_{p} (s_{p}+1) c_{p} \gamma(\beta)^{-s_{p}} (1+\beta^{\top}\xi)^{-s_{p}-2}
\end{bmatrix}
\,,
\end{align}
for $i=1,\ldots,d$. Then
\begin{equation*}
\mathbf{L}_{\beta} \mathbf{f}_{0,\beta,i} = \mathbf{0}
\qquad\text{and}\qquad
(\mathbf{I}-\mathbf{L}_{\beta}) \mathbf{f}_{1,\beta} = \mathbf{0} \,.
\end{equation*}
\end{lemma}
\begin{proof}
Recall that
\begin{equation*}
\psi_{\beta,T}^{*}(t,x) = c_{p}\gamma(\beta)^{-s_{p}}(T-t+\beta^{\top}x)^{-s_{p}}
\end{equation*}
defines a $(d+1)$-parameter family of blowup solutions satisfying
\begin{equation*}
\Box \psi_{\beta,T}^{*} = - \psi_{\beta,T}^{*} |\psi_{\beta,T}^{*}|^{p-1} \,.
\end{equation*}
Differentiating this equation with respect to the parameters $\beta^{i},T$ for $i=1,\ldots,d$ implies
\begin{equation*}
\Box \pd_{\beta^{i}} \psi_{\beta,T}^{*} = - \mathcal{V}_{\beta,T} \pd_{\beta^{i}} \psi_{\beta,T}^{*}
\qquad\text{and}\qquad
\Box \pd_{T} \psi_{\beta,T}^{*} = - \mathcal{V}_{\beta,T} \pd_{T} \psi_{\beta,T}^{*} \,.
\end{equation*}
Here,
\begin{align*}
\pd_{\beta^{i}} \psi_{\beta,T}^{*}(t,x) &= - c_{p}s_{p} \gamma(\beta)^{-s_{p}}(T-t+\beta^{\top}x)^{-s_{p}-1} x_{i} \\&\indent
- c_{p}s_{p} \gamma(\beta)^{-s_{p}+2}(T-t+\beta^{\top}x)^{-s_{p}} \beta_{i} \,, \\
\pd_{t} \pd_{\beta^{i}} \psi_{\beta,T}^{*}(t,x) &= - c_{p}s_{p}(s_{p}+1) \gamma(\beta)^{-s_{p}}(T-t+\beta^{\top}x)^{-s_{p}-2} x_{i} \\&\indent
- c_{p}s_{p}^{2} \gamma(\beta)^{-s_{p}+2}(T-t+\beta^{\top}x)^{-s_{p}-1} \beta_{i} \,,
\end{align*}
and
\begin{align*}
\pd_{T} \psi_{\beta,T}^{*}(t,x) &= - s_{p}c_{p}\gamma(\beta)^{-s_{p}}(T-t+\beta^{\top}x)^{-s_{p}-1} \,, \\
\pd_{t} \pd_{T} \psi_{\beta,T}^{*}(t,x) &= - s_{p}(s_{p}+1)c_{p}\gamma(\beta)^{-s_{p}}(T-t+\beta^{\top}x)^{-s_{p}-2} \,,
\end{align*}
and an application of relation \eqref{LinWaveRel} yields the above stated eigenfunctions.
\end{proof}
In the following, we prove that the two symmetry-induced eigenvalues are the only unstable spectral points of the generator. As a result, we will formulate a spectral theorem for the generator that allows us to identify the stable and unstable components of the evolution.
\subsubsection{Spectral analysis with parameter value zero}
If $\beta=0$, a description of the unstable spectrum of the generator follows from ODE analysis.
\begin{lemma}
\label{SpecL0}
Let $(d,p,k,R)\in \NN \times\RR_{>1}\times\NN\times\RR_{\geq 1}$ such that $\frac{d}{2}-s_{p}-k < 0$. Put
\begin{equation*}
\omega_{d,p,k} \coloneqq \max\left\{-1,\frac{d}{2}-s_{p}-k,-s_{p} \right\}
\end{equation*}
and let $\omega_{d,p,k}<\omega_{0}<0$ be arbitrary but fixed. We have
\begin{equation*}
\sigma(\mathbf{L}_{0}) \cap \overline{\mathbb{H}_{\omega_{0}}} = \{0,1\} \,.
\end{equation*}
Both points $0,1\in\sigma(\mathbf{L}_{0})$ belong to the point spectrum of $\mathbf{L}_{0}$ with eigenspaces
\begin{equation*}
\ker(\mathbf{L}_{0}) = \langle \mathbf{f}_{0,0,i} \rangle_{i=1}^{d}
\qquad\text{and}\qquad
\ker(\mathbf{I}-\mathbf{L}_{0}) = \langle \mathbf{f}_{1,0} \rangle
\end{equation*}
spanned by the symmetry modes $\mathbf{f}_{0,0,i},\mathbf{f}_{1,0}\in C^{\infty}(\overline{\BB^{d}_{R}})^{2}$ given by
\begin{equation*}
\mathbf{f}_{0,0,i}(\xi) =
\begin{bmatrix}
s_{p}c_{p}\xi_{i} \\
s_{p}(s_{p}+1)c_{p} \xi_{i}
\end{bmatrix}
\qquad\text{and}\qquad
\mathbf{f}_{1,0}(\xi) =
\begin{bmatrix}
s_{p}c_{p} \\
s_{p}(s_{p}+1)c_{p}
\end{bmatrix}
\,,\qquad
i=1,\ldots,d.
\end{equation*}
\end{lemma}
\begin{proof}
Due to \Cref{FreeSemigroupTHM} and \Cref{PotentialProperties}, we can employ \cite[Theorem B.1]{MR4469070} to get that the set $\sigma(\mathbf{L}_{0}) \cap \overline{\mathbb{H}_{\omega_{0}}}$ consists of finitely many isolated eigenvalues of $\mathbf{L}_{0}$. Thus, if $\lambda\in\sigma(\mathbf{L}_{0}) \cap \overline{\mathbb{H}_{\omega_{0}}}$ then there exists an $\mathbf{f}_{\lambda}\in\dom(\mathbf{L}_{0})\setminus\{\mathbf{0}\}\subset \mathfrak{H}^{k}(\BB^{d}_{R})$ such that $(\lambda\mathbf{I}-\mathbf{L}_{0})\mathbf{f}_{\lambda} = \mathbf{0}$. From this equation we conclude that the first component $f_{\lambda,1}\in H^{k}(\BB^{d}_{R})$ is a non-zero weak solution to
\begin{align}
\nonumber
0 &= \Big( (\delta^{ij}-\xi^{i}\xi^{j})\pd_{\xi^{i}}\pd_{\xi^{j}} - 2(\lambda+s_{p}+1)\xi^{i}\pd_{\xi^{i}} - (\lambda+s_{p})(\lambda+s_{p}+1) \Big) f_{\lambda,1}(\xi) \\\nonumber&\indent+ V_{0}(\xi) f_{\lambda,1}(\xi) \\\label{SpectralEquation}&=
\Big( (\delta^{ij}-\xi^{i}\xi^{j})\pd_{\xi^{i}}\pd_{\xi^{j}} - 2(\lambda+s_{p}+1)\xi^{i}\pd_{\xi^{i}} - (\lambda-1)(\lambda+2s_{p}+2) \Big) f_{\lambda,1}(\xi)
\end{align}
in $\BB^{d}_{R}$, where $V_{0}(\xi)=(s_{p}+1)(s_{p}+2)$ is the constant potential. In what follows, we treat the multidimensional cases separately from the one-dimensional case.\par
\medskip
In the multidimensional case $d\geq 2$, elliptic regularity implies $f_{\lambda,1}\in C^{\infty}(\BB^{d}_{R}\setminus\mathbb{S}^{d-1})\cap H^{k}(\BB^{d}_{R})$. Next, we consider the spherical harmonics expansion
\begin{equation*}
f_{\lambda,1}(\rho\omega) = \sum_{\ell=0}^{\infty} \sum_{m\in\Omega_{d,\ell}} f_{\lambda,\ell,m}(\rho) Y_{\ell,m}(\omega) \,, \quad f_{\lambda,\ell,m}(\rho) = \int_{\mathbb{S}^{d-1}} \overline{Y_{\ell,m}(\omega)} f_{\lambda,1}(\rho\omega) \dd\sigma(\omega) \,.
\end{equation*}
From this, the bounds
\begin{align*}
\| f_{\lambda,\ell,m}^{(j)} \|_{L^{2}((\frac{1}{2},1))}^{2} &= \int_{\frac{1}{2}}^{1} \left| \int_{\mathbb{S}^{d-1}} \overline{Y_{\ell,m}(\omega)} \omega^{i_{1}} \ldots \omega^{i_{j}} (\pd_{i_{1}} \ldots \pd_{i_{j}} f_{\lambda,1})(\rho\omega) \dd\sigma(\omega) \right|^{2} \dd\rho \\&\lesssim
\int_{\frac{1}{2}}^{1} \int_{\mathbb{S}^{d-1}} \sum_{|\alpha|=j} |(\pd^{\alpha} f_{\lambda,1})(\rho\omega)|^{2} \dd\sigma(\omega) \dd\rho \\&\simeq
\sum_{|\alpha|=j} \int_{\frac{1}{2}}^{1} \int_{\mathbb{S}^{d-1}} |(\pd^{\alpha} f_{\lambda,1})(\rho\omega)|^{2} \dd\sigma(\omega) \rho^{d-1} \dd\rho \\&\lesssim
\| f_{\lambda,1} \|_{H^{k}(\BB^{d}_{R})}^{2}
\end{align*}
follow for $j=0,1,\ldots,k$ and so $f_{\lambda,\ell,m} \in C^{\infty}([0,R)\setminus\{1\})\cap H^{k}((\frac{1}{2},R))$. Moreover, inserting the spherical harmonics expansion into \Cref{SpectralEquation} yields that $f_{\lambda,\ell,m}$ is a solution to the differential equation
\begin{align*}
0 =
(1-\rho^{2}) f''(\rho) &- \Big( 2(\lambda+s_{p}+1)\rho - \frac{d-1}{\rho} \Big) f'(\rho) \\&- \Big( (\lambda-1)(\lambda+2(s_{p}+1)) + \frac{\ell(\ell+d-2)}{\rho^{2}} \Big) f(\rho)
\end{align*}
in $(0,R)$ for each $\ell\in\NN_{0}$, $m\in\Omega_{d,\ell}$. This equation has Frobenius indices
\begin{equation*}
\{\ell,-(d+\ell-2)\} \text{ at } \rho = 0
\qquad\text{and}\qquad
\Big\{0,\frac{d}{2}-s_{p}-\lambda-\frac{1}{2} \Big\} \text{ at } \rho = 1 \,.
\end{equation*}
To proceed, let $g_{\lambda,\ell,m}\in C^{\infty}((0,R^{2})\setminus\{1\})\cap H^{k}((\frac{1}{4},R^{2}))$ be defined through the transformation
\begin{equation}
\label{SpectralTransformation}
f_{\lambda,\ell,m}(\rho) = \rho^{\ell} g_{\lambda,\ell,m}(\rho^{2}) \,.
\end{equation}
Then $g_{\lambda,\ell,m}$ is a solution to the hypergeometric differential equation
\begin{equation}
\label{HypGeomEq}
z(1-z) g''(z) + ( c - (a+b+1)z) g'(z) - ab g(z) = 0
\end{equation}
with parameters
\begin{equation*}
a = \frac{1}{2}(\ell+\lambda-1) \,, \quad
b = \frac{1}{2}(\ell+\lambda + 2(s_{p}+1)) \,, \quad
c = \frac{d}{2}+\ell \,.
\end{equation*}
The Frobenius indices of \Cref{HypGeomEq} are
\begin{equation*}
\Big\{0,-\frac{d-2}{2}-\ell\Big\} \quad \text{at } z = 0
\qquad\text{and}\qquad
\Big\{0,-\frac{1}{2} + \frac{d}{2}-s_{p}-\lambda\Big\} \text{ at } z = 1
\end{equation*}
and a fundamental system near $z=0$ is known in terms of hypergeometric functions. Namely, for each $d\geq 2$ and $\ell\geq 0$ the analytic fundamental solution is given by ${}_{2}F_{1}(a,b,c;z)$ and the second fundamental solution by
\begin{equation*}
\renewcommand{\arraystretch}{1.2}
\begin{array}{ll}
z^{-\frac{d-2}{2}-\ell} {}_{2}F_{1}(a-c+1,b-c+1,2-c;z) & \text{if } d \text{ is odd,} \\
{}_{2}F_{1}(a,b,c;z) \log z + z^{-\frac{d-2}{2}-\ell} \varphi_{a,b,c}(z) & \text{if } d \text{ is even,}
\end{array}
\end{equation*}
see \cite[p. 395]{MR2723248} for detailed formulas. This forces together with \Cref{SpectralTransformation} and smoothness of $f_{\lambda,\ell,m}$ at the origin that $g_{\lambda,\ell,m}$ is a non-zero multiple of ${}_{2}F_{1}(a,b,c;z)$ and it remains to investigate its behaviour near $z=1$.
\begin{enumerate}[wide,itemsep=1em,topsep=1em]
\item[\textit{Case $-a\notin\NN_{0}$.}] We use the identity
\begin{equation*}
\pd_{z}^{k} {}_{2}F_{1}(a,b,c;z) = \frac{(a)_{k} (b)_{k}}{(c)_{k}} {}_{2}F_{1}(a+k,b+k,c+k;z) \,,
\end{equation*}
see \cite[p. 387, Eq. (15.5.4)]{MR2723248}. Since $-a\notin\NN_{0}$ and $\Re(b)>0$, we have for the rising factorials $(a)_{k}, (b)_{k} \neq 0$. Now, as $\Re(\lambda) \geq \omega_{0} > \frac{d}{2} - s_{p} - k$,
\begin{equation*}
\Re(c-a-b-k) = - \frac{1}{2} + \Big( \frac{d}{2} - s_{p} - k \Big) - \Re(\lambda) < - \frac{1}{2} \,,
\end{equation*}
so \cite[p. 387, Eq. (15.4.23)]{MR2723248} implies
\begin{equation*}
| \pd_{z}^{k} {}_{2}F_{1}(a,b,c;z)| \simeq |1-z|^{\Re(c-a-b-k)}
\end{equation*}
near $z=1$. Hence $\pd_{z}^{k} {}_{2}F_{1}(a,b,c;z)$ fails to be square integrable near $z=1$ which excludes all possible $\lambda\in\overline{\mathbb{H}_{\omega_{0}}}$ in the present case as an eigenvalue of $\mathbf{L}_{0}$.
\item[\textit{Case $-a\in\NN_{0}$.}] Then there is some $n\in\NN_{0}$ such that $a=-n$, i.e.,
\begin{equation*}
\lambda = -2n-\ell+1 \,.
\end{equation*}
As $\Re(\lambda) \geq \omega_{0} > -1$ and $\ell\geq 0$, we conclude $2(1-n) > \ell \geq 0$. However, this can only hold for $n=0$, which leads to the only admissible pairs $(\lambda,\ell) \in \{ (1,0),(0,1) \}$. By \Cref{EV}, the symmetry modes $\mathbf{f}_{0,0,i}$ and $\mathbf{f}_{1,0}$ are eigenvectors of $\mathbf{L}_{0}$ to the eigenvalues $\lambda = 0,1$, respectively, which occupy the spherical moments $\ell = 1$, $m\in \Omega_{d,1}$ and $\ell = 0$, $m\in \Omega_{d,0}$, respectively. Since $\dim\Omega_{d,0} = 1$ and $\dim\Omega_{d,1} = d$, see \cite[p. 19]{MR2934227}, there are no further linearly independent eigenvectors.
\end{enumerate}
\par
In case $d=1$, we have that $f_{\lambda,1} \in H^{k}((-R,R))$ solves the equation
\begin{equation*}
(1-\xi^{2})f_{\lambda,1}''(\xi) - 2(\lambda+s_{p}+1)\xi f_{\lambda,1}'(\xi) - (\lambda-1)(\lambda+2s_{p}+2) f_{\lambda,1}(\xi) = 0
\end{equation*}
in $(-R,R)$, which has regular singular points at $\xi = \pm 1$ with Frobenius indices $\{0,-\lambda-s_{p}\}$, respectively. Upon performing the transformation
\begin{equation*}
f_{\lambda,1}(\xi) = g_{\lambda,1}\big( \tfrac{1-\xi}{2} \big) \,, 
\end{equation*}
we find that $g_{\lambda,1}$ solves the hypergeometric differential equation
\begin{equation}
\label{d=1HypGeomEq}
z(1-z) g''(z) + (c - (a+b+1)z) g'(z) - ab g(z) = 0
\end{equation}
with parameters
\begin{equation*}
a = \lambda - 1 \,, \qquad
b = \lambda + 2s_{p} + 2 \,, \qquad
c = \lambda + s_{p} + 1 \,.
\end{equation*}
A fundamental system to \Cref{d=1HypGeomEq} in $(0,1)$ is given by the hypergeometric function ${}_{2}F_{1}(a,b,c;z)$, which is analytic at $z=0$, and
\begin{equation*}
\renewcommand{\arraystretch}{1.2}
\begin{array}{ll}
z^{-\lambda-s_{p}} {}_{2}F_{1}(a-c+1,b-c+1,2-c;z) & \text{if } c \notin \NN, \\
{}_{2}F_{1}(a,b,c;z) \log z + z^{-\lambda-s_{p}} \varphi_{a,b,c}(z) & \text{if } c\in\NN.
\end{array}
\end{equation*}
However, any of the latter solutions fails to be in $H^{1}((0,1))$ for any value of $\lambda\in\overline{\mathbb{H}_{\omega_{0}}}$ and thus $g_{\lambda,1}$ is a multiple of ${}_{2}F_{1}(a,b,c;z)$ in $(0,1)$. Similarly, another fundamental system in $(0,1)$ is given by ${}_{2}F_{1}(a,b,a+b+1-c;1-z)$, which is analytic at $z=1$, and a second function which fails to be in $H^{1}((0,1))$. Thus, $g_{\lambda,1}$ is also a multiple of ${}_{2}F_{1}(a,b,a+b+1-c;1-z)$ in $(0,1)$ and therefore the analytic solutions at $z=0$ and $z=1$ have to be linearly dependent. According to the connection formula relating both hypergeometric functions, this is only possible if $-a\in\NN_{0}$. As above, this yields the eigenvalues $\lambda\in\{0,1\}$ which correspond to the eigenmodes from \Cref{EV}.
\end{proof}
In order to identify a stable subspace for the linearized wave evolution, we also need information about the algebraic multiplicity of the unstable eigenvalues.
\begin{lemma}
\label{MAMG}
Let $(d,p,k,R)\in \NN\times\RR_{>1}\times\NN\times\RR_{\geq 1}$ such that $\frac{d}{2}-s_{p}-k < 0$. The algebraic multiplicity of both eigenvalues $\lambda\in\{0,1\}$ of $\mathbf{L}_{0}$ is finite and equal to the respective geometric multiplicity. In particular,
\begin{equation*}
\ran(\mathbf{P}_{\lambda,0}) = \ker(\lambda\mathbf{I}-\mathbf{L}_{0}) \,,
\end{equation*}
where $\mathbf{P}_{\lambda,0}\in \mathfrak{L}\left( \mathfrak{H}^{k}(\BB^{d}_{R})\right)$, defined by
\begin{equation*}
\mathbf{P}_{\lambda,0} \coloneqq \frac{1}{2\pi\ii}\int_{\pd\mathbb{D}_{r}(\lambda)} \mathbf{R}_{\mathbf{L}_{0}}(z) \dd z \,,
\end{equation*}
is the \emph{Riesz projection} associated to the isolated eigenvalue $\lambda\in\{0,1\}$ of $\mathbf{L}_{0}$, respectively.
\end{lemma}
\begin{proof}
Let $\lambda\in\{0,1\}$. Assume $\mathbf{f}\in\ker(\lambda\mathbf{I}-\mathbf{L}_{0})$. Then
\begin{equation*}
(\lambda\mathbf{I}-\mathbf{L}_{0})(\mathbf{I}-\mathbf{P}_{\lambda,0})\mathbf{f} = (\mathbf{I}-\mathbf{P}_{\lambda,0})(\lambda\mathbf{I}-\mathbf{L}_{0})\mathbf{f} = \mathbf{0} \,.
\end{equation*}
So $(\mathbf{I}-\mathbf{P}_{\lambda,0})\mathbf{f} \in \ker(\lambda\mathbf{I} - \mathbf{L}_{0} \restriction_{\ran(\mathbf{I}-\mathbf{P}_{\lambda,0})})$. But by construction $\sigma(\mathbf{L}_{0} \restriction_{\ran(\mathbf{I}-\mathbf{P}_{\lambda,0})}) = \sigma(\mathbf{L}_{0})\setminus\{\lambda\}$, hence $(\mathbf{I}-\mathbf{P}_{\lambda,0}) \mathbf{f} = \mathbf{0}$. Thus $\mathbf{f} = \mathbf{P}_{\lambda,0} \mathbf{f} \in \ran(\mathbf{P}_{\lambda,0})$ and we have
\begin{equation*}
\ker(\lambda\mathbf{I}-\mathbf{L}_{0}) \subseteq \ran(\mathbf{P}_{\lambda,0}) \,.
\end{equation*}
\par
To prove the other inclusion, we start noting that the algebraic multiplicity of both eigenvalues $\lambda\in\{0,1\}$ is finite, i.e., $\ran(\mathbf{P}_{\lambda,0}) \subset \mathfrak{H}^{k}(\BB^{d}_{R})$ is a finite-dimensional subspace, see \cite[Theorem B.1]{MR4469070}. By construction, $\sigma(\mathbf{L}_{0} \restriction_{\ran(\mathbf{P}_{\lambda,0})})=\{\lambda\}$ which implies that the operator $\lambda\mathbf{I}-\mathbf{L}_{0} \restriction_{\ran(\mathbf{P}_{\lambda,0})}$ is nilpotent, i.e., there is a minimal $n_{\lambda}\in\NN$ such that $(\lambda\mathbf{I}-\mathbf{L}_{0} \restriction_{\ran(\mathbf{P}_{\lambda,0})})^{n_{\lambda}} = \mathbf{0}$.
\par
\medskip
First, assume that $n_{\lambda}>1$. Then $\ran(\lambda\mathbf{I}-\mathbf{L}_{0} \restriction_{\ran(\mathbf{P}_{\lambda,0})}) \cap \ker(\lambda\mathbf{I}-\mathbf{L}_{0} \restriction_{\ran(\mathbf{P}_{\lambda,0})}) \neq \{\mathbf{0}\}$ and it follows that there exist $\mathbf{F}_{\lambda}\in\ker(\lambda\mathbf{I}-\mathbf{L}_{0}\restriction_{\ran(\mathbf{P}_{\lambda,0})}) \setminus \{\mathbf{0}\}$ and $\mathbf{f}_{\lambda}\in\ran(\mathbf{P}_{\lambda,0})$ such that
\begin{equation*}
(\lambda\mathbf{I}-\mathbf{L}_{0} \restriction_{\ran(\mathbf{P}_{\lambda,0})})\mathbf{f}_{\lambda} = \mathbf{F}_{\lambda} \,.
\end{equation*}
This yields the equation
\begin{equation}
\label{lambdaell}
\Big( -(\delta^{ij}-\xi^{i}\xi^{j})\pd_{\xi^{i}}\pd_{\xi^{j}} + 2(\lambda+s_{p})\xi^{i}\pd_{\xi^{i}} + (\lambda-1)(\lambda+2(s_{p}+1)) \Big) f_{\lambda,1}(\xi) = F_{\lambda}(\xi) \,,
\end{equation}
where $F_{\lambda}(\xi) = (\lambda+s_{p}+1)F_{\lambda,1}(\xi) + \xi^{i}\pd_{\xi^{i}} F_{\lambda,1}(\xi) + F_{\lambda,2}(\xi)$ defines a smooth inhomogeneity. Elliptic regularity implies $f_{\lambda,1}\in C^{\infty}(\overline{\BB^{d}_{R}}\setminus\mathbb{S}^{d-1})\cap H^{k}(\BB^{d}_{R})$. To continue, we treat both eigenvalues $\lambda\in\{0,1\}$ separately.
\begin{enumerate}[wide,topsep=1em,itemsep=1em]
\item[\textit{Eigenvalue $\lambda=0$.}] There is an $\alpha\in\CC^{d}\setminus\{0\}$ such that the inhomogeneity reads
\begin{equation*}
\mathbf{F}_{0} = \sum_{i=1}^{d} \alpha_{i}\mathbf{f}_{0,0,i}
\qquad\text{and thus}\qquad
F_{0}(\xi) = (2s_{p}+3) \sum_{i=1}^{d} \alpha_{i}\xi^{i} \,,
\end{equation*}
see \Cref{SpecL0}. From the same lemma we see that only the first spherical moment contributes in the spherical harmonics expansion of $f_{0,1}$, so
\begin{equation*}
f_{0,1}(\rho\omega) = \sum_{m\in\Omega_{d,1}} f_{0,1,m}(\rho) Y_{1,m}(\omega) \,.
\end{equation*}
Plugging this relation into \Cref{lambdaell} leaves us with inspecting the solution $f$ of the equation
\begin{equation*}
\Big( (1-\rho^{2}) \pd_{\rho}^{2} - \Big( 2(s_{p}+1)\rho - \frac{d-1}{\rho} \Big) \pd_{\rho} + \Big( 2(s_{p}+1) - \frac{d-1}{\rho^{2}} \Big) \Big) f(\rho) = (2s_{p}+3) \rho \,.
\end{equation*}
By letting
\begin{equation*}
f(\rho) = \frac{2s_{p}+3}{2(d+2)} \rho g(\rho^{2})
\end{equation*}
we obtain the equation
\begin{equation*}
z(1-z)g''(z) + \Big( \frac{d+2}{2} - \Big(s_{p}+\frac{5}{2}\Big)z \Big) g'(z) = \frac{d+2}{2} \,.
\end{equation*}
Thus, there is a constant $c\in\CC$ such that
\begin{equation*}
g'(z) = c z^{-\frac{d+2}{2}} (1-z)^{\frac{d}{2}-s_{p}-\frac{3}{2}} + \tfrac{d+2}{2} z^{-\frac{d+2}{2}} (1-z)^{\frac{d}{2}-s_{p}-\frac{3}{2}} \int_{0}^{z} \zeta^{\frac{d}{2}} (1-\zeta)^{-\left( \frac{d}{2}-s_{p}-\frac{1}{2} \right)} \dd\zeta \,.
\end{equation*}
As $f_{0,1}\in C^{\infty}([0,R]\setminus\{1\}) \cap H^{k}((\frac{1}{2},R))$, it follows that $c=0$, so
\begin{align*}
g'(z) &= \tfrac{d+2}{2} z^{-\frac{d+2}{2}} (1-z)^{\frac{d}{2}-s_{p}-\frac{3}{2}} \int_{0}^{z} \zeta^{\frac{d}{2}} (1-\zeta)^{-\left( \frac{d}{2}-s_{p}-\frac{1}{2} \right)} \dd\zeta \\&= {}_{2}F_{1}(s_{p}+\tfrac{5}{2},1,\tfrac{d+4}{2};z) \,,
\end{align*}
where we used \cite[p. 183, Eq. 8.17.8]{MR2723248}. This implies
\begin{equation*}
\Big| (1-z)^{-(\frac{d}{2}-s_{p}-k-\frac{1}{2})} g^{(k)}(z) \Big| \simeq 1
\end{equation*}
near $z=1$, which contradicts $g^{(k)}\in L^{2}(\frac{1}{4},R^{2})$.
\item[\textit{Eigenvalue $\lambda=1$.}] There is an $\alpha\in\CC\setminus\{0\}$ such that the inhomogeneity reads $\mathbf{F}_{1} = \alpha \mathbf{f}_{1,0}$ and thus $F_{1}(\xi) = \alpha(2s_{p} + 3)$ is constant, see \Cref{SpecL0}. The spherical harmonics expansion in the proof of the same lemma shows that $f_{1,1}$ is actually radial, i.e.
\begin{equation*}
f_{1,1}(\rho\omega) = f_{1,0,0}(\rho) Y_{0,0}(\omega) \,.
\end{equation*}
With this, we infer from \Cref{lambdaell} that a multiple of $f_{1,0,0}$ solves the equation
\begin{equation*}
\Big( (1-\rho^{2}) \pd_{\rho}^{2} - \Big( 2(s_{p}+2)\rho - \frac{d-1}{\rho} \Big) \pd_{\rho} \Big) f(\rho) = \alpha (2s_{p}+3) \,.
\end{equation*}
For convenience, we consider the transformation
\begin{equation*}
f(\rho) = \frac{\alpha (2s_{p}+3)}{2d} g(\rho^{2}) \,,
\end{equation*}
which leads to the equation
\begin{equation*}
z(1-z)g''(z) + \Big( \frac{d}{2} - \Big(s_{p}+\frac{5}{2}\Big) \Big) g'(z) = \frac{d}{2} \,.
\end{equation*}
Again, this equation can be integrated and it follows that there exists a constant $c\in\CC$ such that
\begin{equation*}
g'(z) = c z^{-\frac{d}{2}}(1-z)^{\frac{d}{2}-s_{p}-\frac{5}{2}} + \tfrac{d}{2} z^{-\frac{d}{2}}(1-z)^{\frac{d}{2}-s_{p}-\frac{5}{2}} \int_{0}^{z} \zeta^{\frac{d-2}{2}} (1-\zeta)^{\frac{d}{2}-s_{p}-\frac{3}{2}} \dd\zeta \,.
\end{equation*}
As $f_{1,0,0}$ is smooth near $z=0$ we conclude $c=0$ and so
\begin{align*}
g'(z) &= \tfrac{d}{2} z^{-\frac{d}{2}}(1-z)^{\frac{d}{2}-s_{p}-\frac{5}{2}} \int_{0}^{z} \zeta^{\frac{d-2}{2}} (1-\zeta)^{\frac{d}{2}-s_{p}-\frac{3}{2}} \dd\zeta \\&=
{}_{2}F_{1}(s_{p}+\tfrac{5}{2},1,\tfrac{d}{2};z) \,.
\end{align*}
Thus,
\begin{equation*}
\Big| (1-z)^{-(\frac{d}{2}-s_{p}-k-\frac{5}{2})} g^{(k)}(z) \Big| \simeq 1
\end{equation*}
near $z=1$, which contradicts $g^{(k)}\in L^{2}(\frac{1}{4},R^{2})$.
\end{enumerate}
\par
It follows that the assumption $n_{\lambda}>1$ is wrong.
\par
\medskip
Thus $n_{\lambda}=1$. So $\lambda\mathbf{I}-\mathbf{L}_{0} \restriction_{\ran(\mathbf{P}_{\lambda,0})} = \mathbf{0}$, i.e. $(\lambda\mathbf{I}-\mathbf{L}_{0})\mathbf{P}_{\lambda,0}\mathbf{f} = \mathbf{0}$ for all $\mathbf{f}\in\mathfrak{H}^{k}(\BB^{d}_{R})$ and thus we infer the other inclusion $ \ran(\mathbf{P}_{\lambda,0})\subseteq\ker(\lambda\mathbf{I}-\mathbf{L}_{0})$.
\end{proof}
\subsubsection{Spectral analysis with small parameter values}
We show that the spectrum of the generator in a right half-plane is stable under small variations of the Lorentz parameter.
\begin{proposition}
\label{SpectralImplication}
Let $(d,p,k,R)\in \NN\times\RR_{>1}\times\NN\times\RR_{\geq 1}$ such that $\frac{d}{2}-s_{p}-k < 0$. Put
\begin{equation*}
\omega_{d,p,k} \coloneqq \max\left\{-1,\frac{d}{2}-s_{p}-k,-s_{p} \right\}
\end{equation*}
and let $\omega_{d,p,k}<\omega_{0}<0$ be arbitrary but fixed. There is an $R_{0} > R$ such that for each $\lambda\in\overline{\mathbb{H}_{\omega_{0}}}\setminus \Big( \mathbb{D}_{\frac{|\omega_{0}|}{2}}(0) \cup \mathbb{D}_{\frac{|\omega_{0}|}{2}}(1) \Big)$ and $\beta\in\overline{\BB^{d}_{R_{0}^{-1}}}$
\begin{equation*}
\lambda \in \sigma(\mathbf{L}_{\beta})
\qquad\text{implies}\qquad
\lambda \in \sigma(\mathbf{L}_{0}) \,.
\end{equation*}
In particular, $\overline{\mathbb{H}_{\omega_{0}}}\setminus \big( \mathbb{D}_{\frac{|\omega_{0}|}{2}}(0) \cup \mathbb{D}_{\frac{|\omega_{0}|}{2}}(1) \big) \subset \varrho(\mathbf{L}_{\beta})$ and there is a constant $C>0$ such that
\begin{equation*}
\| \mathbf{R}_{\mathbf{L}_{\beta}}(z) \|_{\mathfrak{L}\left(\mathfrak{H}^{k}(\BB^{d}_{R})\right)} \leq C
\end{equation*}
holds for all $z\in\overline{\mathbb{H}_{\omega_{0}}}\setminus \Big( \mathbb{D}_{\frac{|\omega_{0}|}{2}}(0) \cup \mathbb{D}_{\frac{|\omega_{0}|}{2}}(1) \Big)$ and all $\beta\in\overline{\BB^{d}_{R_{0}^{-1}}}$.
\end{proposition}
\begin{proof}
If $z\in\overline{\mathbb{H}_{\omega_{0}}}\setminus\{0,1\}$ then $z\in \varrho(\mathbf{L}_{0})$ and the identity
\begin{equation}
\label{Birman}
z\mathbf{I} - \mathbf{L}_{\beta} = \big( \mathbf{I} - ( \mathbf{L}_{\beta}' - \mathbf{L}_{0}' ) \mathbf{R}_{\mathbf{L}_{0}}(z) \big) \big( z\mathbf{I} - \mathbf{L}_{0} \big)
\end{equation}
holds. Thus $z\mathbf{I}-\mathbf{L}_{\beta}$ is boundedly invertible if and only if the bounded linear operator $\mathbf{I}-(\mathbf{L}_{\beta}'-\mathbf{L}_{0}')\mathbf{R}_{\mathbf{L}_{0}}(z)$ is invertible. The latter is true if the Neumann series of $(\mathbf{L}_{\beta}'-\mathbf{L}_{0}')\mathbf{R}_{\mathbf{L}_{0}}(z)$ converges. To achieve this, consider the Riesz projections $\mathbf{P}_{0,0}, \mathbf{P}_{0,1}$ from \Cref{MAMG} and split
\begin{equation}
\label{ResSplit}
\mathbf{R}_{\mathbf{L}_{0}}(z) \mathbf{f} = \mathbf{R}_{\mathbf{L}_{0}}(z) ( \mathbf{I} - \mathbf{P}_{0} ) \mathbf{f} + \mathbf{R}_{\mathbf{L}_{0}}(z) \mathbf{P}_{0} \mathbf{f}
\end{equation}
with $\mathbf{P}_{0} \coloneqq \mathbf{P}_{0,0} + \mathbf{P}_{0,1}$. If we combine \cite[Theorem B.1]{MR4469070} with \cite[p. 55, 1.10 Theorem]{MR1721989}, we get for each $\omega>\omega_{d,p,k}$ that $\mathbf{R}_{\mathbf{L}_{0}}(z)(\mathbf{I}-\mathbf{P}_{0})$ defines an analytic map in the whole right half-plane $\mathbb{H}_{\omega}$ and there is a constant $M_{\omega}\geq 1$ such that
\begin{equation}
\label{ResSplitStableBound}
\| \mathbf{R}_{\mathbf{L}_{0}}(z) (\mathbf{I}-\mathbf{P}_{0}) \|_{\mathfrak{L}\left(\mathfrak{H}^{k}(\BB^{d}_{R})\right)} \leq \frac{M_{\omega}}{\Re(z)-\omega} \,,
\end{equation}
for all $z\in \mathbb{H}_{\omega}$. On the other hand, using \Cref{SpecL0,MAMG}, there are unique $\mathbf{g}^{1}_{0},\ldots,\mathbf{g}^{d}_{0},\mathbf{g}_{1}\in\mathfrak{H}^{k}(\BB^{d}_{R})$ such that
\begin{equation*}
\mathbf{P}_{0}\mathbf{f} = \sum_{i=1}^{d} \Big( \mathbf{g}^{i}_{0} \,\Big|\, \mathbf{f} \Big)_{\mathfrak{H}^{k}(\BB^{d}_{R})} \mathbf{f}_{0,0,i} + \Big( \mathbf{g}_{1} \,\Big|\, \mathbf{f} \Big)_{\mathfrak{H}^{k}(\BB^{d}_{R})} \mathbf{f}_{1,0} \,.
\end{equation*}
Since this is an expansion in terms of eigenvectors of $\mathbf{L}_{0}$,
\begin{equation}
\label{ResSplitUnstable}
\mathbf{R}_{\mathbf{L}_{0}}(z) \mathbf{P}_{0} \mathbf{f} = \sum_{i=1}^{d} z^{-1} \Big( \mathbf{g}^{i}_{0} \,\Big|\, \mathbf{f} \Big)_{\mathfrak{H}^{k}(\BB^{d}_{R})} \mathbf{f}_{0,0,i} + (z-1)^{-1} \Big( \mathbf{g}_{1} \,\Big|\, \mathbf{f} \Big)_{\mathfrak{H}^{k}(\BB^{d}_{R})} \mathbf{f}_{1,0}
\end{equation}
follows. From \Cref{ResSplit,ResSplitStableBound,ResSplitUnstable} we infer that there is a constant $C_{0}>0$ such that the resolvent estimate
\begin{equation*}
\| \mathbf{R}_{\mathbf{L}_{0}}(z) \|_{\mathfrak{L}\left(\mathfrak{H}^{k}(\BB^{d}_{R})\right)} \leq C_{0}
\end{equation*}
holds for all $z\in\overline{\mathbb{H}_{\omega_{0}}}\setminus \Big( \mathbb{D}_{\frac{|\omega_{0}|}{2}}(0) \cup \mathbb{D}_{\frac{|\omega_{0}|}{2}}(1) \Big)$. This bound implies with \Cref{PotentialProperties}
\begin{align*}
\|(\mathbf{L}_{0}'-\mathbf{L}_{\beta}')\mathbf{R}_{\mathbf{L}_{0}}(z)\|_{\mathfrak{L}\left(\mathfrak{H}^{k}(\BB^{d}_{R})\right)} &\leq
\|(\mathbf{L}_{0}'-\mathbf{L}_{\beta}')\|_{\mathfrak{L}\left(\mathfrak{H}^{k}(\BB^{d}_{R})\right)}
\| \mathbf{R}_{\mathbf{L}_{0}}(z) \|_{\mathfrak{L}\left(\mathfrak{H}^{k}(\BB^{d}_{R})\right)} \\&\leq
L C_{0} |\beta|
\end{align*}
for all $z\in\overline{\mathbb{H}_{\omega_{0}}}\setminus \Big( \mathbb{D}_{\frac{|\omega_{0}|}{2}}(0) \cup \mathbb{D}_{\frac{|\omega_{0}|}{2}}(1) \Big)$ and all $\beta\in\overline{\BB^{d}_{R_{0}^{-1}}}$. Consequently, if we fix $R_{0}>R$ large enough, then we have for all $z\in\overline{\mathbb{H}_{\omega_{0}}}\setminus \Big( \mathbb{D}_{\frac{|\omega_{0}|}{2}}(0) \cup \mathbb{D}_{\frac{|\omega_{0}|}{2}}(1) \Big)$ and all $\beta\in\overline{\BB^{d}_{R_{0}^{-1}}}$ that $z\in\varrho(\mathbf{L}_{0})$ implies $z\in\varrho(\mathbf{L}_{\beta})$ with
\begin{equation*}
\mathbf{R}_{\mathbf{L}_{\beta}}(z) = \mathbf{R}_{\mathbf{L}_{0}}(z) \big( \mathbf{I} - (\mathbf{L}_{\beta}'-\mathbf{L}_{0}') \mathbf{R}_{\mathbf{L}_{0}}(z) \big)^{-1} \,.
\end{equation*}
From here, the uniform resolvent bound follows.
\end{proof}
Arrived at this point, the fact that the generator depends continuously on the Lorentz parameter lets us conclude this subsection with the anticipated spectral theorem for the generator.
\begin{theorem}
\label{SpecThmLbeta}
Let $(d,p,k,R)\in \NN\times\RR_{>1}\times\NN\times\RR_{\geq 1}$ such that $\frac{d}{2}-s_{p}-k < 0$. Put
\begin{equation*}
\omega_{d,p,k} \coloneqq \max\left\{-1,\frac{d}{2}-s_{p}-k,-s_{p} \right\}
\end{equation*}
and let $\omega_{d,p,k}<\omega_{0}<0$ be arbitrary but fixed. There is an $R_{0} > R$ such that for each $\beta\in\overline{\BB^{d}_{R_{0}^{-1}}}$
\begin{equation*}
\sigma(\mathbf{L}_{\beta}) \cap \overline{\mathbb{H}_{\omega_{0}}} = \{0,1\} \,.
\end{equation*}
Moreover, both points $0,1\in\sigma(\mathbf{L}_{\beta})$ belong to the point spectrum of $\mathbf{L}_{\beta}$ with eigenspaces
\begin{equation*}
\ker(\mathbf{L}_{\beta}) = \langle \mathbf{f}_{0,\beta,i} \rangle_{i=1}^{d}
\qquad\text{and}\qquad
\ker(\mathbf{I}-\mathbf{L}_{\beta}) = \langle \mathbf{f}_{1,\beta} \rangle
\end{equation*}
spanned by the symmetry modes $\mathbf{f}_{0,\beta,i},\mathbf{f}_{1,\beta}\in C^{\infty}(\overline{\BB^{d}_{R}})^{2}$ from \Cref{EV}. Furthermore, the algebraic multiplicity of both eigenvalues $\lambda\in\{0,1\}$ is finite and equal to the respective geometric multiplicity. In particular,
\begin{equation*}
\ran(\mathbf{P}_{\lambda,\beta}) = \ker(\lambda\mathbf{I}-\mathbf{L}_{\beta})
\end{equation*}
where $\mathbf{P}_{\lambda,\beta} \in \mathfrak{L}\left( \mathfrak{H}^{k}(\BB^{d}_{R})\right)$, defined by
\begin{equation*}
\mathbf{P}_{\lambda,\beta} \coloneqq \frac{1}{2\pi\ii}\int_{\pd\mathbb{D}_{r}(\lambda)} \mathbf{R}_{\mathbf{L}_{\beta}}(z) \dd z \,,
\end{equation*}
is the \emph{Riesz projection} associated to the isolated eigenvalue $\lambda\in\{0,1\}$ of $\mathbf{L}_{\beta}$, respectively.
\end{theorem}
\begin{proof}
According to \Cref{SpectralImplication}, $\pd( \mathbb{D}_{1+|\omega_{0}|} \setminus \mathbb{H}_{\omega_{0}} ) \subset \varrho(\mathbf{L}_{\beta})$ and the projection
\begin{equation*}
\mathbf{P}_{\beta} = \int_{\pd( \mathbb{D}_{1+|\omega_{0}|} \setminus \mathbb{H}_{\omega_{0}} )} \mathbf{R}_{\mathbf{L}_{\beta}}(z) \dd z \in \mathfrak{L}\left( \mathfrak{H}^{k}(\BB^{d}_{R})\right)
\end{equation*}
is well-defined and depends continuously on the parameter. By \cite[p. 34, Lemma 4.10]{MR1335452}, the ranges of $\mathbf{P}_{\beta}$ define a family of isomorphic subspaces. By \Cref{MAMG}, the range of $\mathbf{P}_{0}=\mathbf{P}_{0,0}+\mathbf{P}_{0,1}$ is $(d+1)$-dimensional. Furthermore, as $\ker(\lambda\mathbf{I}-\mathbf{L}_{\beta}) \subseteq \ran(\mathbf{P}_{\lambda,\beta})$, the range of $\mathbf{P}_{\beta}$ contains the $(d+1)$-dimensional subspace spanned by the symmetry modes from \Cref{EV}. Consequently, there are no further spectral points of $\mathbf{L}_{\beta}$ contained in $\mathbb{D}_{1+|\omega_{0}|} \setminus \mathbb{H}_{\omega_{0}}$ other than $0,1$.
\end{proof}
\subsection{Linearized flow along the stable and unstable subspace}
Now, we use the spectral information from the previous subsection to determine the instabilities as well as a stable subspace in the linearized wave evolution. Under the assumptions of \Cref{SpecThmLbeta}, we set $\mathbf{P}_{\beta} \coloneqq \mathbf{P}_{0,\beta} + \mathbf{P}_{1,\beta} \in \mathfrak{L}\left( \mathfrak{H}^{k}(\BB^{d}_{R})\right)$ and get the decomposition 
\begin{equation*}
\mathfrak{H}^{k}(\BB^{d}_{R}) = \ran(\mathbf{P}_{0,\beta})\oplus\ran(\mathbf{P}_{1,\beta})\oplus\ran(\mathbf{I}-\mathbf{P}_{\beta})
\end{equation*}
of the Hilbert space along the Riesz projections. As an immediate consequence of the spectral analysis, we will see that $\ran(\mathbf{P}_{0,\beta}), \ran(\mathbf{P}_{1,\beta}) \subset \mathfrak{H}^{k}(\BB^{d}_{R})$ are unstable subspaces in the evolution of $\mathbf{S}_{\beta}(\tau)$. Proving that the complementary subspace $\ran(\mathbf{I} - \mathbf{P}_{\beta}) \subset \mathfrak{H}^{k}(\BB^{d}_{R})$ is stable uniformly with respect to the Lorentz parameter $\beta$ is nontrivial and requires preparation which is carried out in \Cref{AppendixA}. This gives a complete dynamical portrait of the linearized wave flow near the blowup as follows.
\begin{theorem}
\label{stableEvo}
Let $(d,p,k,R)\in \NN\times\RR_{>1}\times\NN\times\RR_{\geq 1}$ such that $\frac{d}{2}-s_{p}-k < 0$. The semigroup $\mathbf{S}_{\beta}(\tau)$, introduced in \Cref{SGbeta}, and the projections $\mathbf{P}_{0,\beta},\mathbf{P}_{1,\beta} \in \mathfrak{L}\left(\mathfrak{H}^{k}(\BB^{d}_{R})\right)$, introduced in \Cref{SpecThmLbeta}, satisfy
\begin{align*}
&&
\mathbf{P}_{0,\beta}\mathbf{S}_{\beta}(\tau) &= \mathbf{S}_{\beta}(\tau)\mathbf{P}_{0,\beta} = \mathbf{P}_{0,\beta} \,,
&
\mathbf{P}_{0,\beta}\mathbf{P}_{1,\beta} &= \mathbf{0} \,, & \\
&&
\mathbf{P}_{1,\beta}\mathbf{S}_{\beta}(\tau) &= \mathbf{S}_{\beta}(\tau)\mathbf{P}_{1,\beta} = \ee^{\tau}\mathbf{P}_{1,\beta} \,,
&
\mathbf{P}_{1,\beta}\mathbf{P}_{0,\beta} &= \mathbf{0} \,. &
\end{align*}
Moreover, for each arbitrary but fixed $\omega_{d,p,k}<\omega_{0}<0$ there is a constant $M_{\omega_{0}}\geq 1$ such that
\begin{equation*}
\| \mathbf{S}_{\beta}(\tau) ( \mathbf{I} - \mathbf{P}_{\beta} ) \|_{\mathfrak{L}\left(\mathfrak{H}^{k}(\BB^{d}_{R})\right)} \leq M_{\omega_{0}} \ee^{\omega_{0}\tau} \| \mathbf{I} - \mathbf{P}_{\beta} \|_{\mathfrak{L}\left(\mathfrak{H}^{k}(\BB^{d}_{R})\right)}
\end{equation*}
for all $\tau\geq 0$ and all $\beta\in\overline{\BB^{d}_{R_{0}^{-1}}}$.
\end{theorem}
\begin{proof}
Let $R_{0} > R$ as in \Cref{SpecThmLbeta}. Any strongly continuous semigroup commutes with its generator and hence also with the resolvent of its generator. This implies that $\mathbf{S}_{\beta}(\tau)$ commutes with the spectral projections $\mathbf{P}_{\lambda,\beta}$. Using $\ker(\lambda\mathbf{I}-\mathbf{L}_{\beta}) = \ran(\mathbf{P}_{\lambda,\beta})$ from \Cref{SpecThmLbeta}, we derive the differential equation
\begin{equation*}
\pd_{\tau} \mathbf{S}_{\beta}(\tau) \mathbf{P}_{\lambda,\beta} \mathbf{f} = \lambda \mathbf{S}_{\beta}(\tau) \mathbf{P}_{\lambda,\beta} \mathbf{f}
\end{equation*}
and get the unique solution $\mathbf{S}_{\beta}(\tau) \mathbf{P}_{\lambda,\beta} \mathbf{f} = \ee^{\lambda \tau} \mathbf{P}_{\lambda,\beta} \mathbf{f}$, which yields the representation of the semigroup on the unstable subspace.\par
\medskip
Lastly, we verify a uniform resolvent bound. By \Cref{SpectralImplication} we have
\begin{equation*}
\| \mathbf{R}_{\mathbf{L}_{\beta}}(z) \|_{\mathfrak{L}\left(\mathfrak{H}^{k}(\BB^{d}_{R})\right)} \leq C
\end{equation*}
for all $z\in\overline{\mathbb{H}_{\omega_{0}}}\setminus \Big( \mathbb{D}_{\frac{|\omega_{0}|}{2}}(0) \cup \mathbb{D}_{\frac{|\omega_{0}|}{2}}(1) \Big)$ and all $\beta\in\overline{\BB^{d}_{R_{0}^{-1}}}$. Now, the map $\mathbf{R}_{\mathbf{L}_{\beta}}(z) (\mathbf{I}-\mathbf{P}_{\beta})$ is analytic in $\mathbb{H}_{\omega_{0}}$ and since
\begin{equation*}
\mathbf{R}_{\mathbf{L}_{\beta} \restriction_{\mathfrak{H}_{\beta}}} (z) = \mathbf{R}_{\mathbf{L}_{\beta}} (z) \restriction_{\mathfrak{H}_{\beta}} = \mathbf{R}_{\mathbf{L}_{\beta}}(z) (\mathbf{I}-\mathbf{P}_{\beta}) \restriction_{\mathfrak{H}_{\beta}}
\end{equation*}
we conclude that there is a $C>0$ such that
\begin{equation*}
\| \mathbf{R}_{\mathbf{L}_{\beta}}(z) (\mathbf{I}-\mathbf{P}_{\beta}) \|_{\mathfrak{L}\left(\mathfrak{H}^{k}(\BB^{d}_{R})\right)} \leq C
\end{equation*}
for all $z\in\overline{\mathbb{H}_{\omega_{0}}}$ and all $\beta\in\overline{\BB^{d}_{R_{0}^{-1}}}$ so that \Cref{GearhartPruessGreinerUniform} yields the uniform decay estimate.
\end{proof}
\section{Stability analysis of the nonlinear wave evolution}
To arrive at the full nonlinear problem in terms of the variables \eqref{CauchyVariable}, we derive from \Cref{LinWaveRel} the nonlinear relation
\begin{align}
\label{NonlinWaveRel}
\pd_{\tau} \mathbf{u}(\tau,\,.\,) &= \mathbf{L}_{\beta} \mathbf{u}(\tau,\,.\,)  + \mathbf{N}_{\beta}(\mathbf{u}(\tau,\,.\,)) \\\nonumber&\indent-
\begin{bmatrix}
0 \\
(T\ee^{-\tau})^{s_{p}+2}\Big(\big(\Box u + \mathcal{V}_{\beta,T} u + \mathcal{N}_{\beta,T}(u)\big) \circ {\chi\mathstrut}_{T}\Big)(\tau,\,.\,)
\end{bmatrix}
\,,
\end{align}
where
\begin{equation}
\label{IntroNonlin}
\mathbf{N}_{\beta}(\mathbf{u}(\tau,\,.\,))(\xi) =
\begin{bmatrix}
0\\
F\big(f_{\beta}^{*}(\xi)+u_{1}(\tau,\xi)\big) - F'\big(f_{\beta}^{*}(\xi)\big) u_{1}(\tau,\xi) - F\big(f_{\beta}^{*}(\xi)\big)
\end{bmatrix}
\end{equation}
with
\begin{equation*}
F(z) = z|z|^{p-1}
\qquad\text{and}\qquad
f_{\beta}^{*}(\xi) = c_{p} \gamma(\beta)^{-s_{p}} (1+\beta^{\top}\xi)^{-s_{p}} \,.
\end{equation*}
This shows that the Cauchy problem \eqref{CauchyProblem2} is formulated equivalently as an abstract Cauchy problem
\begin{equation}
\label{abstractCauchy}
\renewcommand{\arraystretch}{1.2}
\left\{
\begin{array}{rcl}
\pd_{\tau} \mathbf{u}(\tau,\,.\,) &=& \mathbf{L}_{\beta} \mathbf{u}(\tau,\,.\,) + \mathbf{N}_{\beta}(\mathbf{u}(\tau,\,.\,)) \,, \\
\mathbf{u}(0,\,.\,) &=&  \mathbf{f}^{T} + \mathbf{f}_{0}^{T} - \mathbf{f}_{\beta} \,,
\end{array}
\right.
\end{equation}
with smooth initial data given by
\begin{align*}
&&
\mathbf{f}^{T}(\xi) &=
\begin{bmatrix}
\hfill T^{s_{p}} f(T\xi) \\
T^{s_{p}+1} g(T\xi)
\end{bmatrix}
\,, &
\mathbf{f}_{0}^{T}(\xi) &=
\begin{bmatrix}
\hfill T^{s_{p}} c_{p} \\
T^{s_{p}+1} s_{p} c_{p}
\end{bmatrix}
\,, & \\
&&
\mathbf{f}_{\beta}(\xi) &=
\begin{bmatrix}
\hfill c_{p} \gamma(\beta)^{-s_{p}} (1+\beta^{\top}\xi)^{-s_{p}} \\
s_{p} c_{p} \gamma(\beta)^{-s_{p}} (1+\beta^{\top}\xi)^{-s_{p}-1}
\end{bmatrix}
\,,
&& &
\end{align*}
see \cref{SubSecOutline}. With the semigroup from \Cref{SGbeta} at hand, Duhamel's principle yields a strong formulation
\begin{equation}
\label{NonLinDuhamel}
\mathbf{u}(\tau,\,.\,) = \mathbf{S}_{\beta}(\tau) \mathbf{u}(0,\,.\,) + \int_{0}^{\tau} \mathbf{S}_{\beta}(\tau-\tau')\mathbf{N}_{\beta}(\mathbf{u}(\tau',\,.\,)) \dd \tau'
\end{equation}
of problem \eqref{abstractCauchy}. To solve the latter equation, we set up a fixed-point argument in a suitable Banach space and  suppress all linear instabilities \emph{at once} with a Lyapunov-Perron-type argument. This will carry the linear decay rate on the stable subspace over to the nonlinear flow.
\subsection{Properties of the nonlinearity}
Before we set up function spaces for the main fixed-point argument, we prove useful identities that allow us to fit the nonlinearity properly into our functional setting.
\begin{lemma}
\label{locLip}
Let $F\in C^{2}([a,b])$. Let $x_{0}\in(a,b)$ and define
\begin{equation*}
N\in C^{2}([a-x_{0},b-x_{0}])
\qquad\text{by}\qquad
N(x) = F(x_{0}+x) - F'(x_{0})x - F(x_{0}) \,.
\end{equation*}
Then
\begin{equation*}
N(x) = x^{2} \int_{0}^{1} F''(x_{0}+tx) (1-t) \dd t
\end{equation*}
and
\begin{equation*}
N(x) - N(y) = \int_{0}^{1} \int_{0}^{1} \big( x-y \big) \big( sx + (1-s) y \big)  F''\big( x_{0} + t(sx + (1-s) y) \big) \dd t \dd s \,.
\end{equation*}
\end{lemma}
\begin{proof}
Indeed, the first identity is the assertion of Taylor's theorem with integral remainder. The second identity follows from two applications of the fundamental theorem of calculus,
\begin{align*}
N(x) - N(y) &= F(x_{0}+x) - F(x_{0}+y) - F'(x_{0}) (x-y) \\&=
\int_{0}^{1} \pd_{s} F(x_{0} + y + s (x-y)) \dd s - F'(x_{0}) (x-y) \\&=
(x-y) \int_{0}^{1} \Big( F'(x_{0} + sx + (1-s) y) - F'(x_{0}) \Big) \dd s \\&=
(x-y) \int_{0}^{1} \int_{0}^{1} \big( sx + (1-s) y \big)  F''\big( x_{0} + t(sx + (1-s) y) \big) \dd t \dd s \,.
\qedhere
\end{align*}
\end{proof}
With this and expression \eqref{IntroNonlin} in mind, let us consider the following power-type nonlinearity.
\begin{definition}
\label{nonlinearity}
Let $(d,p,R)\in \NN\times\RR_{>1}\times\RR_{\geq 1}$. Let $R_{0} > R$ and $\beta\in\overline{\BB^{d}_{R_{0}^{-1}}}$. Consider
\begin{equation*}
F(z) = z|z|^{p-1} \,,
\qquad
f_{\beta}^{*}(\xi) = c_{p} \gamma(\beta)^{-s_{p}}(1+\beta^{\top}\xi)^{-s_{p}} \,,
\end{equation*}
and
\begin{equation*}
N_{\beta}(f)(\xi) = F\big(f_{\beta}^{*}(\xi)+f(\xi)\big) - F'\big(f_{\beta}^{*}(\xi)\big) f(\xi) - F\big(f_{\beta}^{*}(\xi)\big) \,,
\end{equation*}
for $\xi\in\overline{\BB^{d}_{R}}$ and $f\in C^{\infty}(\overline{\BB^{d}_{R}})$. We define
\begin{equation*}
\mathbf{N}_{\beta}(\mathbf{f})(\xi) \coloneqq
\begin{bmatrix}
0 \\
N_{\beta}(f_{1})(\xi)
\end{bmatrix}
\,.
\end{equation*}
for $\xi\in\overline{\BB^{d}_{R}}$ and $\mathbf{f}\in C^{\infty}(\overline{\BB^{d}_{R}})^{2}$.
\end{definition}
An application of \Cref{locLip} and an algebra property yields a locally Lipschitz continuous nonlinearity in the Sobolev spaces from \Cref{SobolevSpaces}.
\begin{lemma}
\label{NonLinLip}
Let $(d,p,k,R)\in \NN\times\RR_{>1}\times\NN\times\RR_{\geq 1}$ with $k>\frac{d}{2}$. Let $R_{0} > R$ and $\beta\in\overline{\BB^{d}_{R_{0}^{-1}}}$. There is a $\delta>0$ such that the map
\begin{equation*}
\mathbf{N}_{\beta}: \mathfrak{H}^{k}_{\delta}(\BB^{d}_{R}) \rightarrow \mathfrak{H}^{k}(\BB^{d}_{R}) \,, \qquad \mathbf{f} \mapsto \mathbf{N}_{\beta}(\mathbf{f}) \,,
\end{equation*}
is well-defined with
\begin{equation*}
\mathbf{N}_{\beta}(\mathbf{f}) \in \mathfrak{H}^{k+1}(\BB^{d}_{R}) \subset \mathfrak{H}^{k}(\BB^{d}_{R})
\end{equation*}
for all $\mathbf{f}\in \mathfrak{H}^{k}_{\delta}(\BB^{d}_{R})$. Moreover,
\begin{align*}
\| \mathbf{N}_{\beta} (\mathbf{f}) \|_{\mathfrak{H}^{k}(\BB^{d}_{R})} &\lesssim \delta^{2} \,, \\
\| \mathbf{N}_{\beta} (\mathbf{f}) - \mathbf{N}_{\beta} (\mathbf{g}) \|_{\mathfrak{H}^{k}(\BB^{d}_{R})} &\lesssim ( \| \mathbf{f} \|_{\mathfrak{H}^{k}(\BB^{d}_{R})} + \| \mathbf{g} \|_{\mathfrak{H}^{k}(\BB^{d}_{R})} ) \| \mathbf{f} - \mathbf{g} \|_{\mathfrak{H}^{k}(\BB^{d}_{R})} \,,
\end{align*}
for all $\mathbf{f},\mathbf{g}\in \mathfrak{H}^{k}_{\delta}(\BB^{d}_{R})$ and all $\beta\in\overline{\BB^{d}_{R_{0}^{-1}}}$.
\end{lemma}
\begin{proof}
Note that $|\beta^{\top}\xi| \leq R_{0}^{-1} R < 1$ for all $\beta\in\overline{\BB^{d}_{R_{0}^{-1}}}$ and all $\xi\in\overline{\BB^{d}_{R}}$. Thus, there is a $c>0$ such that
\begin{equation*}
f_{\beta}^{*}(\xi) = c_{p} \gamma(\beta)^{-s_{p}}(1+\beta^{\top}\xi)^{-s_{p}} \geq c
\end{equation*}
for all $\beta\in\overline{\BB^{d}_{R_{0}^{-1}}}$ and all $\xi\in\overline{\BB^{d}_{R}}$. As $k>\frac{d}{2}$, we can exploit the continuous embedding $H^{k}(\BB^{d}_{R})\hookrightarrow L^{\infty}(\BB^{d}_{R})$ to find a $\delta>0$ so that
\begin{equation*}
\| f \|_{H^{k}(\BB^{d}_{R})} \leq \delta
\qquad\text{implies}\qquad
|f(\xi)| \leq \frac{c}{2}
\text{ for } \xi\in\overline{\BB^{d}_{R}}
\end{equation*}
for all $f\in C^{\infty}(\overline{\BB^{d}_{R}})$. This shows that if $\mathbf{f}\in C^{\infty}(\overline{\BB^{d}_{R}})^{2}\cap \mathfrak{H}^{k}_{\delta}(\BB^{d}_{R})$ then $\mathbf{N}_{\beta}(\mathbf{f})\in C^{\infty}(\overline{\BB^{d}_{R}})^{2}$ for all $\beta\in\overline{\BB^{d}_{R_{0}^{-1}}}$. Invoking the algebra property of $H^{k}(\BB^{d}_{R})$ and \Cref{locLip}, we get the estimates
\begin{equation*}
\| \mathbf{N}_{\beta} (\mathbf{f}) \|_{\mathfrak{H}^{k}(\BB^{d}_{R})} = \| N_{\beta}(f_{1}) \|_{H^{k-1}(\BB^{d}_{R})} \lesssim \| N_{\beta}(f_{1}) \|_{H^{k}(\BB^{d}_{R})} \lesssim \delta^{2}
\end{equation*}
and
\begin{align*}
\| \mathbf{N}_{\beta} (\mathbf{f}) - \mathbf{N}_{\beta} (\mathbf{g}) \|_{\mathfrak{H}^{k}(\BB^{d}_{R})} &\leq \| \mathbf{N}_{\beta} (\mathbf{f}) - \mathbf{N}_{\beta} (\mathbf{g}) \|_{\mathfrak{H}^{k+1}(\BB^{d}_{R})} \\&=
\| N_{\beta}(f_{1}) - N_{\beta}(g_{1})) \|_{H^{k}(\BB^{d}_{R})} \\&\lesssim
( \| f_{1} \|_{H^{k}(\BB^{d}_{R})} + \| g_{1} \|_{H^{k}(\BB^{d}_{R})} ) \| f_{1}-g_{1} \|_{H^{k}(\BB^{d}_{R})} \\&\lesssim
( \| \mathbf{f} \|_{\mathfrak{H}^{k}(\BB^{d}_{R})} + \| \mathbf{g} \|_{\mathfrak{H}^{k}(\BB^{d}_{R})} ) \| \mathbf{f} - \mathbf{g} \|_{\mathfrak{H}^{k}(\BB^{d}_{R})} \,,
\end{align*}
for all $\mathbf{f},\mathbf{g}\in C^{\infty}(\overline{\BB^{d}_{R}})^{2}\cap \mathfrak{H}^{k}_{\delta}(\BB^{d}_{R})$ and all $\beta\in\overline{\BB^{d}_{R_{0}^{-1}}}$. With the latter estimate, we get a unique extension $\mathbf{N}_{\beta}: \mathfrak{H}^{k}_{\delta}(\BB^{d}_{R}) \rightarrow \mathfrak{H}^{k}(\BB^{d}_{R})$ whose image is contained in $\mathfrak{H}^{k+1}(\BB^{d}_{R})$ and for which the above estimates continue to hold.
\end{proof}
\subsection{Stabilized nonlinear evolution}
The delicate part is to prove existence of a solution to \Cref{NonLinDuhamel} that is globally defined on the unbounded interval. First, we define a Banach space that is suitable for a fixed-point argument and has the linear decay rate encoded.
\begin{definition}
Let $(d,p,k,R)\in\NN\times\RR_{>1}\times\NN\times\RR_{\geq 1}$ with $k>\frac{d}{2}$. Put
\begin{equation*}
\omega_{p} \coloneqq \min\left\{1,s_{p} \right\} > 0
\end{equation*}
and let $-\omega_{p}<\omega_{0}<0$ be arbitrary but fixed. We define a Banach space $(\mathfrak{X}^{k}(\BB^{d}_{R}),\|\,.\,\|_{\mathfrak{X}^{k}(\BB^{d}_{R})})$ by
\begin{align*}
\mathfrak{X}^{k}(\BB^{d}_{R}) &\coloneqq \Big\{ \mathbf{u} \in C\big( [0,\infty),\mathfrak{H}^{k}(\BB^{d}_{R}) \big) \,\Big|\, \| \mathbf{u}(\tau) \|_{\mathfrak{H}^{k}(\BB^{d}_{R})} \lesssim \ee^{\omega_{0}\tau} \text{ all } \tau\geq 0 \Big\} \,, \\
\|\mathbf{u} \|_{\mathfrak{X}^{k}(\BB^{d}_{R})} &\coloneqq \sup_{\tau\geq 0} \Big( \ee^{-\omega_{0}\tau}  \| \mathbf{u}(\tau) \|_{\mathfrak{H}^{k}(\BB^{d}_{R})} \Big) \,.
\end{align*}
We also define for $\delta>0$ a closed ball
\begin{equation*}
\mathfrak{X}^{k}_{\delta}(\BB^{d}_{R}) \coloneqq \{ \mathbf{u} \in \mathfrak{X}^{k}(\BB^{d}_{R}) \mid \| \mathbf{u} \|_{\mathfrak{X}^{k}(\BB^{d}_{R})} \leq \delta \} \,.
\end{equation*}
\end{definition}
The directions spanned by the symmetry modes induce instabilities in the nonlinear wave evolution, obstructing global existence on the unbounded interval $[0,\infty)$. By considering the projection of the Duhamel formula \eqref{NonLinDuhamel} onto the stable subspace, one is led to introduce a modification of the initial data into unstable directions.
\begin{definition}
\label{CorrTerm}
Let $(d,p,k,R)\in\NN\times\RR_{>1}\times\NN\times\RR_{\geq 1}$. Let $\mathbf{P}_{0,\beta},\mathbf{P}_{1,\beta} \in \mathfrak{L} \left( \mathfrak{H}^{k}(\BB^{d}_{R}) \right)$ be the spectral projections from  \Cref{SpecThmLbeta} and set $\mathbf{P}_{\beta} \coloneqq \mathbf{P}_{0,\beta} + \mathbf{P}_{1,\beta}$. We define a correction term
\begin{align*}
\mathbf{C}_{\beta} : \, &\mathfrak{H}^{k}(\BB^{d}_{R}) \times \mathfrak{X}^{k}_{\delta}(\BB^{d}_{R}) \rightarrow \mathfrak{H}^{k}(\BB^{d}_{R}) \\&
(\mathbf{f},\mathbf{u}) \mapsto \mathbf{P}_{\beta}\mathbf{f} + \mathbf{P}_{0,\beta} \int_{0}^{\infty} \mathbf{N}_{\beta}(\mathbf{u}(\tau')) \dd\tau' + \mathbf{P}_{1,\beta} \int_{0}^{\infty} \ee^{-\tau'} \mathbf{N}_{\beta}(\mathbf{u}(\tau')) \dd\tau' \,.
\end{align*}
\end{definition}
Note that the correction term belongs to the unstable subspace $\ran(\mathbf{P}_{\beta})\subset\mathfrak{H}^{k}(\BB^{d}_{R})$. Subtracting it from the initial data in the Duhamel formula stabilizes the nonlinear wave evolution.
\begin{proposition}
\label{NonLinEvoStableSub}
Let $(d,p,k,R)\in\NN\times\RR_{>1}\times\NN\times\RR_{\geq 1}$ with $k>\frac{d}{2}$. Let $-\omega_{p}<\omega_{0}<0$. There are $0<\delta_{0}<1$, $C_{0} > 1$, $R_{0}>R$ such that for all $0<\delta\leq\delta_{0}$, $C\geq C_{0}$, all $\beta\in\overline{\BB^{d}_{R_{0}^{-1}}}$ and all $\mathbf{f}\in\mathfrak{H}^{k}(\BB^{d}_{R})$ with $\| \mathbf{f} \|_{\mathfrak{H}^{k}(\BB^{d}_{R})} \leq \frac{\delta}{C}$ there is a unique $\mathbf{u}_{\beta}\in\mathfrak{X}^{k}(\BB^{d}_{R})$ such that $\| \mathbf{u}_{\beta} \|_{\mathfrak{X}^{k}(\BB^{d}_{R})} \leq \delta$ and
\begin{equation}
\label{StabilizedSolution}
\mathbf{u}_{\beta}(\tau) = \mathbf{S}_{\beta}(\tau) (\mathbf{f} - \mathbf{C}_{\beta}(\mathbf{f},\mathbf{u}_{\beta})) + \int_{0}^{\tau} \mathbf{S}_{\beta}(\tau-\tau')\mathbf{N}_{\beta}(\mathbf{u}_{\beta}(\tau')) \dd \tau'
\end{equation}
for all $\tau\geq 0$. Moreover, the data-to-solution map
\begin{equation*}
\mathfrak{H}^{k}_{\delta}(\BB^{d}_{R}) \rightarrow \mathfrak{X}^{k}(\BB^{d}_{R}) \,, \qquad \mathbf{f} \mapsto \mathbf{u}_{\beta} \,,
\end{equation*}
is Lipschitz continuous.
\end{proposition}
\begin{proof}
Fix $R_{0}>R$ from \Cref{stableEvo}. For now, let $\delta>0$. For $\mathbf{u}\in\mathfrak{X}^{k}_{\delta}(\BB^{d}_{R})$ and $\tau\geq 0$ we define
\begin{equation*}
\mathbf{K}_{\beta}(\mathbf{f},\mathbf{u})(\tau) \coloneqq
\mathbf{S}_{\beta}(\tau) (\mathbf{f} - \mathbf{C}_{\beta}(\mathbf{f},\mathbf{u})) + \int_{0}^{\tau} \mathbf{S}_{\beta}(\tau-\tau')\mathbf{N}_{\beta}(\mathbf{u}(\tau')) \dd \tau' \,.
\end{equation*}
Exploiting the semigroup identities provided in \Cref{stableEvo}, we note that the decomposition
\begin{align*}
\mathbf{K}_{\beta}(\mathbf{f},\mathbf{u})(\tau) = \Big(\mathbf{I}&-\mathbf{P}_{\beta}\Big)\Big( \mathbf{S}_{\beta}(\tau) \mathbf{f} + \int_{0}^{\tau} \mathbf{S}_{\beta}(\tau-\tau')\mathbf{N}_{\beta}(\mathbf{u}(\tau')) \dd \tau'\Big) \\&-\mathbf{P}_{0,\beta} \int_{\tau}^{\infty} \mathbf{N}_{\beta}(\mathbf{u}(\tau')) \dd\tau' - \mathbf{P}_{1,\beta} \int_{\tau}^{\infty} \ee^{\tau-\tau'} \mathbf{N}_{\beta}(\mathbf{u}(\tau')) \dd\tau'
\end{align*}
along the stable and unstable subspace holds. Using the uniform bounds on the semigroup and nonlinearity from \Cref{stableEvo,NonLinLip}, we get the estimate
\begin{align*}
\| \mathbf{K}_{\beta}(\mathbf{f},\mathbf{u})(\tau) \|_{\mathfrak{H}^{k}(\BB^{d}_{R})} &\lesssim \ee^{\omega_{0}\tau} \| \mathbf{f} \|_{\mathfrak{H}^{k}(\BB^{d}_{R})} + \int_{0}^{\tau} \ee^{\omega_{0}(\tau-\tau')} \| \mathbf{u}(\tau') \|_{\mathfrak{H}^{k}(\BB^{d}_{R})}^{2}  \dd\tau' \\&\indent+
\int_{\tau}^{\infty} \| \mathbf{u}(\tau') \|_{\mathfrak{H}^{k}(\BB^{d}_{R})}^{2} \dd\tau' + \int_{\tau}^{\infty} \ee^{\tau-\tau'} \| \mathbf{u}(\tau') \|_{\mathfrak{H}^{k}(\BB^{d}_{R})}^{2} \dd\tau' \\&\lesssim
\frac{\delta}{C} \ee^{\omega_{0}\tau} + \int_{0}^{\tau} \ee^{\omega_{0}(\tau+\tau')} \| \mathbf{u} \|_{\mathfrak{X}^{k}(\BB^{d}_{R})}^{2}  \dd\tau' \\&\indent+
\int_{\tau}^{\infty} \ee^{2\omega_{0}\tau'} \| \mathbf{u} \|_{\mathfrak{X}^{k}(\BB^{d}_{R})}^{2} \dd\tau' + \int_{\tau}^{\infty} \ee^{\tau-\tau'+2\omega_{0}\tau'} \| \mathbf{u} \|_{\mathfrak{X}^{k}(\BB^{d}_{R})}^{2} \dd\tau' \\&\lesssim
\frac{\delta}{C} \ee^{\omega_{0}\tau} + \delta^{2} \ee^{\omega_{0}\tau}
\end{align*}
for all $\mathbf{u}\in\mathfrak{X}^{k}_{\delta}(\BB^{d}_{R})$, all $\tau\geq 0$ and all $\beta\in\overline{\BB^{d}_{R_{0}^{-1}}}$. Thus, if $0<\delta\leq\delta_{0}$ with $\delta_{0} > 0$ small enough, for each $\mathbf{f}\in\mathfrak{H}^{k}_{\delta}(\BB^{d}_{R})$ the map
\begin{equation*}
\mathbf{K}_{\beta}(\mathbf{f},\,.\,): \mathfrak{X}^{k}_{\delta}(\BB^{d}_{R}) \rightarrow \mathfrak{X}^{k}_{\delta}(\BB^{d}_{R}) \,, \quad \mathbf{u} \mapsto \Big( \tau \mapsto \mathbf{K}_{\beta}(\mathbf{f},\mathbf{u})(\tau) \Big) \,,
\end{equation*}
is well-defined.\par
\medskip
In order to show that this map is a contraction, we also note the decomposition
\begin{align*}
&
\mathbf{K}_{\beta}(\mathbf{f},\mathbf{u})(\tau) - \mathbf{K}_{\beta}(\mathbf{f},\mathbf{v})(\tau) \\&\indent=
\Big(\mathbf{I}-\mathbf{P}_{\beta}\Big)\int_{0}^{\tau} \mathbf{S}_{\beta}(\tau-\tau')\Big( \mathbf{N}_{\beta}(\mathbf{u}(\tau')) - \mathbf{N}_{\beta}(\mathbf{v}(\tau')) \Big) \dd \tau'
\\&\indent\indent-
\mathbf{P}_{0,\beta} \int_{\tau}^{\infty} \Big( \mathbf{N}_{\beta}(\mathbf{u}(\tau')) - \mathbf{N}_{\beta}(\mathbf{v}(\tau')) \Big) \dd\tau'
\\&\indent\indent-
\mathbf{P}_{1,\beta} \int_{\tau}^{\infty} \ee^{\tau-\tau'} \Big( \mathbf{N}_{\beta}(\mathbf{u}(\tau')) - \mathbf{N}_{\beta}(\mathbf{v}(\tau')) \Big) \dd\tau'
\end{align*}
and estimate as before
\begin{align*}
&\| \mathbf{K}_{\beta}(\mathbf{f},\mathbf{u})(\tau) - \mathbf{K}_{\beta}(\mathbf{f},\mathbf{v})(\tau) \|_{\mathfrak{H}^{k}(\BB^{d}_{R})} \\&\indent\lesssim
\int_{0}^{\tau} \ee^{\omega_{0}(\tau-\tau')} \| \mathbf{N}_{\beta}(\mathbf{u}(\tau')) - \mathbf{N}_{\beta}(\mathbf{v}(\tau')) \|_{\mathfrak{H}^{k}(\BB^{d}_{R})} \dd\tau' \\&\indent\indent+
\int_{\tau}^{\infty} \ee^{2\omega_{0}\tau'} \| \mathbf{N}_{\beta}(\mathbf{u}(\tau')) - \mathbf{N}_{\beta}(\mathbf{v}(\tau')) \|_{\mathfrak{H}^{k}(\BB^{d}_{R})} \dd\tau' \\&\indent\indent+
\int_{\tau}^{\infty} \ee^{\tau-\tau'+2\omega_{0}\tau'} \| \mathbf{N}_{\beta}(\mathbf{u}(\tau')) - \mathbf{N}_{\beta}(\mathbf{v}(\tau')) \|_{\mathfrak{H}^{k}(\BB^{d}_{R})} \dd\tau' \\&\indent\lesssim
(\|\mathbf{u}\|_{\mathfrak{X}^{k}(\BB^{d}_{R})}+\|\mathbf{v}\|_{\mathfrak{X}^{k}(\BB^{d}_{R})}) \|\mathbf{u}-\mathbf{v}\|_{\mathfrak{X}^{k}(\BB^{d}_{R})} \int_{0}^{\tau} \ee^{\omega_{0}(\tau+\tau')} \dd\tau' \\&\indent\indent+
(\|\mathbf{u}\|_{\mathfrak{X}^{k}(\BB^{d}_{R})}+\|\mathbf{v}\|_{\mathfrak{X}^{k}(\BB^{d}_{R})}) \|\mathbf{u}-\mathbf{v}\|_{\mathfrak{X}^{k}(\BB^{d}_{R})} \int_{\tau}^{\infty} \ee^{-2\omega_{0}\tau'} \dd\tau' \\&\indent\indent+
(\|\mathbf{u}\|_{\mathfrak{X}^{k}(\BB^{d}_{R})}+\|\mathbf{v}\|_{\mathfrak{X}^{k}(\BB^{d}_{R})}) \|\mathbf{u}-\mathbf{v}\|_{\mathfrak{X}^{k}(\BB^{d}_{R})} \int_{\tau}^{\infty} \ee^{\tau-\tau'-2\omega_{0}\tau'} \dd\tau' \\&\indent\lesssim
\delta \ee^{\omega_{0}\tau} \|\mathbf{u}-\mathbf{v}\|_{\mathfrak{X}^{k}(\BB^{d}_{R})}
\end{align*}
for all $\mathbf{u},\mathbf{v}\in\mathfrak{X}^{k}_{\delta}(\BB^{d}_{R})$, all $\tau\geq 0$ and all $\beta\in\overline{\BB^{d}_{R_{0}^{-1}}}$. Upon possibly choosing $\delta_{0}>0$ smaller, we infer
\begin{equation*}
\| \mathbf{K}_{\beta}(\mathbf{f},\mathbf{u}) - \mathbf{K}_{\beta}(\mathbf{f},\mathbf{v}) \|_{\mathfrak{X}^{k}(\BB^{d}_{R})} \leq \frac{1}{2} \| \mathbf{u} - \mathbf{v} \|_{\mathfrak{X}^{k}(\BB^{d}_{R})}
\end{equation*}
for all $\mathbf{u},\mathbf{v}\in\mathfrak{X}^{k}_{\delta}(\BB^{d}_{R})$ and all $\beta\in\overline{\BB^{d}_{R_{0}^{-1}}}$. So $\mathbf{K}_{\beta}(\mathbf{f},\,.\,)$ is a contraction mapping and Banach's fixed-point theorem yields a unique fixed point $\mathbf{u}_{\beta}\in \mathfrak{X}^{k}_{\delta}(\BB^{d}_{R})$ of $\mathbf{K}_{\beta}(\mathbf{f},\,.\,)$.\par
\medskip
Lastly, let $\mathbf{f},\mathbf{g}\in\mathfrak{H}^{k}_{\delta}(\BB^{d}_{R})$ and let $\mathbf{u}_{\beta},\mathbf{v}_{\beta}\in\mathfrak{X}^{k}_{\delta}(\BB^{d}_{R})$ be the unique fixed point of $\mathbf{K}_{\beta}(\mathbf{f},\,.\,)$, $\mathbf{K}_{\beta}(\mathbf{g},\,.\,)$, respectively. Then
\begin{align*}
\mathbf{u}_{\beta}(\tau) - \mathbf{v}_{\beta}(\tau) &= \mathbf{K}_{\beta}(\mathbf{f},\mathbf{u}_{\beta})(\tau) - \mathbf{K}_{\beta}(\mathbf{g},\mathbf{v}_{\beta})(\tau) \\&=
\mathbf{K}_{\beta}(\mathbf{f},\mathbf{u}_{\beta})(\tau) - \mathbf{K}_{\beta}(\mathbf{f},\mathbf{v}_{\beta})(\tau) + \mathbf{S}_{\beta}(\tau)(\mathbf{I}-\mathbf{P}_{\beta})(\mathbf{f}-\mathbf{g}) \,.
\end{align*}
Applying previous estimates yields
\begin{align*}
\| \mathbf{K}_{\beta}(\mathbf{f},\mathbf{u}_{\beta}) - \mathbf{K}_{\beta}(\mathbf{f},\mathbf{v}_{\beta}) \|_{\mathfrak{X}^{k}(\BB^{d}_{R})} &\leq \frac{1}{2} \| \mathbf{u}_{\beta} - \mathbf{v}_{\beta} \|_{\mathfrak{X}^{k}(\BB^{d}_{R})} \,, \\
\| \mathbf{S}_{\beta}(\tau)(\mathbf{I}-\mathbf{P}_{\beta})(\mathbf{f}-\mathbf{g}) \|_{\mathfrak{H}^{k}(\BB^{d}_{R})} &\lesssim \ee^{\omega_{0}\tau} \|\mathbf{f}-\mathbf{g} \|_{\mathfrak{H}^{k}(\BB^{d}_{R})} \,,
\end{align*}
for all $\tau\geq 0$ and $\beta\in\overline{\BB^{d}_{R_{0}^{-1}}}$. This shows the Lipschitz continuous dependence on the initial data.
\end{proof}
\subsection{Stable flow near the blowup solution}
The initial data for our abstract Cauchy problem \eqref{abstractCauchy} are introduced as follows.
\begin{definition}
Let $(d,p,k,R)\in\NN\times\RR_{>1}\times\NN\times\RR_{\geq 1}$. Let $R_{0} > R$ and $\beta\in\overline{\BB^{d}_{R_{0}^{-1}}}$, $T>0$. We define the operator
\begin{equation*}
\mathbf{U}_{\beta,T} : C^{\infty}(\RR^{d})^{2} \rightarrow \mathfrak{H}^{k}(\BB^{d}_{R}) \,, \qquad \mathbf{f} \mapsto \mathbf{f}^{T} + \mathbf{f}_{0}^{T} - \mathbf{f}_{\beta} \,,
\end{equation*}
where $\mathbf{f}^{T},\mathbf{f}_{0}^{T},\mathbf{f}_{\beta}\in C^{\infty}(\overline{\BB^{d}_{R}})^{2}$ are defined by
\begin{align*}
&&
\mathbf{f}^{T}(\xi) &=
\begin{bmatrix}
\hfill T^{s_{p}} f_{1}(T\xi) \\
T^{s_{p}+1} f_{2}(T\xi)
\end{bmatrix}
\,, &
\mathbf{f}_{0}^{T}(\xi) &=
\begin{bmatrix}
\hfill T^{s_{p}} c_{p} \\
T^{s_{p}+1} s_{p} c_{p}
\end{bmatrix}
\,, & \\
&&
\mathbf{f}_{\beta}(\xi) &=
\begin{bmatrix}
\hfill c_{p} \gamma(\beta)^{-s_{p}} (1+\beta^{\top}\xi)^{-s_{p}} \\
s_{p} c_{p} \gamma(\beta)^{-s_{p}} (1+\beta^{\top}\xi)^{-s_{p}-1}
\end{bmatrix}
\,. && &
\end{align*}
\end{definition}
To guarantee a stabilized evolution for such data, we need a smallness property of the initial data operator. This is ensured by the following lemma.
\begin{lemma}
\label{InitDatOpSmall}
Let $(d,p,k,R)\in\NN\times\RR_{>1}\times\NN\times\RR_{\geq 1}$. Let $R_{0} > R$ and $\beta\in\overline{\BB^{d}_{R_{0}^{-1}}}$, $T\in\overline{\BB^{1}_{\frac{1}{2}}(1)}$. We have
\begin{equation*}
\mathbf{U}_{\beta,T}(\mathbf{f}) = \mathbf{f}^{T} + (T-1) \mathbf{f}_{1,\beta} + \beta^{i} \mathbf{f}_{0,\beta,i} + \mathbf{r}(\beta,T) \,,
\end{equation*}
for any $\mathbf{f}\in C^{\infty}(\RR^{d})^{2}$, where $\mathbf{f}_{1,\beta},\mathbf{f}_{0,\beta,i}$ are the symmetry modes from \Cref{EV}, and
\begin{equation*}
\| \mathbf{r}(\beta,T) \|_{\mathfrak{H}^{k}(\BB^{d}_{R})} \lesssim |T-1|^{2} + |\beta|^{2}
\end{equation*}
for all $\beta\in\overline{\BB^{d}_{R_{0}^{-1}}}$ and all $T\in\overline{\BB^{1}_{\frac{1}{2}}(1)}$.
\end{lemma}
\begin{proof}
For fixed $\xi\in\overline{\BB^{d}_{R}}$, Taylor's theorem applied to the components of the map
\begin{equation*}
\overline{\BB^{d}_{R_{0}^{-1}}}\times\overline{\BB^{1}_{\frac{1}{2}}(1)} \rightarrow\RR^{2} \,, \quad (\beta,T) \mapsto \mathbf{f}_{0}^{T}(\xi) - \mathbf{f}_{\beta}(\xi) \,,
\end{equation*}
yields
\begin{equation*}
\mathbf{f}_{0}^{T}(\xi) - \mathbf{f}_{\beta}(\xi) =
(T-1)\mathbf{f}_{1,0}(\xi) + \beta^{i}\mathbf{f}_{0,0,i}(\xi) + \widetilde{\mathbf{r}}_{\beta,T}(\xi)
\end{equation*}
with remainder
\begin{align*}
[\widetilde{\mathbf{r}}_{\beta,T}(\xi)]_{\ell} &= (T-1)^{2} \int_{0}^{1} \left.\pd_{T'}^{2} [\mathbf{f}_{0}^{T'}(\xi)]_{\ell}\right|_{T'=1+z(T-1)}(1-z)\dd z \\&\indent-
 \sum_{i,j=1}^{d} \beta^{i}\beta^{j} \int_{0}^{1} \left.\pd_{{\beta'}^{i}} \pd_{{\beta'}^{j}} [\mathbf{f}_{\beta'}(\xi)]_{\ell}\right|_{\beta'=z\beta}(1-z)\dd z \,, \qquad \ell = 1,2.
\end{align*}
Thus
\begin{equation*}
\mathbf{f}_{0}^{T} - \mathbf{f}_{\beta} =
(T-1)\mathbf{f}_{1,\beta} + \beta^{i}\mathbf{f}_{0,\beta,i} + \mathbf{r}_{\beta,T} 
\end{equation*}
with
\begin{equation*}
\mathbf{r}_{\beta,T} = \widetilde{\mathbf{r}}_{\beta,T} - (T-1) (\mathbf{f}_{1,\beta} - \mathbf{f}_{1,0}) - \beta^{i} (\mathbf{f}_{0,\beta,i}-\mathbf{f}_{0,0,i})
\end{equation*}
and the bound follows.
\end{proof}
It remains to remove the correction term in \Cref{StabilizedSolution} to obtain a global solution to \Cref{NonLinDuhamel}.
\begin{proposition}
\label{StableNonlinFlowMild}
Let $(d,p,k,R)\in\NN\times\RR_{>1}\times\NN\times\RR_{\geq 1}$ with $k>\frac{d}{2}$. Let $0 < \varepsilon < \omega_{p}$. There are constants $0<\delta_{\varepsilon}<1$, $C_{\varepsilon} > 1$ such that for all $0<\delta\leq\delta_{\varepsilon}$, $C\geq C_{\varepsilon}$ and for all real-valued $\mathbf{f}\in C^{\infty}(\RR^{d})^{2}$ with $\| \mathbf{f} \|_{\mathfrak{H}^{k}(\RR^{d})} \leq \frac{\delta}{C^{2}}$ there are parameters $\beta^{*}\in\overline{\BB^{d}_{\frac{\delta}{C}}}$, $T^{*}\in\overline{\BB^{1}_{\frac{\delta}{C}}(1)}$ and a unique $\mathbf{u}_{\beta^{*},T^{*}}\in C\big([0,\infty),\mathfrak{H}^{k}(\BB^{d}_{R})\big)$ such that $\| \mathbf{u}_{\beta^{*},T^{*}}(\tau) \|_{\mathfrak{H}^{k}(\BB^{d}_{R})} \leq \delta \ee^{(-\omega_{p}+\varepsilon)\tau}$ and
\begin{equation}
\label{CauchyMild}
\mathbf{u}_{\beta^{*},T^{*}}(\tau) = \mathbf{S}_{\beta^{*}}(\tau) \mathbf{U}_{\beta^{*},T^{*}}(\mathbf{f}) + \int_{0}^{\tau} \mathbf{S}_{\beta^{*}}(\tau-\tau')\mathbf{N}_{\beta^{*}}(\mathbf{u}_{\beta^{*},T^{*}}(\tau')) \dd \tau'
\end{equation}
for all $\tau\geq 0$.
\end{proposition}
\begin{proof}
Let $\varepsilon>0$ so that $\omega_{0} \coloneqq -\omega_{p} + \varepsilon < 0$. For such $\omega_{0}<0$, pick $\delta_{0}>0$, $C_{0}>1$ from \Cref{NonLinEvoStableSub} and let $0<\delta'\leq\delta_{0}$ and $C'\geq C_{0}$. Let
\begin{equation*}
0<\delta\leq\delta_{\varepsilon} \coloneqq \frac{\delta'}{M_{\omega}}
\qquad\text{and}\qquad
C\geq C_{\varepsilon} \coloneqq M_{\omega} C'
\end{equation*}
for $M_{\omega} \geq 1$. Let $\mathbf{f}\in C^{\infty}(\RR^{d})^{2}$ with $\| \mathbf{f} \|_{\mathfrak{H}^{k}(\RR^{d})} \leq \frac{\delta}{C^{2}}$. If $M_{\omega}\geq 1$ is large enough, we get from \Cref{InitDatOpSmall}
\begin{align*}
\| \mathbf{U}_{\beta,T}(\mathbf{f}) \|_{\mathfrak{H}^{k}(\BB^{d}_{R})} &\leq \| \mathbf{f}^{T} \|_{\mathfrak{H}^{k}(\BB^{d}_{R})} + |T-1| \| \mathbf{f}_{1,\beta} \|_{\mathfrak{H}^{k}(\BB^{d}_{R})} \\&\indent+ \sum_{i=1}^{d} |\beta^{i}| \| \mathbf{f}_{0,\beta,i} \|_{\mathfrak{H}^{k}(\BB^{d}_{R})} + \| \mathbf{r}_{\beta,T} \|_{\mathfrak{H}^{k}(\BB^{d}_{R})} \\&\leq \frac{\delta}{C'}
\end{align*}
for all $\beta\in\overline{\BB^{d}_{\frac{\delta}{C}}}$ and all $T\in\overline{\BB^{1}_{\frac{\delta}{C}}(1)}$. Hence, we can fix $\delta'>0$ and $C'>1$ for now so that $\mathbf{U}_{\beta,T}(\mathbf{f})\in\mathfrak{H}^{k}(\BB^{d}_{R})$ satisfies the assumptions for the initial data in \Cref{NonLinEvoStableSub} and we conclude the existence of a unique solution $\mathbf{u}_{\beta,T}\in\mathfrak{X}^{k}(\BB^{d}_{R})$ with $\| \mathbf{u}_{\beta,T} \|_{\mathfrak{X}^{k}(\BB^{d}_{R})} \leq \delta$ to
\begin{equation*}
\mathbf{u}_{\beta,T}(\tau) = \mathbf{S}_{\beta}(\tau) (\mathbf{U}_{\beta,T}(\mathbf{f}) - \mathbf{C}_{\beta}(\mathbf{U}_{\beta,T}(\mathbf{f}),\mathbf{u}_{\beta,T}) + \int_{0}^{\tau} \mathbf{S}_{\beta}(\tau-\tau')\mathbf{N}_{\beta}(\mathbf{u}_{\beta,T}(\tau')) \dd \tau' \,,
\end{equation*}
for each $\beta\in\overline{\BB^{d}_{\frac{\delta}{C}}}$, $T\in\overline{\BB^{1}_{\frac{\delta}{C}}(1)}$. Now, the task is to determine parameters such that the correction term $\mathbf{C}_{\beta}(\mathbf{U}_{\beta,T}(\mathbf{f}),\mathbf{u}_{\beta,T}) \in \ran(\mathbf{P}_{\beta})$ from \Cref{CorrTerm} is equal to $\mathbf{0}$. By \Cref{SpecThmLbeta} we have that $\ran(\mathbf{P}_{\beta}) \subset \mathfrak{H}^{k}(\BB^{d}_{R})$ is a finite-dimensional sub-Hilbert space spanned by the linearly independent set $\{\mathbf{f}_{0,\beta,1},\ldots,\mathbf{f}_{0,\beta,d},\mathbf{f}_{1,\beta}\}$ of symmetry modes. Therefore, it is sufficient to prove that there are parameters for which the linear functional
\begin{equation*}
\ell_{\beta,T}: \ran(\mathbf{P}_{\beta}) \rightarrow \RR \,, \quad
\mathbf{g} \mapsto \Big( \mathbf{C}_{\beta}(\mathbf{U}_{\beta,T}(\mathbf{f}),\mathbf{u}_{\beta,T}) \,\Big|\, \mathbf{g} \Big)_{\mathfrak{H}^{k}(\BB^{d}_{R})}
\end{equation*}
is identically zero on a basis of $\ran(\mathbf{P}_{\beta})$. Note with \Cref{InitDatOpSmall} that
\begin{align*}
&
\ell_{\beta,T}(\mathbf{g}) =
\Big( \mathbf{P}_{\beta} \mathbf{f}^{T} \,\Big|\, \mathbf{g} \Big)_{\mathfrak{H}^{k}(\BB^{d}_{R})} +
(T-1) \Big( \mathbf{f}_{1,\beta} \,\Big|\, \mathbf{g} \Big)_{\mathfrak{H}^{k}(\BB^{d}_{R})} +
\beta^{i} \Big( \mathbf{f}_{0,\beta,i} \,\Big|\, \mathbf{g} \Big)_{\mathfrak{H}^{k}(\BB^{d}_{R})} \\&\indent+
\Big( \mathbf{P}_{\beta} \mathbf{r}(\beta,T) \,\Big|\, \mathbf{g} \Big)_{\mathfrak{H}^{k}(\BB^{d}_{R})} \\&\indent+
\Big( \mathbf{P}_{0,\beta} \int_{0}^{\infty} \mathbf{N}_{\beta}(\mathbf{u}_{\beta,T}(\tau')) \dd\tau' \,\Big|\, \mathbf{g} \Big)_{\mathfrak{H}^{k}(\BB^{d}_{R})} \\&\indent+
\Big( \mathbf{P}_{1,\beta} \int_{0}^{\infty} \ee^{-\tau'} \mathbf{N}_{\beta}(\mathbf{u}_{\beta,T}(\tau')) \dd\tau' \,\Big|\, \mathbf{g} \Big)_{\mathfrak{H}^{k}(\BB^{d}_{R})} \,.
\end{align*}
\sloppy To achieve vanishing of $\ell_{\beta,T}$, choose the dual basis $\{\mathbf{g}^{1}_{\beta},\ldots,\mathbf{g}^{d}_{\beta},\mathbf{g}^{d+1}_{\beta}\}$ for the basis $\{\mathbf{f}_{0,\beta,1},\ldots,\mathbf{f}_{0,\beta,d},\mathbf{f}_{1,\beta}\}$, which is obtained by letting $\Gamma(\beta)^{mn}$ for $m,n=1,\ldots,d,d+1$ be the components of the inverse matrix of the real-valued Gram matrix composed of entries
\begin{align*}
&&
\Gamma(\beta)_{ij} &= \Big( \mathbf{f}_{0,\beta,i} \,\Big|\, \mathbf{f}_{0,\beta,j} \Big)_{\mathfrak{H}^{k}(\BB^{d}_{R})} \,,
&
\Gamma(\beta)_{i,d+1} &= \Big( \mathbf{f}_{0,\beta,i} \,\Big|\, \mathbf{f}_{1,\beta} \Big)_{\mathfrak{H}^{k}(\BB^{d}_{R})} \,,
&\\
&&
\Gamma(\beta)_{d+1,j} &= \Big( \mathbf{f}_{1,\beta} \,\Big|\, \mathbf{f}_{0,\beta,j} \Big)_{\mathfrak{H}^{k}(\BB^{d}_{R})} \,,
&
\Gamma(\beta)_{d+1,d+1} &= \Big( \mathbf{f}_{1,\beta} \,\Big|\, \mathbf{f}_{1,\beta} \Big)_{\mathfrak{H}^{k}(\BB^{d}_{R})} \,,
&
\end{align*}
for $i,j=1,\ldots,d$, and putting
\begin{equation*}
\mathbf{g}^{n}_{\beta} \coloneqq \sum_{m=1}^{d} \Gamma(\beta)^{mn} \mathbf{f}_{0,\beta,m} + \Gamma(\beta)^{d+1,n} \mathbf{f}_{1,\beta} \,,
\qquad n=1,\ldots,d,d+1.
\end{equation*}
Then
\begin{align*}
&&
\Big( \mathbf{f}_{0,\beta,i} \,\Big|\, \mathbf{g}^{j}_{\beta} \Big)_{\mathfrak{H}^{k}(\BB^{d}_{R})} &= \delta^{j}_{i} \,,
&
\Big( \mathbf{f}_{0,\beta,i} \,\Big|\, \mathbf{g}^{d+1}_{\beta} \Big)_{\mathfrak{H}^{k}(\BB^{d}_{R})} &= 0 \,,
&\\
&&
\Big( \mathbf{f}_{1,\beta} \,\Big|\, \mathbf{g}^{j}_{\beta} \Big)_{\mathfrak{H}^{k}(\BB^{d}_{R})} &= 0  \,,
&
\Big( \mathbf{f}_{1,\beta} \,\Big|\, \mathbf{g}^{d+1}_{\beta} \Big)_{\mathfrak{H}^{k}(\BB^{d}_{R})} &= 1 \,,
&
\end{align*}
for $i,j=1,\ldots,d$ and the components of each element in $\{\mathbf{g}^{1}_{\beta},\ldots,\mathbf{g}^{d}_{\beta},\mathbf{g}^{d+1}_{\beta}\} \subset \ran(\mathbf{P}_{\beta})$ are smooth functions which depend smoothly on the Lorentz parameter by Cramer's rule. Next, define the continuous map $F = (F_{1},\ldots,F_{d},1+F_{d+1}): \overline{\BB^{d}_{\frac{\delta}{C}}} \times \overline{\BB^{1}_{\frac{\delta}{C}}(1)} \rightarrow\RR^{d+1}$ by
\begin{align*}
F_{n}(\beta,T) &=
-\Big( \mathbf{P}_{\beta} \mathbf{f}^{T} \,\Big|\, \mathbf{g}^{n}_{\beta} \Big)_{\mathfrak{H}^{k}(\BB^{d}_{R})}
-
\Big( \mathbf{P}_{\beta} \mathbf{r}_{\beta,T} \,\Big|\, \mathbf{g}^{n}_{\beta} \Big)_{\mathfrak{H}^{k}(\BB^{d}_{R})}
\\&\indent-
\Big( \mathbf{P}_{0,\beta} \int_{0}^{\infty} \mathbf{N}_{\beta}(\mathbf{u}_{\beta,T}(\tau')) \dd\tau' \,\Big|\, \mathbf{g}^{n}_{\beta} \Big)_{\mathfrak{H}^{k}(\BB^{d}_{R})}
\\&\indent-
\Big( \mathbf{P}_{1,\beta} \int_{0}^{\infty} \ee^{-\tau'} \mathbf{N}_{\beta}(\mathbf{u}_{\beta,T}(\tau')) \dd\tau' \,\Big|\, \mathbf{g}^{n}_{\beta} \Big)_{\mathfrak{H}^{k}(\BB^{d}_{R})} \,, \quad n=1,\ldots,d,d+1.
\end{align*}
We use Cauchy-Schwarz and \Cref{InitDatOpSmall,NonLinLip} to get the estimate
\begin{align*}
|F_{n}(\beta,T)| &\lesssim \| \mathbf{f}^{T} \|_{\mathfrak{H}^{k}(\BB^{d}_{R})} + \| \mathbf{r}_{\beta,T} \|_{\mathfrak{H}^{k}(\BB^{d}_{R})} + \| \mathbf{u}_{\beta,T} \|_{\mathfrak{X}^{k}(\BB^{d}_{R})}^{2} \\&\lesssim
\frac{\delta}{C^{2}} + \frac{\delta^{2}}{C^{2}} + \delta^{2} \\&\lesssim
\frac{\delta}{C} \frac{1}{M_{\omega}C'} + \frac{\delta}{C} \frac{\delta'}{M_{\omega}^{2}C'} + \frac{\delta}{C} C'\delta'
\end{align*}
for all $\beta\in\overline{\BB^{d}_{\frac{\delta}{C}}}$, all $T\in\overline{\BB^{1}_{\frac{\delta}{C}}(1)}$ and all $n=1,\ldots,d,d+1$. Now, upon choosing $\delta'>0$ smaller and $C' > 1$, $M_{\omega}\geq 1$ larger, we can fix the above values of $\delta_{\varepsilon},C_{\varepsilon}$ so that $F$ becomes a continuous self-map on $\overline{\BB^{d}_{\frac{\delta}{C}}} \times \overline{\BB^{1}_{\frac{\delta}{C}}(1)}$ for any $0<\delta\leq\delta_{\varepsilon}$ and $C\geq C_{\varepsilon}$. According to Brouwer's fixed-point theorem, $F$ has a fixed point $(\beta^{*},T^{*})\in \overline{\BB^{d}_{\frac{\delta}{C}}} \times \overline{\BB^{1}_{\frac{\delta}{C}}(1)}$ which, by construction, satisfies
\begin{equation*}
{\beta^{*}}^{i} = F_{i}(\beta^{*},T^{*}) = {\beta^{*}}^{i} - \ell_{\beta^{*},T^{*}}(\mathbf{g}^{i}_{\beta^{*}}) \,, \quad
T^{*} = 1 + F_{d+1}(\beta^{*},T^{*}) = T^{*} - \ell_{\beta^{*},T^{*}}(\mathbf{g}^{d+1}_{\beta^{*}}) \,,
\end{equation*}
for $i = 1,\ldots,d$. Thus $\ell_{\beta^{*},T^{*}} \equiv 0$, as desired.
\end{proof}
The just obtained mild solution is in fact a jointly smooth classical solution.
\begin{proposition}
\label{StableNonlinFlowClassic}
Let $(d,p,k,R)\in\NN\times\RR_{>1}\times\NN\times\RR_{\geq 1}$ with $k>\frac{d}{2}$. Let $0 < \varepsilon < \omega_{p}$. There are constants $0<\delta_{\varepsilon}<1$, $C_{\varepsilon} > 1$ such that for all $0<\delta\leq\delta_{\varepsilon}$, $C\geq C_{\varepsilon}$ and all real-valued $\mathbf{f}\in C^{\infty}(\RR^{d})^{2}$ with $\| \mathbf{f} \|_{\mathfrak{H}^{k}(\RR^{d})} \leq \frac{\delta}{C^{2}}$ there are parameters $\beta^{*}\in\overline{\BB^{d}_{\frac{\delta}{C}}}$, $T^{*}\in\overline{\BB^{1}_{\frac{\delta}{C}}(1)}$ and a unique $\mathbf{u}_{\beta^{*},T^{*}} \in C^{\infty}\big(\overline{(0,\infty)\times\BB^{d}_{R}}\big)^{2}$ such that $\| \mathbf{u}_{\beta^{*},T^{*}}(\tau,\,.\,) \|_{\mathfrak{H}^{k}(\BB^{d}_{R})} \leq \delta \ee^{(-\omega_{p}+\varepsilon)\tau}$ and
\begin{equation}
\label{CauchyClassical}
\renewcommand{\arraystretch}{1.2}
\left\{
\begin{array}{rcl}
\pd_{\tau} \mathbf{u}_{\beta^{*},T^{*}}(\tau,\,.\,) &=& \mathbf{L}_{\beta^{*}} \mathbf{u}_{\beta^{*},T^{*}}(\tau,\,.\,) + \mathbf{N}_{\beta^{*}}(\mathbf{u}_{\beta^{*},T^{*}}(\tau,\,.\,)) \,, \\
\mathbf{u}_{\beta^{*},T^{*}}(0,\,.\,) &=&  \mathbf{f}^{T^{*}} + \mathbf{f}_{0}^{T^{*}} - \mathbf{f}_{\beta^{*}} \,,
\end{array}
\right.
\end{equation}
for all $\tau\geq 0$.
\end{proposition}
\begin{proof}
Pick $0<\delta_{\varepsilon}<1$ and $C_{\varepsilon}\geq 1$ from \Cref{StableNonlinFlowMild} and let $0<\delta\leq\delta_{\varepsilon}$ and $C\geq C_{\varepsilon}$. Let $\mathbf{u}_{\beta^{*},T^{*}} \in C([0,\infty),\mathfrak{H}^{k}(\BB^{d}_{R}))$ with $\| \mathbf{u}_{\beta^{*},T^{*}}(\tau,\,.\,) \|_{\mathfrak{H}^{k}(\BB^{d}_{R})} \leq \delta \ee^{(-\omega_{p}+\varepsilon)\tau}$ be the mild solution to
\begin{equation*}
\mathbf{u}_{\beta^{*},T^{*}}(\tau) = \mathbf{S}_{\beta^{*}}(\tau) \mathbf{U}_{\beta^{*},T^{*}}(\mathbf{f}) + \int_{0}^{\tau} \mathbf{S}_{\beta^{*}}(\tau-\tau')\mathbf{N}_{\beta^{*}}(\mathbf{u}_{\beta^{*},T^{*}}(\tau')) \dd \tau' \,, \quad \tau \geq 0 \,,
\end{equation*}
accordingly. We notice with \Cref{NonLinLip} that $\mathbf{N}_{\beta}(\mathbf{u}_{\beta^{*},T^{*}}(\tau)) \in \mathfrak{H}^{k+1}(\BB^{d}_{R}) \subset \mathfrak{H}^{k}(\BB^{d}_{R})$. Using this with \cite[Lemma C.1]{2022arXiv220706952G} and the above fixed-point equation, we conclude $\mathbf{u}_{\beta^{*},T^{*}}(\tau) \in \mathfrak{H}^{k+1}(\BB^{d}_{R})$. For all $\tau\geq 0$, induction and Sobolev embedding implies $\mathbf{u}_{\beta^{*},T^{*}}(\tau)\in C^{\infty}(\overline{\BB^{d}_{R}})^{2}$. Quoting \cite[p. 189, Theorem 1.6]{MR710486}, we have that our mild solution is a classical solution to the abstract Cauchy problem, hence
\begin{align*}
&
\pd_{\tau} \mathbf{u}_{\beta^{*},T^{*}}(\tau) \\&\indent=
\mathbf{S}_{\beta}(\tau) \mathbf{L}_{\beta} \mathbf{U}_{\beta^{*},T^{*}}(\mathbf{f}) + \int_{0}^{\tau} \mathbf{S}_{\beta}(\tau - \tau') \mathbf{L}_{\beta} \mathbf{N}_{\beta} ( \mathbf{u}_{\beta^{*},T^{*}}(\tau') ) \dd\tau' + \mathbf{N}_{\beta} ( \mathbf{u}_{\beta^{*},T^{*}}(\tau) )
\end{align*}
is satisfied for all $\tau \geq 0$. So $\mathbf{u}_{\beta^{*},T^{*}}\in C^{1}\big([0,\infty),\mathfrak{H}^{k}(\BB^{d}_{R})\big)$ and the above equation also implies for each $\tau\geq 0$ that $\pd_{\tau}\mathbf{u}_{\beta^{*},T^{*}}(\tau)\in C^{\infty}(\overline{\BB^{d}_{R}})^{2}$. Inductively, for all $m\in\NN$ we have $\mathbf{u}_{\beta^{*},T^{*}} \in C^{m}\big([0,\infty),\mathfrak{H}^{k}(\BB^{d}_{R})\big)$ and $\pd_{\tau}^{m}\mathbf{u}_{\beta^{*},T^{*}}(\tau)\in C^{\infty}(\overline{\BB^{d}_{R}})^{2}$, for each $\tau \geq 0$. By Sobolev embedding, all derivatives hold pointwise and Schwarz's theorem \cite[p. 235, Theorem 9.41]{MR0385023} yields $(\tau,\xi) \mapsto \mathbf{u}_{\beta^{*},T^{*}}(\tau)(\xi)$ is a jointly smooth map on $[0,\infty)\times \overline{\BB^{d}_{R}}$.\par
\medskip
To see uniqueness for smooth solutions, suppose that $\widetilde{\mathbf{u}}_{\beta^{*},T^{*}} \in C^{\infty}\big(\overline{(0,\infty)\times\BB^{d}_{R}}\big)^{2}$ with $\| \widetilde{\mathbf{u}}_{\beta^{*},T^{*}}(\tau,\,.\,)  \|_{\mathfrak{H}^{k}(\BB^{d}_{R})} \leq \delta \ee^{(-\omega_{p}+\varepsilon)\tau}$ for all $\tau \geq 0$ is another solution to the abstract Cauchy problem with the same initial data. Then $\widetilde{\mathbf{u}}_{\beta^{*},T^{*}}$ is also a mild solution and
\begin{align*}
&
\widetilde{\mathbf{u}}_{\beta^{*},T^{*}}(\tau,\,.\,) - \mathbf{u}_{\beta^{*},T^{*}}(\tau,\,.\,) \\&\indent=
\int_{0}^{\tau} \mathbf{S}_{\beta^{*}}(\tau-\tau') \Big( \mathbf{N}_{\beta^{*}}(\widetilde{\mathbf{u}}_{\beta^{*},T^{*}}(\tau',\,.\,)) - \mathbf{N}_{\beta^{*}}(\mathbf{u}_{\beta^{*},T^{*}}(\tau',\,.\,)) \Big) \dd \tau' \,.
\end{align*}
\Cref{stableEvo,NonLinLip} yield the estimate
\begin{align*}
&
\ee^{-\tau}\| \widetilde{\mathbf{u}}_{\beta^{*},T^{*}}(\tau,\,.\,) - \mathbf{u}_{\beta^{*},T^{*}}(\tau,\,.\,) \|_{\mathfrak{H}^{k}(\BB^{d}_{R})} \\&\indent\lesssim
\int_{0}^{\tau} \ee^{-\tau'} \| \widetilde{\mathbf{u}}_{\beta^{*},T^{*}}(\tau',\,.\,) - \mathbf{u}_{\beta^{*},T^{*}}(\tau',\,.\,) \|_{\mathfrak{H}^{k}(\BB^{d}_{R})} \dd \tau'
\end{align*}
for all $\tau \geq 0$, so Gr\"{o}nwall's lemma implies uniqueness in the class of smooth functions.
\end{proof}
\subsection{Proof of stability of the ODE blowup}
All that remains is to unwind \Cref{StableNonlinFlowClassic}.
\begin{proof}[Proof of \Cref{MainTHM}]
Fix $0<\varepsilon<\omega_{p}$. From \Cref{StableNonlinFlowClassic}, we pick $0<\delta_{0}<1$ and $C_{0}\geq 1$. Then, for any $0<\delta\leq\delta_{0}$, $C\geq C_{0}$ and real-valued $(f,g)\in C^{\infty}(\RR^{d})^{2}$ with $\| (f,g) \|_{H^{k}(\RR^{d})\times H^{k-1}(\RR^{d})} \leq \frac{\delta}{C^{2}}$ there exists a unique $( u_{\beta^{*},T^{*},1},u_{\beta^{*},T^{*},2} ) \in C^{\infty}\big(\overline{(0,\infty)\times\BB^{d}_{R}}\big)^{2}$ that is a jointly smooth classical solution to the abstract Cauchy problem \eqref{CauchyClassical}. The nonlinear relation \eqref{NonlinWaveRel} implies that $u\in C^{\infty}\big(\Omega^{1,d}_{R}(T^{*})\big)$ given by
\begin{equation*}
u(t,x) \coloneqq (T^{*}-t)^{-s_{p}} u_{\beta^{*},T^{*},1} \Big( \log \tfrac{T^{*}}{T^{*}-t}, \tfrac{x}{T^{*}-t} \Big)
\end{equation*}
is related to $u_{\beta^{*},T^{*},2}$ via
\begin{equation*}
\pd_{t} u(t,x) = (T^{*}-t)^{-s_{p}-1} u_{\beta^{*},T^{*},2} \Big( \log \tfrac{T^{*}}{T^{*}-t}, \tfrac{x}{T^{*}-t} \Big)
\end{equation*}
and is precisely the solution to the Cauchy problem \eqref{CauchyProblem2}. Thus,
\begin{equation*}
\psi \coloneqq \psi_{\beta^{*},T^{*}}^{*} + u \in C^{\infty}\big(\Omega^{1,d}_{R}(T^{*})\big)
\end{equation*}
is the unique solution to the Cauchy problem posed in \Cref{MainTHM}. Moreover, by scaling of the homogeneous seminorms, we infer from \Cref{StableNonlinFlowClassic} the bounds
\begin{align*}
(T^{*}-t)^{ - \frac{d}{2} + s_{p} + s }\| \psi(t,\,.\,) - \psi_{\beta^{*},T^{*}}^{*}(t,\,.\,) \|_{\dot{H}^{s}(\BB^{d}_{R(T^{*}-t)})} &=
\| u_{1} \Big( \log \tfrac{T^{*}}{T^{*}-t}, \,.\, \Big) \|_{\dot{H}^{s}(\BB^{d}_{R})} \\&\leq
\delta \left( \tfrac{T^{*}-t}{T^{*}} \right)^{\omega_{p}-\varepsilon}
\end{align*}
for $s=0,1,\ldots,k$, and
\begin{align*}
&
(T^{*}-t)^{ - \frac{d}{2} + s_{p} + s } \| \pd_{t}\psi(t,\,.\,) - \pd_{t}\psi_{\beta^{*},T^{*}}^{*}(t,\,.\,) \|_{\dot{H}^{s-1}(\BB^{d}_{R(T^{*}-t)})} \\&\indent=
\| u_{2} \Big( \log \tfrac{T^{*}}{T^{*}-t}, \,.\, \Big) \|_{\dot{H}^{s-1}(\BB^{d}_{R})} \\&\indent\leq
\delta \left( \tfrac{T^{*}-t}{T^{*}} \right)^{\omega_{p}-\varepsilon}
\end{align*}
for $s=1,\ldots,k$.
\end{proof}
\appendix
\section{Uniform growth bounds for semigroups}
\label{AppendixA}
In order to keep the line of reasoning self-contained, we study uniform growth bounds for semigroups that are generated by a family of compact perturbations of a given generator of a semigroup. Our conclusion in \Cref{GearhartPruessGreinerUniform} below is an adaptation of the Gearhart-Pr\"{u}ss-Greiner Theorem to this situation. Throughout this section, we assume that
\begin{enumerate}[itemsep=1em,topsep=1em,label=(a\arabic*)]
\item\label{a1} $\mathfrak{H}$ is an infinite-dimensional Hilbert space with $\mathfrak{L}(\mathfrak{H})$ the set of all bounded linear operators on $\mathfrak{H}$ equipped with the operator norm,
\item\label{a2} $\mathbf{L}: \dom(\mathbf{L}) \subseteq \mathfrak{H} \rightarrow \mathfrak{H}$ is the generator of a strongly continuous semigroup $\mathbf{S}: \RR_{\geq 0} \rightarrow \mathfrak{L}(\mathfrak{H})$,
\item\label{a3} $\mathbf{L}_{\beta}' \in \mathfrak{L}(\mathfrak{H})$ is a compact linear operator for each parameter $\beta\in B$ in some non-empty compact set $B\subset \RR^{n}$ and the map $B\rightarrow \mathfrak{L}(\mathfrak{H})$, $\beta\mapsto \mathbf{L}_{\beta}'$, is continuous.
\end{enumerate}
To begin with, let us briefly recall some facts about strongly continuous semigroups and perturbations thereof. Every strongly continuous semigroup $\mathbf{S}:\RR_{\geq 0} \rightarrow \mathfrak{L}(\mathfrak{H})$ is automatically exponentially bounded on $\RR_{\geq 0}$, i.e., there exist $\omega\in\RR$ and $M\geq 1$ such that
\begin{equation*}
\| \mathbf{S}(\tau) \|_{\mathfrak{L}(\mathfrak{H})} \leq M \ee^{\omega\tau}
\end{equation*}
for all $\tau \geq 0$, see \cite[p. 39, Proposition 5.5]{MR1721989}. Thus, the quantity
\begin{equation}
\label{GrowthBound}
\omega(\mathbf{S}) \coloneqq \inf\left\{ \omega\in\RR \mid \exists\, M_{\omega} \geq 1 : \| \mathbf{S}(\tau) \|_{\mathfrak{L}(\mathfrak{H})} \leq M_{\omega} \ee^{\omega\tau} \text{ for all } \tau\geq 0 \right\}
\end{equation}
is well-defined and called the \emph{growth bound} of the semigroup. Moreover, the Bounded Perturbation Theorem \cite[p. 158]{MR1721989} asserts that for each $\beta\in B$ the operator
\begin{equation*}
\mathbf{L}_{\beta} \coloneqq \mathbf{L} + \mathbf{L}_{\beta}' \,, \qquad \dom(\mathbf{L}_{\beta}) = \dom(\mathbf{L}) \,,
\end{equation*}
is the generator of a strongly continuous semigroup $\mathbf{S}_{\beta} : \RR_{\geq 0} \rightarrow \mathfrak{L}(\mathfrak{H})$ which satisfies for each fixed $\omega > \omega(\mathbf{S})$ the estimate
\begin{equation}
\label{UniformGrowthEstimate}
\| \mathbf{S}_{\beta}(\tau) \|_{\mathfrak{L}(\mathfrak{H})} \leq M_{\omega} \ee^{\left(\omega + M_{\omega} \|\mathbf{L}_{\beta}'\|_{\mathfrak{L}(\mathfrak{H})} \right) \tau}
\end{equation}
for all $\tau\geq 0$ and all $\beta\in B$. In practice, however, perturbative arguments usually require improved versions of \eqref{UniformGrowthEstimate} that are uniform with respect to the parameter $\beta\in B$. This can be achieved by a separation of the spectra of the generators $\mathbf{L}_{\beta}$. In this respect, let $\omega_{0}>\omega(\mathbf{S})$ be arbitrary but fixed. It follows from \cite[Theorem B.1]{MR4469070} that for each $\beta\in B$, the sets
\begin{equation*}
\Sigma_{\beta} \coloneqq \sigma(\mathbf{L}_{\beta}) \cap \overline{\mathbb{H}_{\omega_{0}}}
\end{equation*}
are finite and each $\lambda\in\Sigma_{\beta}$ is an isolated eigenvalue of $\mathbf{L}_{\beta}$ with finite algebraic multiplicity. Consequently, the spectral projections $\mathbf{P}_{\lambda,\beta} \in \mathfrak{L}(\mathfrak{H})$ given by
\begin{equation*}
\mathbf{P}_{\lambda,\beta} = \frac{1}{2\pi\ii} \int_{\pd\mathbb{D}_{r_{\beta}}(\lambda)} \mathbf{R}_{\mathbf{L}_{\beta}}(z) \dd z \,, \qquad \lambda \in \Sigma_{\beta} \,,
\end{equation*}
with $r_{\beta} > 0$ chosen such that $\overline{\mathbb{D}_{r_{\beta}}(\lambda)} \subset \varrho(\mathbf{L}_{\beta})$ are disjoint for different $\lambda$, have finite rank. We remark that since $\omega_{0}$ is fixed and does not vary, we can omit it in our notation. With this, we define the total projection.
\begin{definition}
\label{TotalProjection}
Fix $\omega_{0}>\omega(\mathbf{S})$. For $\beta\in B$ we define
\begin{equation*}
\mathbf{P}_{\beta} \coloneqq \sum_{\lambda\in \Sigma_{\beta}} \mathbf{P}_{\lambda,\beta} \in \mathfrak{L}(\mathfrak{H}) \,.
\end{equation*}
\end{definition}
Since $\mathbf{S}_{\beta}(\tau)$ commutes with its generator $\mathbf{L}_{\beta}$, the linear subspace $\mathfrak{H}_{\beta} \coloneqq \ran(\mathbf{I} - \mathbf{P}_{\beta}) \subset \mathfrak{H}$ is $\mathbf{S}_{\beta}(\tau)$-invariant. Hence the subspace semigroups
\begin{equation}
\label{SubspaceSemigroup}
\RR_{\geq 0} \rightarrow \mathfrak{L}(\mathfrak{H}_{\beta}) \,, \quad \tau \mapsto \mathbf{S}_{\beta}(\tau)\restriction_{\mathfrak{H}_{\beta}} = \mathbf{S}_{\beta}(\tau)( \mathbf{I} - \mathbf{P}_{\beta} )\restriction_{\mathfrak{H}_{\beta}} \,,
\end{equation}
with generator $\mathbf{L}_{\beta}\restriction_{\mathfrak{H}_\beta}$, are well-defined. We recall from \cite[p. 178]{MR1335452} the property $\sigma(\mathbf{L}_{\beta}\restriction_{\mathfrak{H}_\beta}) = \sigma(\mathbf{L}_{\beta}) \setminus \Sigma_{\beta}$ and that the map
\begin{equation}
\label{SubspaceResolvent}
\varrho(\mathbf{L}_{\beta}\restriction_{\mathfrak{H}_\beta}) \rightarrow \mathfrak{L}(\mathfrak{H}) \,, \quad z \mapsto
\mathbf{R}_{\mathbf{L}_{\beta}\restriction_{\mathfrak{H}_\beta}}(z) =
\mathbf{R}_{\mathbf{L}_{\beta}}(z)\restriction_{\mathfrak{H}_\beta} =
\mathbf{R}_{\mathbf{L}_{\beta}}(z) ( \mathbf{I} - \mathbf{P}_{\beta} ) \restriction_{\mathfrak{H}_\beta} \,,
\end{equation}
is analytic. In particular, $\overline{\mathbb{H}_{\omega_{0}}} \subseteq \varrho(\mathbf{L}_{\beta}\restriction_{\mathfrak{H}_\beta})$. Our aim is to prove a criterion for an exponential growth bound in $\tau\geq 0$ uniformly with respect to the parameter $\beta\in B$ for the subspace semigroups \eqref{SubspaceSemigroup} in terms of a uniform bound for the subspace resolvents \eqref{SubspaceResolvent} in a right half-plane. We quantify this with a uniform analogue of \eqref{GrowthBound}.
\begin{definition}
\label{UnifExpStab}
Fix $\omega_{0}>\omega(\mathbf{S})$. For $\beta\in B$ let
\begin{equation*}
\mathbf{S}\restriction_{\beta}: \RR_{\geq 0} \rightarrow \mathfrak{L}(\mathfrak{H}_{\beta}) \,, \qquad
\mathbf{S}\restriction_{\beta}(\tau) \coloneqq \ee^{-\omega_{0}\tau} \mathbf{S}_{\beta}(\tau) \restriction_{\mathfrak{H}_{\beta}} \,,
\end{equation*}
be a family of rescaled subspace semigroups with generator $\mathbf{L}\restriction_{\beta} \coloneqq \mathbf{L}_{\beta}\restriction_{\mathfrak{H}_\beta} - \omega_{0} \mathbf{I}$. We define the \emph{uniform growth bound} by
\begin{align*}
\omega(\mathbf{S}\restriction) \coloneqq \inf\Big\{ \omega\in\RR \mid \exists\, M_{\omega} \geq 1 : \| \mathbf{S}\restriction_{\beta}(\tau) \|_{\mathfrak{L}(\mathfrak{H}_{\beta})} & \leq M_{\omega} \ee^{\omega\tau} \\&
\text{ for all } \beta\in B \text{ and all } \tau\geq 0 \Big\} \,.
\end{align*}
\end{definition}
Note with \eqref{UniformGrowthEstimate} and assumption \ref{a3} that $\omega(\mathbf{S}\restriction)<\infty$. The uniform growth bound is characterized by the following formula.
\begin{lemma}
\label{UniformGrowthBoundFormula}
We have
\begin{equation*}
\omega(\mathbf{S}\restriction) =
\inf_{\tau>0} \tau^{-1} \log \sup_{\beta\in B} \| \mathbf{S}\restriction_{\beta}(\tau) \|_{\mathfrak{L}(\mathfrak{H}_{\beta})} =
\lim_{\tau\to\infty} \tau^{-1} \log \sup_{\beta\in B} \| \mathbf{S}\restriction_{\beta}(\tau) \|_{\mathfrak{L}(\mathfrak{H}_{\beta})} \,.
\end{equation*}
\end{lemma}
\begin{proof}
We adapt the proof of \cite[p. 251, Proposition 2.2]{MR1721989} to our situation. First, note that
\begin{equation}
\label{UniformExponentialGrowthBound}
\omega(\mathbf{S}\restriction) =
\inf\left\{ \omega\in\RR \mid \exists\, M_{\omega} \geq 1 : \sup_{\beta\in B} \| \mathbf{S}\restriction_{\beta}(\tau) \|_{\mathfrak{L}(\mathfrak{H}_{\beta})} \leq M_{\omega} \ee^{\omega\tau} \text{ for all } \tau\geq 0 \right\} \,.
\end{equation}
The map $\tau\mapsto\log\sup_{\beta\in B} \| \mathbf{S}\restriction_{\beta}(\tau) \|_{\mathfrak{L}(\mathfrak{H}_{\beta})}$ is well-defined on $[0,\infty)$, bounded on compact intervals and sub-additive. Thus \cite[p. 251, Lemma 2.3]{MR1721989} asserts that
\begin{equation}
\label{omegaInfty}
\omega_{\infty} \coloneqq \inf_{\tau>0} \tau^{-1} \log \sup_{\beta\in B} \| \mathbf{S}\restriction_{\beta}(\tau) \|_{\mathfrak{L}(\mathfrak{H}_{\beta})} = \lim_{\tau\to\infty} \tau^{-1} \log \sup_{\beta\in B} \| \mathbf{S}\restriction_{\beta}(\tau) \|_{\mathfrak{L}(\mathfrak{H}_{\beta})}
\end{equation}
exists. From this we get for all $\tau > 0$
\begin{equation*}
\omega_{\infty} \leq \tau^{-1} \log \sup_{\beta\in B} \| \mathbf{S}\restriction_{\beta}(\tau) \|_{\mathfrak{L}(\mathfrak{H}_{\beta})} \,,
\qquad\text{i.e.,}\qquad
\sup_{\beta\in B} \| \mathbf{S}\restriction_{\beta}(\tau) \|_{\mathfrak{L}(\mathfrak{H}_{\beta})} \geq \ee^{\omega_{\infty} \tau} \,.
\end{equation*}
It follows from \Cref{UniformExponentialGrowthBound} that $\omega_{\infty} \leq \omega(\mathbf{S}\restriction)$.\par
\medskip
Conversely, let $\omega>\omega_{\infty}$. From \Cref{omegaInfty} it follows that there is a $\tau_{0} > 0$ such that for all $\tau\geq\tau_{0}$
\begin{equation*}
\tau^{-1} \log \sup_{\beta\in B} \| \mathbf{S}\restriction_{\beta}(\tau) \|_{\mathfrak{L}(\mathfrak{H}_{\beta})} \leq \omega \,,
\qquad\text{i.e.,}\qquad
\sup_{\beta\in B} \| \mathbf{S}\restriction_{\beta}(\tau) \|_{\mathfrak{L}(\mathfrak{H}_{\beta})} \leq \ee^{\omega\tau} \,.
\end{equation*}
By \Cref{SGbeta}, $(\tau,\beta)\mapsto\| \mathbf{S}\restriction_{\beta}(\tau) \|_{\mathfrak{L}(\mathfrak{H}_{\beta})}$ is bounded on $[0,\tau_{0}]\times B$. From this we conclude that there is an $M\geq 1$ such that
\begin{equation*}
\sup_{\beta\in B} \| \mathbf{S}\restriction_{\beta}(\tau) \|_{\mathfrak{L}(\mathfrak{H}_{\beta})} \leq M \ee^{\omega\tau}
\end{equation*}
for all $\tau\geq 0$. So $\omega\geq \omega(\mathbf{S}\restriction)$ and since $\omega>\omega_{\infty}$ was arbitrary, also $\omega_{\infty} \geq  \omega(\mathbf{S}\restriction)$.
\end{proof}
Adapting the steps from the proof of the Gearhart-Pr\"{u}ss-Greiner Theorem \cite[p. 302]{MR1721989} to the present situation leads to a practical criterion for the existence of an improved exponential growth estimate for the subspace semigroup \eqref{SubspaceSemigroup} in $\tau \geq 0$ that is uniform with respect to the parameter $\beta\in B$.
\begin{theorem}
\label{GearhartPruessGreinerUniform}
Assume \ref{a1}-\ref{a3}. For fixed $\omega_{0} > \omega(\mathbf{S})$ and for $\beta\in B$ let $\mathbf{P}_{\beta}\in\mathfrak{L}(\mathfrak{H})$ be the spectral projection as in \Cref{TotalProjection}. Then, there exist $\omega<\omega_{0}$ and a constant $M_{\omega}\geq 1$ such that
\begin{equation*}
\| \mathbf{S}_{\beta}(\tau) ( \mathbf{I} - \mathbf{P}_{\beta} ) \|_{\mathfrak{L}(\mathfrak{H})} \leq M_{\omega} \ee^{\omega \tau} \| \mathbf{I} - \mathbf{P}_{\beta} \|_{\mathfrak{L}(\mathfrak{H})}
\end{equation*}
for all $\tau\geq 0$ and all $\beta\in B$, if and only if there exists a constant $C_{\omega_{0}} > 0$ such that the analytic maps $\varrho(\mathbf{L}_{\beta}) \setminus\Sigma_{\beta} \rightarrow \mathfrak{L}(\mathfrak{H})$, $z \mapsto \mathbf{R}_{\mathbf{L}_{\beta}}(z) ( \mathbf{I} - \mathbf{P}_{\beta} )$, satisfy
\begin{equation*}
\| \mathbf{R}_{\mathbf{L}_{\beta}}(z) ( \mathbf{I} - \mathbf{P}_{\beta} ) \|_{\mathfrak{L}(\mathfrak{H})} \leq C_{\omega_{0}}
\end{equation*}
for all $z\in\overline{\mathbb{H}_{\omega_{0}}}$ and all $\beta\in B$.
\end{theorem}
\begin{proof}
\begin{enumerate}[wide,itemsep=1em,topsep=1em]
\item[``$\Rightarrow$'':] Put $\varepsilon_{0} \coloneqq \omega_{0} - \omega>0$. In terms of the rescaled subspace semigroups from \Cref{UnifExpStab}, we get from the assumption that there is a constant  $M_{\omega}\geq 1$ such that $\| \mathbf{S}\restriction_{\beta}(\tau) \|_{\mathfrak{L}(\mathfrak{H}_{\beta})} \leq M_{\omega} \ee^{-\varepsilon_{0}\tau}$ for all $\beta\in B$ and all $\tau\geq 0$. Then \cite[p. 55, Theorem 1.10]{MR1721989} yields the bound
\begin{equation*}
\| \mathbf{R}_{\mathbf{L} \restriction_{\beta}} (z) \|_{\mathfrak{L}(\mathfrak{H}_{\beta})} \leq \frac{M_{\omega}}{\Re(z) + \varepsilon_{0}}
\end{equation*}
for all $z\in \mathbb{H}_{-\varepsilon_{0}}$ and all $\beta\in B$. This implies with $\mathbf{R}_{\mathbf{L}_{\beta} \restriction_{\mathfrak{H}_{\beta}}} (z) = \mathbf{R}_{\mathbf{L} \restriction_{\beta}} (-\omega_{0} + z)$ the uniform resolvent bound
\begin{equation*}
\| \mathbf{R}_{\mathbf{L}_{\beta} \restriction_{\mathfrak{H}_{\beta}}} (z) \|_{\mathfrak{L}(\mathfrak{H}_{\beta})} \leq \frac{M_{\omega}}{\Re(z)-\omega_{0}+\varepsilon_{0}} \leq \frac{M_{\omega}}{\varepsilon_{0}} \eqqcolon C_{\omega_{0}}
\end{equation*}
for all $z\in\overline{\mathbb{H}_{\omega_{0}}}$ and $\beta\in B$.
\item[``$\Leftarrow$'':] Fix $\omega' > |\omega(\mathbf{S}\restriction)| + \frac{1}{2}$ and define a rescaled semigroup by
\begin{equation*}
\mathbf{T}_{\beta}(\tau) \coloneqq
\ee^{-\omega'\tau} \mathbf{S}\restriction_{\beta}(\tau) \,.
\end{equation*}
As $\omega'>\omega(\mathbf{S}\restriction) + \frac{1}{2}$, by \Cref{UniformExponentialGrowthBound} there is an $M\geq 1$ such that
\begin{equation*}
\sup_{\beta\in B} \| \mathbf{T}_{\beta}(\tau) \|_{\mathfrak{L}(\mathfrak{H}_{\beta})} \leq M \ee^{-\frac{\tau}{2}}
\end{equation*}
for all $\tau \geq 0$. Again by \cite[p. 55, Theorem 1.10]{MR1721989}, this rescaled semigroup is related to the resolvent via the formula
\begin{equation*}
\mathbf{R}_{\mathbf{L}\restriction_{\beta}} (\omega'+\ii\lambda) \mathbf{f} = \mathbf{R}_{\mathbf{L}\restriction_{\beta}-\omega'\mathbf{I}} (\ii\lambda) \mathbf{f} = \int_{0}^{\infty} \ee^{-\ii\lambda\tau} \mathbf{T}_{\beta}(\tau) \mathbf{f} \dd\tau
\end{equation*}
for all $\mathbf{f}\in \mathfrak{H}_{\beta}$. Now, Plancherel's theorem yields
\begin{equation*}
\int_{\RR} \| \mathbf{R}_{\mathbf{L}\restriction_{\beta}} (\omega'+\ii\lambda) \mathbf{f} \|_{\mathfrak{H}}^{2} \dd \lambda = 2\pi \int_{0}^{\infty} \| \mathbf{T}_{\beta}(\tau) \mathbf{f} \|_{\mathfrak{H}}^{2} \dd\tau \leq 2\pi M^{2} \| \mathbf{f} \|_{\mathfrak{H}}^{2}
\end{equation*}
for all $\mathbf{f}\in \mathfrak{H}_{\beta}$. Using the first resolvent identity
\begin{equation*}
\mathbf{R}_{\mathbf{L}\restriction_{\beta}} (\ii\lambda) = \mathbf{R}_{\mathbf{L}\restriction_{\beta}} (\omega'+\ii\lambda) + \omega' \mathbf{R}_{\mathbf{L}\restriction_{\beta}} (\ii\lambda) \mathbf{R}_{\mathbf{L}\restriction_{\beta}} (\omega' + \ii\lambda)
\end{equation*}
and the assumption that
\begin{equation*}
\| \mathbf{R}_{\mathbf{L}\restriction_{\beta}} (\ii\lambda) \|_{\mathfrak{L}(\mathfrak{H}_{\beta})} = \| \mathbf{R}_{\mathbf{L}_{\beta}\restriction_{\mathfrak{H}_{\beta}}} (\omega_{0} + \ii\lambda) \|_{\mathfrak{L}(\mathfrak{H}_{\beta})} \leq C_{\omega_{0}}
\end{equation*}
for all $\lambda\in\RR$ and all $\beta\in B$, we get
\begin{equation*}
\| \mathbf{R}_{\mathbf{L}\restriction_{\beta}} (\ii\lambda) \mathbf{f} \|_{\mathfrak{H}} \leq (1+\omega' C_{\omega_{0}}) \| \mathbf{R}_{\mathbf{L}\restriction_{\beta}} (\omega' + \ii\lambda) \mathbf{f} \|_{\mathfrak{H}}
\end{equation*}
for all $\mathbf{f}\in \mathfrak{H}_{\beta}$ and all $\lambda\in\RR$, $\beta\in B$. These estimates imply
\begin{align}
\begin{split}
\label{ResL2}
\int_{\RR} \| \mathbf{R}_{\mathbf{L}\restriction_{\beta}} (\ii\lambda) \mathbf{f} \|_{\mathfrak{H}}^{2} \dd \lambda &\leq (1+\omega' C_{\omega_{0}})^{2} \int_{\RR} \| \mathbf{R}_{\mathbf{L}\restriction_{\beta}} (\omega'+\ii\lambda) \mathbf{f} \|_{\mathfrak{H}}^{2} \dd \lambda \\&\leq (1+\omega' C_{\omega_{0}})^{2} 2\pi M^{2} \| \mathbf{f} \|_{\mathfrak{H}}^{2}
\end{split}
\end{align}
for all $\mathbf{f}\in \mathfrak{H}_{\beta}$ and all $\beta\in B$. The same applies for the adjoint semigroup $\mathbf{T}_{\beta}(\tau)^{*}$ with generator $\mathbf{L}\restriction_{\beta}^{*}$, i.e.,
\begin{equation}
\label{AdjResL2}
\int_{\RR} \| \mathbf{R}_{\mathbf{L}\restriction_{\beta}^{*}} (\ii\lambda) \mathbf{g} \|_{\mathfrak{H}}^{2} \dd \lambda \leq (1+\omega' C_{\omega_{0}})^{2} 2\pi M^{2} \| \mathbf{g} \|_{\mathfrak{H}}^{2}
\end{equation}
for all $\mathbf{g}\in \mathfrak{H}_{\beta}$ and all $\beta\in B$. Now, the inversion formula \cite[p. 234, Corollary 5.16]{MR1721989} applies and asserts
\begin{equation*}
\mathbf{S}\restriction_{\beta}(\tau) \mathbf{f} =
\frac{1}{\tau} \frac{1}{2\pi\ii} \lim_{n\to\infty} \int_{\omega'-\ii n}^{\omega'+\ii n} \ee^{z\tau} \mathbf{R}_{\mathbf{L}\restriction_{\beta}} (z)^{2} \mathbf{f} \dd z
\end{equation*}
for all $\mathbf{f}\in\dom(\mathbf{L}\restriction_{\beta}^{2})$ and all $\tau > 0$. This implies with Cauchy's integral theorem
\begin{align*}
\Big( \mathbf{S}\restriction_{\beta}(\tau) \mathbf{f} \,\Big|\, \mathbf{g} \Big)_{\mathfrak{H}} &=
\frac{1}{\tau} \frac{1}{2\pi} \lim_{n\to\infty}
\int_{-n}^{n} \ee^{(\omega'+\ii \lambda)\tau}
\Big(
\mathbf{R}_{\mathbf{L}\restriction_{\beta}} (\omega'+\ii\lambda) \mathbf{f}
\,\Big|\,
\mathbf{R}_{\mathbf{L}\restriction_{\beta}^{*}} (\omega'-\ii\lambda) \mathbf{g}
\Big)_{\mathfrak{H}}
\dd\lambda \\&=
\frac{1}{\tau} \frac{1}{2\pi}
\int_{\RR} \ee^{\ii \lambda\tau}
\Big(
\mathbf{R}_{\mathbf{L}\restriction_{\beta}} (\ii\lambda) \mathbf{f}
\,\Big|\,
\mathbf{R}_{\mathbf{L}\restriction_{\beta}^{*}} (-\ii\lambda) \mathbf{g}
\Big)_{\mathfrak{H}}
\dd\lambda
\end{align*}
for all $\mathbf{f}\in\dom(\mathbf{L}\restriction_{\beta}^{2})$ and all $\mathbf{g}\in\mathfrak{H}$. Applying the Cauchy-Schwarz inequality, using the estimates \eqref{ResL2}, \eqref{AdjResL2} and density of $\dom(\mathbf{L}\restriction_{\beta}^{2}) \subset \mathfrak{H}$ yields
\begin{equation*}
\Big| \Big( \mathbf{S}\restriction_{\beta}(\tau) \mathbf{f} \,\Big|\, \mathbf{g} \Big)_{\mathfrak{H}} \Big| \leq
\frac{1}{\tau} \frac{1}{2\pi} (1+\omega' C_{\omega_{0}})^{2} 2\pi M^{2} \|\mathbf{f}\|_{\mathfrak{H}} \|\mathbf{g}\|_{\mathfrak{H}}
\end{equation*}
for all $\mathbf{f},\mathbf{g}\in \mathfrak{H}_{\beta}$ and all $\tau >0$, $\beta\in B$. Consequently,
\begin{equation*}
\|\mathbf{S}\restriction_{\beta}(\tau) \|_{\mathfrak{L}(\mathfrak{H}_{\beta})} \leq \frac{1}{\tau} (1+\omega' C_{\omega_{0}})^{2} M^{2}
\end{equation*}
for all $\tau >0$, $\beta\in B$. It follows that for large enough $\tau>0$,
\begin{equation*}
\sup_{\beta\in B} \|\mathbf{S}\restriction_{\beta}(\tau) \|_{\mathfrak{L}(\mathfrak{H}_{\beta})} < 1 \,,
\end{equation*}
so \Cref{UniformGrowthBoundFormula} yields $\omega(\mathbf{S}\restriction)<0$.
\qedhere
\end{enumerate}
\end{proof}
\begin{acknowledgment}
The author thanks Roland Donninger and Irfan Glogi\'{c} for useful comments that improved a first draft of this paper.
\end{acknowledgment}

\end{document}